\documentclass[11pt]{amsart}
\usepackage{amssymb,amsmath, amsthm, amsfonts}
\usepackage{graphicx}
\usepackage{tikz-cd}
\usepackage{listings}
\usepackage{lstautogobble}
\usepackage{enumerate}
\usepackage[shortlabels]{enumitem}
\usepackage{thmtools}
\usepackage{thm-restate}
\usepackage{amsthm}
\usepackage{verbatim}
\usepackage[a4paper, portrait, margin=0.75in]{geometry}
\usepackage{mathtools}
\usepackage{physics}
\usepackage{hyperref}
\usepackage[capitalise]{cleveref}
\usepackage[bottom]{footmisc}
\crefformat{equation}{(#2#1#3)}

\usepackage[utf8]{inputenc}

\usepackage{xcolor} 
\usepackage{ esint } 
\usepackage{comment}
\usepackage{hyperref}
\usepackage{mathrsfs}
\usepackage{bm}

\theoremstyle{plain}
\newtheorem{theorem}{Theorem}[section]
\newtheorem{lemma}[theorem]{Lemma}
\newtheorem{proposition}[theorem]{Proposition}

\theoremstyle{definition}
\newtheorem{definition}[theorem]{Definition}
\newtheorem{remark}[theorem]{Remark}

\newcommand{\G}{\mathfrak{G}}
\newcommand{\g}{\mathfrak{g}}
\newcommand{\ep}{\epsilon}

\usepackage{float}
\restylefloat{table}

\Crefname{theorem}{Theorem}{Theorems}
\Crefname{proposition}{Proposition}{Propositions}
\Crefname{table}{Table}{Tables}
\Crefname{figure}{Figure}{Figures}
\Crefname{tikzpicture}{Diagram}{Diagrams}

\numberwithin{equation}{section} 

\DeclarePairedDelimiter\ipp{\langle}{\rangle}

\DeclarePairedDelimiter{\paren}{\lparen}{\rparen}

\newcommand{\p}{{\partial}}

\newcommand{\R}{{\mathbb{R}}}

\newcommand{\N}{{\mathbb{N}}}

\newcommand{\Ss}{{\mathbb{S}}}

\newcommand{\Sc}{{\mathcal{S}}}
\newcommand{\A}{{\mathcal{A}}}

\newcommand{\Ec}{\mathcal{E}}

\newcommand{\tl}{\tilde}

\newcommand{\ul}{\underline}

\newcommand{\uz}{\underline{z}}

\newcommand{\uw}{\underline{w}}

\newcommand{\ab}{\mathbf{a}}
\newcommand{\bp}{\mathbf{p}}

\newcommand{\J}{\mathbb{J}}
\newcommand{\I}{\mathbb{I}}

\newcommand{\Ga}{\Gamma}
\newcommand{\ga}{\gamma}
\newcommand{\om}{\omega}
\newcommand{\al}{\alpha}
\newcommand{\be}{\beta}
\renewcommand{\d}{\delta}

\newcommand{\Cc}{\mathcal{C}}

\newcommand{\indic}{\mathbf{1}}

\newcommand{\Fc}{\mathcal{F}}
\newcommand{\Gc}{\mathcal{G}}
\newcommand{\Hc}{\mathcal{H}}
\newcommand{\Hcl}{\mathcal{H}_{Lio}}
\newcommand{\Hvlh}{\mathcal{H}_{VlH}}
\newcommand{\Hvl}{\mathcal{H}_{Vl}}
\newcommand{\Hbb}{\mathcal{H}_{BBGKY}}

\newcommand{\nn}{\nonumber}

\let\div\relax
\DeclareMathOperator{\div}{\mathrm{div}}

\newcommand{\W}{{\mathbf{W}}}

\DeclareMathOperator{\Sym}{Sym}

\newcommand{\Rd}{(\R^{2d})}

\newcommand{\ds}{\mathsf{d}}

\title[A rigorous derivation of the Hamiltonian structure for the Vlasov Equation]{A rigorous derivation of the Hamiltonian structure for the Vlasov Equation}

\author[J.K. Miller]{Joseph K. Miller}
\email{jkmiller@utexas.edu}
\thanks{J.M. was partially supported by the NSF under grants No. DMS-1840314, DMS-2009549, DMS-2052789 and by the University of Texas at Austin through a Provost’s Graduate Excellence Fellowship.}
\author[A.R. Nahmod]{Andrea R. Nahmod}
\email{nahmod@math.umass.edu}
\thanks{A.N. was partially supported by the NSF under grants No. DMS-2101381, DMS-2052740 and by the Simons Foundation Collaboration Grant on Wave Turbulence (Nahmod’s Award ID 651469).}
\author[N. Pavlovi\'{c}]{Nata\v{s}a Pavlovi\'c}
\email{natasa@math.utexas.edu}
\thanks{N.P. was partially supported by the NSF under grants No. DMS-1840314, DMS-2009549, DMS-2052789.}
\author[M. Rosenzweig]{Matthew Rosenzweig}
\email{mrosenzw@mit.edu}
\thanks{M.R. was partially supported by the NSF under grant No. DMS-2052651 and by the Simons Foundation Collaboration Grant on Wave Turbulence (Staffilani's Award ID 6941649).}
\author[G. Staffilani]{Gigliola Staffilani}
\email{gigliola@math.mit.edu}
\thanks{G.S. was partially supported by the NSF under grant No. DMS-2052651 and by the Simons Foundation Collaboration Grant on Wave Turbulence (Staffilani's Award ID 6941649).}

\begin{document}
\begin{abstract}
We consider the Vlasov equation in any spatial dimension, which has long been known \cite{IT1976, Morrison1980, Gibbons1981, MW1982}  to be an infinite-dimensional Hamiltonian system whose bracket structure is of \emph{Lie-Poisson type}. In parallel, it is classical that the Vlasov equation is a \emph{mean-field limit} for a pairwise interacting Newtonian system. Motivated by this knowledge, we provide a rigorous derivation of the Hamiltonian structure of the Vlasov equation, both the Hamiltonian functional and Poisson bracket, directly from the many-body problem. One may view this work as a classical counterpart to \cite{MNPRS2020}, which provided a rigorous derivation of the Hamiltonian structure of the cubic nonlinear Schr\"odinger equation from the many-body problem for interacting bosons in a certain infinite particle number limit, the first result of its kind. In particular, our work settles a question of Marsden, Morrison, and Weinstein \cite{MMW1984} on providing a ``statistical basis'' for the bracket structure of the Vlasov equation.
\end{abstract}

\maketitle

\section{Introduction}\label{sec:intro}

\subsection{Motivation}\label{ssec:intromot}

Several decades ago, Marsden, Morrison, Weinstein, and others initiated a program on understanding the geometric structure of common partial differential equations (PDEs) in mathematical physics. A key question of this program is the passage or ``contraction,'' to use the language of \cite{MW1982}, of one Hamiltonian system to another through scaling limits. In the present article, we consider this question in the context of the \emph{Vlasov equation}, which is the nonlinear PDE 
\begin{equation}\label{eq:Vl}
\begin{cases}
\p_t\ga + v\cdot\nabla_x \ga  -2(\nabla W\ast\rho)\cdot\nabla_v \ga =0\\
\rho = \int_{\R^d}d\ga(\cdot,v) \\
\ga|_{t=0} = \ga^0,
\end{cases}
\qquad (t,x,v)\in \R\times\R^{2d}.
\end{equation}
The unknown $\ga$ models the distribution of the particles in the position-velocity phase space $(x,v)\in\R^{2d}$, with $d\geq 1$. Assuming $\ga$ is normalized to have unit integral, one can interpret $\ga^t(x,v)dxdv$ as approximately the probability at time $t$ of finding a particle in a phase space box of area $dxdv$  around the position $x$ and velocity $v$. The function $\rho$ is the spatial density of the particles, obtained by integrating out velocity. We use the same notation for a density and its associated measure. The function $W:\R^d\rightarrow\R$ is a potential governing the interactions between the particles, which for simplicity we will always assume is even, though this assumption is not essential. In physics, one typically chooses $W$ to be a multiple of the Coulomb/Newton potential in $\R^d$. The sign of $W$ determines whether the potential is repulsive (+), which is relevant for electrostatic interactions, or attractive (-), which is relevant for gravitational interactions. For such a $W$, equation \eqref{eq:Vl} is commonly referred to in the literature as \emph{Vlasov-Poisson}. This specific form of the equation was first proposed by Jeans \cite{Jeans1915} as model for galaxies; its use in plasma physics originates in work of its eponym Vlasov \cite{Vlasov1938}.

While not the primary subject of this article, we mention that the Vlasov equation as a PDE has been actively studied over the years, with basic questions of well-posedness now well understood. When $W$ is regular (e.g., $\nabla W$ is Lipschitz), well-posedness of measure-valued weak solutions is classical \cite{BH1977,Dobrushin1979}. In the case when $W$ is not regular, for instance as in Vlasov-Poisson, well-posedness is not known in the class of measures, but global well-posedness is known for solution classes at higher regularities \cite{Iordanskii1961, Arsenev1975, HH1984, Batt1977, UO1978, Wollman1980, BD1985, Pfaffelmoser1992, Schaeffer1991, Horst1993, LP1991, Pallard2012}. Subsequent work has investigated sufficient conditions for the uniqueness of solutions \cite{Robert1997, Loeper2006, Miot2016, Iacobelli2022} and well-posedness when $W$ is even more singular at the origin than the Coulomb potential (e.g., general Riesz potentials) \cite{CJ2022vr}. An active topic of current research concerns the long-time dynamics of Vlasov equations; for example, see \cite{MV2011, BMM2016, FR2016, CK2016, GNR2020, HkNR2021, PW2021, FOPW2021, IPW2022, GNR2022} and references therein.

Iwinski and Turski \cite{IT1976} and Morrison \cite{Morrison1980} independently made the formal observation that there is a Poisson bracket structure with respect to which the Vlasov equation is Hamiltonian. We remind the reader that the Hamiltonian formulation of an equation consists of a Hamiltonian functional and an underlying manifold equipped with a Poisson bracket, which serves as the phase space. Marsden and Weinstein \cite{MW1982} and Gibbons \cite{Gibbons1981} later observed that this bracket is of \emph{Lie-Poisson} type, which we briefly outline ignoring any functional-analytic difficulties. There is a Lie algebra $(\g,\comm{\cdot}{\cdot}_\g)$, elements of which are functions $f(x,v)$ corresponding to observables. On the dual $\g^*$, elements of which correspond to states (think measures, more generally distributions $\ga$ on $\R^{2d}$), there is a Poisson bracket $\pb{\cdot}{\cdot}_{\g^*}$ canonically obtained from the Lie algebra $(\g, \comm{\cdot}{\cdot}_{\g})$ through
\begin{equation}
\ipp{\comm{\ds\Fc[\ga]}{\ds\Gc[\ga]}_{\g},\ga}_{\g-\g^*}, \qquad \Fc,\Gc\in \Cc^\infty(\g^*), \ \ga\in\g^*.
\end{equation}
Here, $\Fc,\Gc$ are smooth (using the G\^ateaux differential calculus) real-valued functions on $\g^*$; and using the isomorphism $(\g^*)^* \cong \g$ (assuming the space $\g$ is chosen appropriately), the G\^ateaux derivatives $\ds\Fc[\ga],\ds\Gc[\ga]$, which are linear functionals, may be identified as elements of the Lie algebra $\g$. The notation $\ipp{\cdot,\cdot}_{\g-\g^*}$ denotes the duality pairing between $\g$ and $\g^*$. The \emph{Vlasov Hamiltonian functional} is 
\begin{equation}
\Hvl(\gamma) = \frac{1}{2}\int_{(\R^d)^2}d\ga(x,v)|v|^2+\int_{(\R^d)^2}d\rho^{\otimes 2}(x,y)W(x-y).
\end{equation}
For any sufficiently nice functional $\Fc\in\Cc^\infty(\g^*)$, there exists a unique Hamiltonian vector field $X_{\Fc}$ on $\g^*$ characterized by the property that
\begin{equation}
\forall \Gc\in \Cc^\infty(\g^*), \qquad X_{\Fc}(\Gc) = \pb{\Gc}{\Fc}_{\g^*},
\end{equation}
where the vector field $X_{\Fc}$ is understood as a derivation in the left-hand side. By direct computation of $X_{\Hvl}$, one sees that the Vlasov equation is equivalent to the infinite-dimensional ODE
\begin{equation}
\dot{\ga} = X_{\Hvl}(\ga).
\end{equation}

\medskip
The physical significance of the Vlasov equation is as a macroscopic limit of a system of indistinguishable Newtonian particles with pairwise interactions. The starting point for the description of this limit is the system of $N$ ordinary differential equations
\begin{equation}\label{eq:New}
\begin{cases}
\dot{x}_i^t = v_i^t\\
\dot{v}_i^t = \displaystyle-\frac{2}{N}\sum_{\substack{1\leq j\leq N :j\neq i}}\nabla W(x_i-x_j),
\end{cases}
\qquad \forall i\in\{1,\ldots,N\},
\end{equation}
where $i$ is the particle index. We adopt the convention that $\nabla W(0):= 0$, which allows for singular $W$ and is consistent with the even assumption if $W$ is regular. This allows us to add back into the summation $j=i$.

As is well-known, the system \eqref{eq:New} can be rewritten in the form of Hamilton's equations. Introducing the total energy of the system
\begin{equation}\label{eq:HNdef}
H_N(\uz_N) \coloneqq \frac{1}{2}\sum_{i=1}^N|v_i|^2 + \frac{1}{N}\sum_{1\leq i\neq j\leq N} W(x_i-x_j),
\end{equation}
and writing $\uz_N = (z_1,\ldots,z_N)$ with $z_i = (x_i,v_i)$, \eqref{eq:New} is equivalent to
\begin{equation}\label{eq:VlasJ}
\dot{z}_N^t = \J_N\nabla_{\uz_N}H_N(\uz_N)
\end{equation}
where $\nabla_{\uz_N} = (\nabla_{z_1},\ldots,\nabla_{z_N})$ with $\nabla_{z_i} = (\nabla_{x_i},\nabla_{v_i})$ and $\J_N$ is the block-diagonal matrix whose diagonal entries are the rotation matrix $\J(x,v) = (-v,x)$. 

Given a solution $\uz_N^t$ of \eqref{eq:New}, one can associate to it a probability measure $\mu_{N}^t \coloneqq \frac{1}{N}\sum_{i=1}^N \d_{z_i^t}(z)$ on $\R^{2d}$ called the \emph{empirical measure}. By integrating $\mu_{N}^t$ against a test function, it is a straightforward calculation to show that $\mu_{N}^t$ is a weak solution to the Vlasov equation if and only if $\uz_N^t$ is a solution of \eqref{eq:New}. Accordingly, if the initial empirical measures $\mu_{N}^0$ converge weakly as $N\rightarrow\infty$ to an expected or \emph{mean-field} measure $\mu^0$ with regular density, then one expects---or hopes---that $\mu_N^t$ converges weakly to a solution $\mu^t$ of the Vlasov equation with initial datum $\mu^0$ for all time, a law of large numbers type result. The Vlasov equation \eqref{eq:Vl} is then referred to as the \emph{mean-field limit} of the system \eqref{eq:New}.

Alternatively, one may adopt a statistical point of view and suppose that the initial position-velocities $z_1,\ldots,z_N$ are themselves random variables with some exchangeable (i.e., invariant under permutations of particle labels) law $\ga_N^0$. The starting point is now the \emph{Liouville equation}
\begin{equation}\label{eq:Lio}
\p_t\ga_N + \sum_{i=1}^N v_i\cdot\nabla_{x_i}\ga_N -\frac{2}{N}\sum_{i=1}^N\sum_{\substack{1\leq j\leq N}}\nabla W(x_i-x_j)\cdot\nabla_{v_i}\ga_N = 0.
\end{equation}
Given a solution $\ga_N$ of the Liouville equation \eqref{eq:Lio}, we form the sequence of marginals
\begin{equation}
\ga_N^{(k)} \coloneqq \int_{\Rd^{N-k}}d\ga_N(\cdot,z_{k+1},\ldots,z_N), \qquad 1\leq k\leq N,
\end{equation}
where by convention $\ga_N^{(N)} \coloneqq \ga_N$. The marginals $(\ga_N^{(k)})_{k=1}^N$ satisfy the \emph{(classical) BBGKY hierarchy} of equations
\begin{multline}\label{eq:BBGKY}
\partial_t \ga^{(k)}_N + \sum_{i=1}^k v_i \cdot \nabla_{x_i} \ga_N^{(k)} = \frac{2}{N} \sum_{1\leq i,j\leq k} \nabla W(x_i - x_j) \cdot \nabla_{v_i} \ga_N^{(k)} \\
+ \frac{2(N-k)}{N} \sum_{i=1}^k \int_{\R^{2d}}dz_{k+1} \nabla W(x_i - x_{k+1}) \cdot \nabla_{v_i} \ga_{N}^{(k+1)}.
\end{multline}
Letting $N\rightarrow\infty$, the first term on the right-hand side of \eqref{eq:BBGKY} is formally $O(1/N)$ and therefore vanishes, while the prefactor of the second term becomes $2$, leading to the \emph{Vlasov hierarchy}
\begin{equation}\label{eq:VlH}
\p_t \ga^{(k)} + \sum_{i=1}^k v_i\cdot\nabla_{x_i}\ga^{(k)} = 2\sum_{i=1}^k \int_{\R^{2d}}dz_{k+1}\nabla W(x_i-x_{k+1})\cdot\nabla_{v_i}\ga^{(k+1)}, \qquad k\geq 1.
\end{equation}
The equations \eqref{eq:VlH} form an infinite coupled system of linear equations, where the coupling of the $k$-th marginal to the $(k+1)$-th marginal reflects that there are only binary interactions in \eqref{eq:New}. Making the ansatz that there exists a $\ga^t$ such that $\ga^{(k),t} = (\ga^t)^{\otimes k}$ for every $k\in\N$ and $t\geq 0$, one computes that $(\ga^{(k)})_{k=1}^\infty$ is a solution of the Vlasov hierarchy if and only if $\ga$ is a solution of the Vlasov equation. Thus, if for each $k\in\N$, the marginals $\ga_N^{(k),0}$ of the initial $N$-particle laws converge to $(\ga^0)^{\otimes k}$ as $N\rightarrow\infty$, then one expects---or hopes---that the time evolutions $\ga_N^{(k),t}$ converge to $(\ga^t)^{\otimes k}$ as $N\rightarrow\infty$. This asymptotic factorization is referred to as \emph{propagation of molecular chaos}. 

One can make the formal derivation sketched in the preceding paragraphs rigorous in the sense that the empirical measure is shown to converge weakly to a solution of the Vlasov equation as $N\rightarrow\infty$ under suitable assumptions on $W$. This in turn implies propagation of chaos in a certain topology (see \cite{GMR2013, HM2014} and references therein). The convergence when $W$ is regular (e.g., $\nabla W$ Lipschitz) is classical \cite{NW1974, BH1977, Dobrushin1979, Duerinckx2021gl}. However, the situation when the force $\nabla W$ fails to be Lipschitz is much less understood. In particular, it is an outstanding problem to prove the mean-field limit for Vlasov-Poisson, except in dimension one \cite{Trocheris1986, Hauray2014}. Some results have been obtained for forces $\nabla W$ which are bounded \cite{JW2016} or even mildly singular (e.g., $|x|^{-\alpha}$ for $\alpha<1$) \cite{HJ2007, HJ2015}. In other directions, mean-field convergence has been shown for Coulomb-type potentials which are regularized at some small scale $\ep_N$ vanishing as $N\rightarrow\infty$ \cite{BP2016, Lazarovici2016, LP2017, Grass2021} or when the initial data is of so-called monokinetic type \cite{Serfaty2020}. For reviews of Vlasov mean-field limits, the reader may consult \cite{Jabin2014, Golse2016ln} and, in particular, the recent lecture notes \cite{Golse2022ln}.

\medskip

The formal derivation from above, let alone any of the just cited mathematical results, does not give any information on how the Hamiltonian structure of the Vlasov equation itself arises from that of Newton's second law. To the best of our knowledge, a detailed description of the Hamiltonian structure for the Vlasov equation as itself a ``mean-field limit'' (in other words, a derivation of the Vlasov Hamiltonian structure) remains an unanswered question. Some partial progress has been made: Marsden, Morrison, and Weinstein \cite{MMW1984} formally showed that the BBGKY hierarchy equations \eqref{eq:BBGKY} are Lie-Poisson (i.e., they are Hamiltonian with respect to the canonical Poisson bracket on the dual of a Lie algebra) and that this hierarchy bracket is such that its pullback under the map corresponding to taking marginals equals the Poisson bracket for the Liouville equation. However, Marsden \emph{et al.}'s expressed goal of showing ``how this structure is inherited by truncated systems, providing a statistical basis for recently discovered bracket structures for plasma systems,'' such as those identified in \cite{IT1976, Morrison1980, Gibbons1981, MW1982} for the Vlasov equation and \cite{MG1980, Morrison1982pois, MRW1984} for other related equations, has not been realized prior to this paper.

\subsection{Informal description of main results}\label{ssec:intromr}
In this article, we settle the question of \cite{MMW1984} on providing a statistical foundation for the Poisson structure underlying the Vlasov equation, by giving a rigorous derivation of the Hamiltonian structure, both the underlying Poisson vector space and Hamiltonian functional, directly from Newtonian mechanics in the limit as $N\rightarrow\infty$. Our results parallel the previous subsection's discussion of the formal derivation of the Vlasov equation, but from a perspective focused on geometric structure, in particular morphisms between different Lie algebras and Lie-Poisson spaces, as well as limits of such structures as the number of particles $N\rightarrow\infty$. In addition to placing the formal calculations of \cite{MMW1984} on firm functional-analytic footing by identifying appropriate spaces of functions and distributions, corresponding to observables and states, respectively, on which all brackets are well-defined, we show that operations in the formal derivation, such as taking the marginals of an $N$-particle distribution or forming the empirical measure from a position-velocity configuration, are Poisson morphisms (i.e., they preserve Poisson brackets). Moreover---and most importantly---we show that the Hamiltonian structure of the Vlasov equation, both the Lie-Poisson bracket and the Hamiltonian functional, may be interpreted as a \emph{``geometric mean-field limit,''} which is directly obtainable as the pullback of the Hamiltonian structure of the Vlasov hierarchy, both novel observations.

\cref{thm:maininf} stated below is an informal description of the main results of this paper. Of course, \cref{thm:maininf} is a gross caricature. The reader will forgive us for not being more precise at this stage, so as to maintain the accessibility of the introduction. A detailed description of the results, with all background material and notation explained, is given in \cref{sec:Out}, which is the technical introduction to the paper. It is important for the reader to understand that there is not a single main result, but a chain of connected results that should be considered in their totality.

\begin{theorem}[Informal statement of the main result]\label{thm:maininf}
\phantom{}
\begin{description}
\item[$N$-particle Liouville] Let $N\in\N$ denote the number of particles.
\begin{itemize}[leftmargin=*]
\item
There exists a Lie algebra $\g_N$ of symmetric $\Cc^\infty$ functions on $\Rd^N$, constituting $N$-particle observables. Scaling the standard Poisson bracket by $N$, yields a Lie bracket $\comm{\cdot}{\cdot}_{\g_N}$.
\item
Consequently, the strong dual $\g_N^*$, consisting of symmetric distributions with compact support on $\Rd^N$, has a Lie-Poisson bracket $\pb{\cdot}{\cdot}_{\g_N^*}$, with respect to which the Liouville equation \eqref{eq:Lio} admits a Hamiltonian formulation.
\item
Additionally, there is a Poisson morphism $\iota_{Lio}: \Rd^N \rightarrow \g_N^*$ sending a position-velocity configuration $\uz_N$ to a symmetric probably measure (the law) on $\Rd^N$, in particular mapping solutions of Newton's equations \eqref{eq:New} to solutions of the Liouville equation.
\end{itemize}
\item[$N$-particle BBGKY]
\phantom{}
\begin{itemize}[leftmargin=*]
\item
The Lie algebras $(\g_k,\comm{\cdot}{\cdot}_{\g_k})_{k=1}^N$ collectivize into a Lie algebra $(\G_N,\comm{\cdot}{\cdot}_{\G_N})$ of $N$-hierarchies of observables $F=(f^{(k)})_{k=1}^N \in \G_N = \bigoplus_{k=1}^N\g_k$.
\item
On the strong dual space $\G_N^* = \prod_{k=1}^N \g_k^*$ consisting of $N$-hierarchies of states $\Ga=(\ga^{(k)})_{k=1}^N$, there is an associated Lie-Poisson bracket $\pb{\cdot}{\cdot}_{\G_N^*}$, with respect to which the BBGKY hierarchy \eqref{eq:BBGKY} admits a Hamiltonian formulation.
\item
Additionally, the map $\iota_{mar}: \g_N^*\rightarrow \G_N^*$ formed from taking $k$-particle marginals is a Poisson morphism, mapping solutions of the Liouville equation to solutions of the BBGKY hierarchy.
\end{itemize}
\item[Vlasov hierarchy]
\phantom{}
\begin{itemize}[leftmargin=*]
\item
The spaces $\G_N$ ordered by inclusion form an increasing sequence with limit $\G_\infty=\bigoplus_{k=1}^\infty \g_k$. Any $F,G\in\G_\infty$ also must belong to $\G_N$ for $N$ sufficiently large, therefore one can compute the limit of $\comm{F}{G}_{\G_N}$ as $N\rightarrow\infty$, which acquires a simpler form due to vanishing of $O(1/N)$ terms in the expansion. This limit, denoted $\comm{F}{G}_{\G_\infty}$, defines a Lie bracket for $\G_\infty^*$.
\item
On the strong dual $\G_\infty^* = \prod_{k=1}^\infty \g_k^*$, there is an associated Lie-Poisson bracket $\pb{\cdot}{\cdot}_{\G_\infty^*}$ well defined for any $\Fc,\Gc\in\Cc^\infty(\G_\infty^*)$. Restricting to a unital subalgebra $\A_\infty$ generated by expectation and constant functionals, $\G_\infty^*$ acquires a weak Poisson structure, with respect to which the Vlasov hierarchy \eqref{eq:VlH} is Hamiltonian.
\end{itemize}
\item[From Vlasov hierarchy to Vlasov]
\phantom{}
\begin{itemize}[leftmargin=*]
\item The factorization map $\iota: \g_1^*\rightarrow \G_\infty^*$ defined by $\ga\mapsto (\ga^{\otimes k})_{k=1}^\infty$ is a Poisson morphism.
\item The pullback of the Vlasov hierarchy Hamiltonian under $\iota$ equals the Vlasov Hamiltonian.
\item In this sense, the Hamiltonian structure of the Vlasov equation \eqref{eq:Vl} is the pullback of the Hamiltonian structure of the Vlasov hierarchy, and the map $\iota$ sends solutions of the Vlasov equation to the Vlasov hierarchy.
\end{itemize}
\item[From Newton to Vlasov]
\phantom{=}
\begin{itemize}[leftmargin=*]
\item 
Finally, one can connect the $N$-particle Poisson space to the Vlasov Poisson space through the empirical measure assignment $\iota_{EM}: \Rd^N\rightarrow \g_1^*$, which is a Poisson morphism.
\item
The pullback under $\iota_{EM}$ of the Vlasov Hamiltonian equals the energy per particle of \eqref{eq:New}, and therefore $\iota_{EM}$ sends solutions of the Newtonian system to weak solutions of the Vlasov equation.
\end{itemize}
\end{description}
\end{theorem}

\begin{remark}
The reader might wonder about the relevance of  \cref{thm:maininf} for the Vlasov-Poisson equation since the Coulomb potential is not in $\g_1$, failing to be smooth at the origin. While this observation is correct, it is not of great importance, since at the $N$-particle level, one can always regularize the potential $W$ at some small scale, such as the typical interparticle distance $N^{-1/d}$.\footnote{In fact, Vlasov-Poisson dynamics have been derived as the mean-field limit of Newtonian $N$-particle dynamics with such a regularization \cite{LP2017}.} Similarly, it is classical that the Cauchy problem for Vlasov-Poisson is stable with respect to regularizations of $W$ (e.g., see \cite{Hauray2014}). Furthermore, the primary significance of \cref{thm:maininf} is not at the level of Hamiltonian functionals, which depend on the potential $W$, but rather at the level of the underlying Lie algebras and Lie-Poisson brackets, which are completely independent of $W$. If one wishes to have a formalism that directly allows for singular $W$, then one should work with scales of function spaces on $\Rd^{k}$ (e.g., Sobolev) and their duals. In which case, the notion of a Hamiltonian vector field must be modified to allow for mappings from a higher regularity index of the scale to a lower regularity index. Additionally, the states in \cref{thm:maininf} are assumed to have compact support in phase space. This a qualitative, technical assumption stemming from the isomorphism between $\Cc^\infty(\R^n)^*$ and the space $\Ec'(\R^n)$ of distributions with compact support. It is harmless from the perspective of the Vlasov equation due to finite speed of propagation and stability with respect to compact approximation of the initial data.

\end{remark}

Let us state clearly that \cref{thm:maininf} does not address the derivation of \emph{dynamics} of the Vlasov equation from Newton's second law or the Liouville equation in the vein of the works on Vlasov mean-field limits mentioned in \cref{ssec:intromot}. Instead, our work is \emph{complementary}, answering the question of \cite{MMW1984} on a derivation of the Vlasov bracket from $N$-particle brackets, which we argue is both independent of and unaddressed by these prior works on Vlasov mean-field limits. A worthwhile goal for the future is to unify this perspective of derivation of geometric structure with the traditional perspective of derivation of dynamics, using the former to say new things about the latter. In other contexts, the geometric structure of an equation has played an important role in understanding its well-posedness or long-time dynamics. As an example of this interplay to which we aspire, we mention the seminal work of Arnold \cite{Arnold1966, Arnold1969} and Ebin and Marsden \cite{EM1970} for the incompressible Euler equation.

\subsection{Method of proof}\label{ssec:intropf}
Our method for proving \cref{thm:maininf} is heavily inspired by the work \cite{MNPRS2020} of the last four co-authors together with D. Mendelson. This cited work gave a complete, mathematically rigorous description of how the Hamiltonian structure of the nonlinear Schr\"odinger equation emerges in the limit as $N\rightarrow\infty$ from the Hamiltonian structure of the linear Schr\"odinger equation describing the many-body problem for $N$ interacting bosons. The approach of \cite{MNPRS2020} in turn was motivated by the use of the BBGKY hierarchy to derive the dynamics of nonlinear Schr\"odinger-type equations from the $N$-body Schr\"odinger problem \cite{Spohn1980, ABGT2004, AGT2007, ESY2006, ESY2007, ESY2009, ESY2010, KM2008, CP2014, CH2019}.\footnote{We also mention that the BBGKY hierarchy has been a tool \cite{NS1981, Spohn1980, Spohn1981, GMR2013}, though not as powerful, in the derivation of Vlasov dynamics.} While the results obtained in the present paper demonstrate the robustness of the hierarchy formalism developed in \cite{MNPRS2020}, in the sense that there are algebraic parts to our work for which the computations of \cite{MNPRS2020} transfer with little modification, there are important analytic differences between the quantum setting and the classical setting of this work, as well as new challenges encountered here.

The first obvious difference with \cite{MNPRS2020} we highlight is the nature of observables, states, and brackets in classical mechanics vs quantum. Here, the observables (for $k$ particles) are $\Cc^\infty$ functions $f:\Rd^k \rightarrow\R$ invariant under permutation of particle labels, while in the quantum setting, they are continuous linear operators $A\in \mathcal{L}(\Sc_s((\R^d)^k), \Sc_s'((\R^d)^k))$ from the symmetric Schwartz space to the space of symmetric tempered distributions. Similarly, the states here (again for $k$ particles) are distributions $\ga$ on $\Rd^k$ with compact support and with a dually defined permutation symmetry, while in the quantum setting, they are continuous linear operators $A\in \mathcal{L}(\Sc_s'((\R^d)^k),\Sc_s((\R^d)^k))$ from the space of symmetric tempered distributions to the symmetric Schwartz space. The fact that we do not need to consider very irregular distribution-valued operators is a technical advantage of the classical setting over the quantum. It is an interesting observation that the observables are irregular while the states are regular, in terms of Schwartz kernels, in the quantum setting, while in the classical setting the opposite is true. Lastly, the Poisson structures here are all built from the standard Poisson structure on Euclidean space, whereas in the quantum case, they are built from the commutator of two operators on an $L^2$ space.

The next difference with \cite{MNPRS2020} is that the results of the present paper are stronger and the overall proof is significantly less \emph{ad hoc}. Namely, in \cite{MNPRS2020}, we relied on the notion of a weak Poisson vector space (see \cref{def:WPVS}), originally introduced in \cite{NST2014}, at all stages of the derivation. The adjective ``weak'' here refers to the fact that the Poisson bracket is no longer assumed to admit a Hamiltonian vector field for every $\Cc^\infty$ functional, but only for functionals in a unital subalgebra $\A$, which itself is part of the data specifying a weak Poisson vector space. Much of the difficulty throughout \cite{MNPRS2020} boils down to identifying an $\A$ which is large enough to contain all functionals of interest (e.g., BBGKY, GP Hamiltonians) but still small enough so that the brackets can actually be defined. In contrast, the present article works with a notion of strong Poisson vector spaces (see \cref{def:PVS}) at the $N$-particle level, in which the Poisson bracket is assumed to admit a Hamiltonian vector field for every $\Cc^\infty$ functional, omitting the need to restrict to a subalgebra. We then show that our dual spaces $\g_k^*, \G_N^*$ satisfy certain topological conditions (in particular, they are $k^\infty$ spaces; see \cref{def:kinfty}) and that our Lie brackets are jointly continuous, allowing us to use an abstract theorem of Gl\"ockner \cite{Glockner2009} (see \cref{thm:Glockner} for a review) to obtain a well-defined Lie-Poisson structure. To the best of our knowledge, our work is the first application of Gl\"ockner's theorem for problems involving Hamiltonian PDE. Unfortunately, we run into a technical issue at the infinite-particle level when attempting to verify the conditions to apply Gl\"ockner's theorem for $\G_\infty^*$---namely, showing that this is a $k^\infty$ space, given the $k^\infty$ property is not necessarily preserved under countable products. To overcome this issue, we resort to directly verifying that for the subalgebra $\A_\infty$ generated by constants and expectation functionals (see \eqref{eq:Ainfdef}), which are the classical analogue of the trace functionals from \cite{MNPRS2020}, there is a weak Poisson structure for $\G_\infty^*$. Importantly, this algebra $\A_\infty$ contains the Vlasov hierarchy Hamiltonian.

\subsection{Future directions}\label{ssec:introfd}
This article and the prior work \cite{MNPRS2020} raise the interesting question of how to connect the classical and quantum worlds through the limit $\hbar\rightarrow 0$. We believe that by combining geometric structures from each of these papers and relating them through the Wigner transform, which is a Poisson morphism, the combined mean-field limit $N\rightarrow \infty$ and $\hbar\rightarrow 0$ can be handled to obtain a rigorous derivation of the Hamiltonian structure of the Vlasov equation directly from the $N$-body Schr\"odinger equation. In other words, the diagram in Figure 1 commutes in terms of geometric structure. We plan to investigate this direction in future work.

\begin{figure}[H]\label{fig}
\caption{Mean Field and Classical Limits}
\begin{tikzpicture}[scale=.6]
\node (A) at (-12,3) {\tiny $N$-Schr\"odinger/Quantum};
\node (B) at (5,3) {\tiny Hartree/Quantum} ;
\node (C) at (-12,-3) {\tiny $N$-Liouville/Classical};
\node (D) at (5,-3) {\tiny Vlasov/Classical};
\path[->,font=\scriptsize]
(A) edge node[above]{\tiny $N\rightarrow \infty$} (B)
(A) edge node[right]{\tiny $\hbar\rightarrow 0$} (C)
(A) edge node[above]{\tiny \hspace{15mm}$N\rightarrow\infty,\ \hbar\rightarrow 0$} (D)
(B) edge node[right]{\tiny $\hbar\rightarrow 0$} (D)
(C) edge node[above]{\tiny $N\rightarrow \infty$} (D);
\end{tikzpicture}
\end{figure}
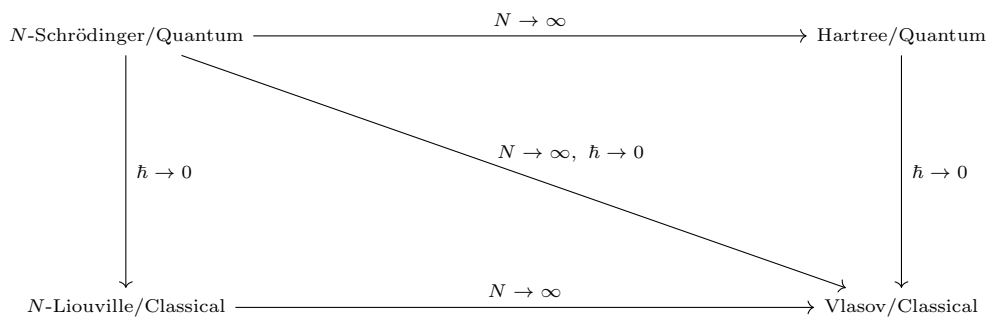

It is appropriate to conclude this subsection by mentioning some works that are related to the spirit of our paper in terms of understanding the role of the Hamiltonian formulation of PDE in mathematical physics. We first mention some recent work of Chong \cite{Chong2022} which exhibits a Poisson map from the Poisson manifold underlying the Vlasov equation to the Poisson manifold underlying the compressible Euler equation. We also mention impressive work of Khesin \emph{et al.} \cite{KMM2019, KMM2019, KMM2021} which shows that the Madelung transformation from wave functions to hydrodynamic variables is a K\"ahler morphism and which develops a geometric framework for Newton's equations on groups of diffeomorphisms and spaces of probability densities, covering a number of equations, including (in)compressible fluid and (non)linear Schr\"odinger equations. Finally, we mention the interesting work of Fr\"ohlich \emph{et al.} \cite{FTY2000, FKP2007, FKS2009} on the relationship between quantization, mean-field theory, and the dynamics of the Hartree and Vlasov equations, which are informed by the Hamiltonian perspective.

\subsection{Acknowledgments}
This material is based upon work supported by both the National Science Foundation under Grant No. DMS-1929284 and the Simons Foundation Institute Grant Award ID 507536 while the authors were in residence at the Institute for Computational and Experimental Research in Mathematics in Providence, RI, during the fall 2021 semester. The last four named co-authors thank Dana Mendelson for a number of discussions from which this project has benefited.

\section{Blueprint of the paper}\label{sec:Out}
We now present an outline of our main results and discuss their proofs. This section is intended as a complete blueprint or schematic of the entire paper. We recommend that one reads through this section in its entirety before proceeding to \Cref{sec:pre,sec:Ngeom,sec:infgeom,sec:Ham} and then regularly refer back to it during the reading of these subsequent sections.  At the end of this section (see \cref{ssec:Outorg}), we elaborate on the organization of the remaining portion of the paper. Finally, there are some abstract notions, which may already be familiar to the reader, that are referenced in \Cref{ssec:OutNew,ssec:OutN,ssec:Outinf,ssec:OutVl}, but whose definitions are deferred to \Cref{sec:pre}. This choice of ordering is so as not to get bogged down in material that is not the central focus of the paper.

For the reader's benefit, we include Table \ref{tab:not}, located at the end of the paper, as a guide to the various notation appearing in this work. In the table, we either provide an explanation of the notation or a reference for where the notation is first introduced and defined. Some of the notation in the table has already appeared in the introduction. In such cases, we give references to where the notation first appears in subsequent sections.


\subsection{Newton/Liouville equations}\label{ssec:OutNew}
Consider the function $H_N$ from \eqref{eq:HNdef} with $W\in C^\infty(\mathbb{R}^d)$ satisfying $W(-x) = W(x)$. We recall from the introduction the rotation matrix $\J(x,v) = (-v,x)$ and the block-diagonal matrix $\J_N$ with diagonal entries $\J$. The \emph{standard symplectic structure} on $\Rd^N$ is given by the form
\begin{equation}
\om_N(\uz_N, \uw_N) \coloneqq -\J_N\uz_N\cdot \uw_N, \qquad \forall \uz_N,\uw_N\in \Rd^N,
\end{equation}
where $\cdot$ denotes the Euclidean inner product on $\Rd^N$. We recall that the \emph{Hamiltonian vector field} $X_{H_N}$ associated to $H_N$ is uniquely defined by the formula
\begin{equation}
\ds H_N[\uz_N](\d\uz_N) = \om_N(X_{H_N}(\uz_N), \d\uz_N), \qquad \forall \uz_N,\d\uz_N\in \Rd^N.
\end{equation}
Set $\nabla_{z_j} = (\nabla_{x_j},\nabla_{v_j})$, where $\nabla_{x_j} = (\p_{x_j^1},\ldots,\p_{x_j^d})$ and $\nabla_{v_j}=(\p_{v_j^1},\ldots,\p_{v_j^d})$. Writing \\$\nabla_{\uz_N} = (\nabla_{z_1},\ldots,\nabla_{z_N})$, we compute from the property $\J^2=-\I$ together with the definition of the gradient that
\begin{align}
\ds H_N[\uz_N](\d\uz_N) = \nabla_{\uz_N}H_N(\uz_N) \cdot \d\uz_N = - \J_N^2\nabla_{\uz_N}H_N(\uz_N) \cdot \d\uz_N = -\J_N\Big(\J_N\nabla_{\uz_N}H_N(\uz_N)\Big)\cdot\d\uz_N,
\end{align}
which implies that $X_{H_N}(\uz_N) = \J_N\nabla_{\uz_N}H_N(\uz_N)$. Thus, the functional $H_N$ and the symplectic form $\om_N$ together define the \emph{Hamiltonian equation of motion}
\begin{equation}\label{eq:Newsym}
\dot{\uz}_N^t = X_{H_N}(\uz_N^t),
\end{equation}
which is equivalent to \eqref{eq:New}. As is well-known, the symplectic form $\omega_N$ induces a \emph{canonical Poisson bracket} on $\Rd^N$ by
\begin{equation}\label{eq:pbN}
\pb{F}{G}_{\Rd^N}(\uz_N) \coloneqq \om_N(X_F(\uz_N),X_G(\uz_N)), \qquad \forall F,G\in \Cc^\infty(\Rd^N), \ \uz_N \in \Rd^N,
\end{equation}
referred to as the \emph{standard Poisson structure} on $\Rd^N$. Thus, the symplectic formulation \eqref{eq:Newsym} of Newton's second law of motion can be equivalently written in Poisson form as
\begin{equation}\label{eq:Newpb}
\frac{d}{dt}F(\uz_N^t) = \pb{F}{H_N}_{\Rd^N}(\uz_N^t), \qquad \forall F\in \Cc^\infty(\Rd^N).
\end{equation}
To evaluate $N\rightarrow\infty$ limits, it is convenient to rescale the Poisson bracket and modify the Hamiltonian $H_N$ as follows:
\begin{equation}\label{eq:NewHamdef}
\Hc_{New} \coloneqq \frac{1}{N}\paren*{H_N+W(0)} \qquad \text{and} \qquad \pb{\cdot}{\cdot}_N \coloneqq N\pb{\cdot}{\cdot}_{\Rd^N},
\end{equation}
with the subscript ``New'' abbreviating Newton. Evidently, $\Hc_{New}$ depends on $N$, but we omit this dependence from our notation, as it will be clear from context. The addition of the term $W(0)$ in the Hamiltonian is harmless: it is a constant and so it does not change the Hamiltonian vector field. Its inclusion reflects the fact that we do not need to exclude self-interaction since $W$ is continuous at the origin. With these rescalings and translation, the Poisson formulation \eqref{eq:Newpb} becomes
\begin{equation}
\frac{d}{dt}F(\uz_N^t) = \pb{F}{\Hc_{New}}_N(\uz_N^t), \qquad \forall F\in \Cc^\infty(\Rd^N).
\end{equation}

For each $k \in \N$, we define the set 
\begin{equation}\label{eq:gkdef}
\g_k\coloneqq \mathcal{C}^\infty_s(\Rd^k) \coloneqq \left\{ f \in \mathcal{C}^\infty( \Rd^k) :  f(z_{\pi(1)},\ldots,z_{\pi(k)}) = f(\uz_k), \ \forall \pi \in \mathbb{S}_k\right\}. 
\end{equation}
In the sequel, we will use the shorthand $(f\circ\pi)(\uz_k)\coloneqq f(z_{\pi(1)},\ldots,z_{\pi(k)})$. In other words, the space $\g_k$ consists of smooth real-valued functions which are invariant under permutations of particle labels. We endow the set $\mathfrak{g}_k$ with the locally convex topology induced by the semi-norms
\begin{equation}\label{seminorm def}
\rho_{K, n}: \mathfrak{g}_k \rightarrow [0,\infty) , \qquad \rho_{K,n} (f) \coloneqq  \sup_{|\alpha | \leq n} \Vert  \partial^\alpha f \Vert_{L^\infty(K)}, \qquad K\subset \Rd^k, \ n \in \N,
\end{equation}
where $K$ above is compact and the supremum is taken over all multi-indices $\alpha\in(\N_0)^{2dk}$ with order at most $n$. We then regard $\g_k$ as a real topological vector space, elements of which are our \emph{$k$-particle observables}. We introduce a bracket on $\mathfrak{g}_k$ which will give the space the structure of a Lie algebra. For each $k \in \N$, we define
\begin{equation}\label{eq:gkLBdef}
\comm{\cdot}{\cdot}_{\g_k} : \g_k\times\g_k \rightarrow \g_k, \qquad  \comm{f}{g}_{\g_k}\coloneqq k\pb{f}{g}_{\Rd^k},
\end{equation}
where $\pb{\cdot}{\cdot}_{\Rd^k}$ is the standard Poisson bracket on $\Rd^k$.

\begin{restatable}{proposition}{gkLA}\label{prop:gkLA}
For each $k \in \N$, the pair $(\mathfrak{g}_k, \comm{\cdot}{\cdot}_{\g_k})$ is a Lie algebra in the sense of \cref{def:LA} below. Furthermore, the bracket $\comm{\cdot}{\cdot}_{\g_k}$ is a continuous bilinear map. 
\end{restatable}

Next, for each $k\in\N$, we define the real topological vector space $\mathfrak{g}_k^*$ to be the strong dual of $\mathfrak{g}_k$. It can be characterized as follows:
\begin{equation}\label{eq:gk*def}
\g_k^* = \{\ga\in  \mathcal{C}^\infty(\Rd^k)^* : \pi\# \ga = \ga,\ \forall \pi\in\Ss_k\},
\end{equation}
where $\pi \#\ga (f) = \ga (f \circ \pi)$ for $f \in \mathcal{C}^\infty(\Rd^k)$. Using the isomorphism $\mathcal{C}^\infty(\Rd^k)^*\cong \Ec'(\Rd^k)$, elements of $\g_k^*$, which we call \emph{$k$-particle states}, are distributions on $\Rd^k$ with compact support and which are invariant under the action of $\Ss_k$ (i.e., the permutation of particle labels). The space $\g_k^*$ has the desirable property of being a reflexive, (DF) Montel space (see \cref{lem:gkremon}).

The canonical Lie-Poisson bracket induced by the Lie bracket $\comm{\cdot}{\cdot}_{\g_k}$ gives $\g_k^*$ the structure of a Poisson vector space in the precise sense of \cref{def:PVS}. In fact, the space $\g_k^*$ has stronger topological properties, namely it is a $k^\infty$ space (see \cref{def:kinfty}) that make it an example of a reflexive, locally convex Poisson vector space as defined in \cref{def:rlcps}. We will use these stronger topological properties to prove this Lie-Poisson assertion by appealing to the aforementioned ``black box'' theorem of Gl\"ockner recalled in \cref{thm:Glockner} below.

Before stating the result, we record the following important observation. For any $\Gc\in \mathcal{C}^\infty(\g_k^*)$, we have by definition of the G\^ateaux derivative that $\ds\mathcal{G} \in \mathcal{C}^\infty(\g_k^*; \g_k^{**})$. So, for any $\mu\in \g_k^*$ we have that $\ds\mathcal{G}[\mu]\in \mathfrak{g}_k^{**}$, that is $\ds\mathcal{G}[\mu]$ is a continuous linear functional on $\mathfrak{g}_k^*$. Since we have the isomorphism $\mathfrak{g}_k^{**}\cong \mathfrak{g}_k$, we are justified in making the identification
\begin{equation}
\ds\mathcal{G}[\mu]\eqqcolon g_\mu \quad \text{and} \quad  \ds\mathcal{H}[\mu] \eqqcolon h_\mu, \qquad \text{where} \ g_\mu, h_\mu \in\mathfrak{g}_k.
\end{equation}
We then regard $\comm{g_\mu}{h_\mu}_{\mathfrak{g}_k} $ as an element in $\mathfrak{g}_k $, and we denote the pairing of $\mu$ and $\comm{g_\mu}{h_\mu}_{\mathfrak{g}_k} $ as the ``integral'' 
\begin{equation}\label{eq:brid}
\int_{(\R^{2d})^k}d\mu \comm{g_\mu}{h_\mu}_{\mathfrak{g}_k}.
\end{equation}
This identification will be made throughout this paper.

\begin{restatable}{proposition}{gkLP}\label{prop:gkLP}
For observables $\mathcal{G}, \mathcal{H} \in \mathcal{C}^\infty( \g_k^*)$ and $k$-particle state $\gamma \in \g_k^*$, we define the bracket
\begin{equation}\label{eq:gk*LPdef}
\pb{\mathcal{G}}{\mathcal{H}}_{\g_k^*}(\ga) \coloneqq \ipp*{\comm{\ds\mathcal{G}[\ga]}{\ds\mathcal{H}[\ga]}_{\g_k},\ga}_{\mathfrak{g}_k-\mathfrak{g}_k^*}.
\end{equation}
Then $(\g_k^*,\pb{\cdot}{\cdot}_{\g_k^*})$ is a reflexive, locally convex Lie-Poisson space in the sense of \cref{def:rlcps}. 
\end{restatable}

The reader may check that for any $N\geq 2$, the function $H_N\in\g_N$, hence $\Hc_{New}\in\g_N$. Therefore, it makes sense to introduce the \emph{Liouville Hamiltonian functional}
\begin{equation}\label{eq:LioHam}
\mathcal{H}_{Lio}(\ga) \coloneqq \ipp{\Hc_{New}, \ga}_{\g_N-\g_N^*}, \qquad \forall \ga\in\g_N^*.
\end{equation}
Evidently, $\Hc_{Lio}$ depends on $N$, though we omit this dependence from our notation. Being linear and continuous (by consequence of the separate continuity of the distributional pairing), $\mathcal{H}_{Lio} \in\Cc^\infty(\g_N^*)$. With $\Hc_{Lio}$ and \cref{prop:gkLP}, the Liouville equation may be written in Hamiltonian form. \cref{prop:LioHam} stated below is the classical counterpart to the fact that the von Neumann equation from quantum mechanics is Hamiltonian (see \cite[pp. 17-18]{MNPRS2020}).

\begin{restatable}{proposition}{LioHam}\label{prop:LioHam}
Let $I \subset \R$ be a compact interval and $N \in \N$. Then $\ga \in \mathcal{C}^\infty(I, \mathfrak{g}_N^*)$ is a solution to the Liouville equation \eqref{eq:Lio} if and only if 
\begin{equation}
\dot{\ga} = X_{\mathcal{H}_{Lio}}(\ga),
\end{equation}
where $X_{\mathcal{H}_{Lio}}$ is the unique Hamiltonian vector field generated by the Hamiltonian $\mathcal{H}_{Lio}$ with respect to the Lie-Poisson vector space $(\g_N^*, \pb{\cdot}{\cdot}_{\g_N^*})$. 
\end{restatable}

Given a position-velocity configuration $\uz_N = (z_1,\ldots,z_N) \in \Rd^N$, we can associate a symmetric probability measure on $\Rd^N$ by defining
\begin{equation}\label{eq:iotaLiodef}
\ga_N \coloneqq \frac{1}{N!}\sum_{\pi \in \Ss_N} \d_{z_{\pi(1)}}\otimes\cdots\otimes\d_{z_{\pi(N)}}.
\end{equation}
Evidently, the right-hand side is an element of $\g_N^*$. We call this assignment $\iota_{Lio}:\Rd^N\rightarrow\g_N^*$ the \emph{Liouville map}. The reader may check from the invariance of $\Hc_{New}$ under the action of $\Ss_N$ that
\begin{equation}
\Hcl(\iota_{Lio}(\uz_N)) = \frac{1}{N!}\sum_{\pi\in\Ss_N} \Hc_{New}(z_{\pi(1)},\ldots,z_{\pi(N)}) = \Hc_{New}(\uz_N).
\end{equation}
Moreover, the map $\iota_{Lio}$ is a morphism of Poisson vector spaces, implying that $\iota_{Lio}$ maps solutions of the Newtonian system \eqref{eq:New} to solutions of the Liouville equation \eqref{eq:Lio} (see \cref{rem:NewLio}).

\begin{restatable}{proposition}{LioPM}\label{prop:LioPM}
The map $\iota_{Lio} \in \Cc^\infty(\Rd^N, \g_N^*)$ defines a morphism of the Poisson vector space $(\Rd^N, \pb{\cdot}{\cdot}_{N})$ into the Lie-Poisson space $(\g_N^*,\pb{\cdot}{\cdot}_{\g_N^*})$:
\begin{equation}
\forall \Fc,\Gc\in \Cc^\infty(\g_N^*), \qquad \pb{\iota_{Lio}^*\Fc}{\iota_{Lio}^*\Gc}_{N} = \iota_{Lio}^*\pb{\Fc}{\Gc}_{\g_N^*},
\end{equation}
where $\iota_{Lio}^*$ denotes the pullback under $\iota_{Lio}$.
\end{restatable}

\subsection{The Lie algebra $\G_N$ and Lie-Poisson space $\G_N^*$}\label{ssec:OutN}
Building on the previous subsection, we transition to discussing finite hierarchies of observables and states.

For $N\in\N$, we define the algebraic direct sum
\begin{equation}\label{eq:GNdef}
\mathfrak{G}_N \coloneqq \bigoplus_{k=1}^N \mathfrak{g}_k
\end{equation}
and endow this vector space with the product topology (note that the direct sum is a direct product since the number of summands is finite). This turns $\G_N$ into a locally convex real topological vector space. We refer to elements of $\G_N$ as \emph{$N$-hierarchies of observables}, alternatively \emph{observable $N$-hierarchies}. The Lie brackets $\comm{\cdot}{\cdot}_{\g_1},\ldots,\comm{\cdot}{\cdot}_{\g_N}$ induce a Lie algebra structure on $\G_N$ as follows.

For $N\in\N$ and $1\leq k\leq N$, consider the map
\begin{equation}\label{eq:epkNdef}
\ep_{k,N}: \g_k\rightarrow\g_N, \qquad   \epsilon_{k,N}(f^{(k)})(\uz_N) \coloneqq  \frac{1}{|P_k^N|}\sum_{(j_1, \dots, j_k) \in P_k^N} f_{(j_1,\ldots,j_k)}^{(k)}(\uz_N),
\end{equation}
where
\begin{equation}\label{eq:fj1jkdef}
f_{(j_1,\ldots,j_k)}^{(k)}(\uz_N) \coloneqq f^{(k)}(\uz_{(j_1,\ldots,j_k)}), \qquad \uz_{(j_1,\ldots,j_k)} \coloneqq (z_{j_1},\ldots,z_{j_k}),
\end{equation}
and we are defining the set of length-$k$ tuples drawn from $\{1,\ldots,N\}$ by
\begin{equation}\label{eq:PkNdef}
P_k^N \coloneqq \{ (j_1, \dots, j_k) : \text{$1 \leq j_i \leq N$ and $j_i$ distinct} \}.
\end{equation}
In the sequel, we will use the tuple shorthand $\bm{j}_k$ and write $f_{\bm{j}_k}^{(k)}$ and $f^{(k)}(\uz_{\bm{j}_k})$. One can show that the maps $\ep_{k,N}$ are continuous, linear, and therefore $\Cc^\infty$, and that they are injective (see \Cref{lem:epcont,lem:epinj} respectively). In words, the map $\ep_{k,N}$ embeds a $k$-particle observable in the space of $N$-particle observables. The maps $\ep_{k,N}$ have a filtration property (see \cref{lem:epfil}) asserting that $\comm{\ep_{\ell,N}(f^{(\ell)})}{\ep_{j,N}(g^{(j)})}_{\g_N}$ lies in the image of $\ep_{k,N}$; and using this filtration property together with the injectivity of $\ep_{k,N}$, we can define a Lie bracket on $\G_N$ by
\begin{equation}\label{eq:GNLBdef}
\comm{F}{G}^{(k)}_{\mathfrak{G}_N} \coloneqq \epsilon^{-1}_{k,N}\left(\sum_{\substack{1\leq \ell,j\leq N \\ \min( \ell + j -1, N) =k} } \comm{\epsilon_{\ell, N} (f^{(\ell)})}{\epsilon_{j,N}( g^{(j)})}_{\mathfrak{g}_N}\right), \qquad 1\leq k\leq N.
\end{equation}
In fact, there is an explicit formula for $\comm{F}{G}_{\G_N}^{(k)}$ (see \eqref{eq:GNLBform}), which we do not state here. For $N$ fixed, the maps $\{\ep_{k,N}\}_{k=1}^N$ also have the interesting property that they induce a Lie algebra homomorphism (see \cref{prop:LAhomsum})
\begin{equation}\label{eq:iotaepNdef}
\iota_{\ep}:\G_N\rightarrow\g_N, \qquad \iota_{\ep}(F) \coloneqq \sum_{k=1}^N \ep_{k,N}(f^{(k)}), \qquad \forall F = (f^{(k)})_{k=1}^N,
\end{equation}
the dependence of $\iota_{\ep}$ on $N$ being implicit. The map $\iota_{\ep}$ sends an $N$-hierarchy of observables to a single $N$-particle observable. After a series of lemmas establishing properties of these embedding maps $\ep_{k,N}$, we arrive at our main result for the $N$-particle hierarchy Lie algebra. 

\begin{restatable}{theorem}{GNLA}\label{prop:GNLA}
For any $N\in\N$, the pair $(\mathfrak{G}_N , \comm{\cdot}{\cdot}_{\G_N})$ is a Lie algebra in the sense of \cref{def:LA}. Furthermore, the bracket $\comm{\cdot}{\cdot}_{\G_N}$ is continuous.
\end{restatable}

If we define the real topological vector space $\mathfrak{G}_N^*$ as the strong dual of $\mathfrak{G}_N = \bigoplus_{k=1}^N \mathfrak{g}_k$, then using the duality of direct sums and products \cite[Proposition 2, \S 14, Chapter 3]{Kothe1969I}, we see that
\begin{equation}\label{eq:GN*def}
\mathfrak{G}_N^* = \left( \bigoplus_{k=1}^N \mathfrak{g}_k\right)^* \cong \prod_{k=1}^N \mathfrak{g}_k^*,
\end{equation}
where the right-hand side is endowed with the product topology. The canonical Lie-Poisson bracket induced by the Lie bracket $\comm{\cdot}{\cdot}_{\G_N}$ gives $\G_N^*$ the structure of a Poisson vector space in the precise sense of \cref{def:PVS}. Similar to $\g_N^*$, the space $\G_N^*$ is a reflexive, locally convex Poisson vector space as defined in \cref{def:rlcps}. We will use Gl\"{o}ckner's black box \cref{thm:Glockner} to prove this assertion.

\begin{restatable}{theorem}{GNLP}\label{prop:GNLP}
For functionals $\mathcal{G}, \mathcal{H} \in \mathcal{C}^\infty( \mathfrak{G}_N^*)$ and state $N$-hierarchy $\Gamma = (\ga^{(k)})_{k=1}^N\in\G_N^*$, we define the bracket
\begin{equation}\label{eq:GN*LPdef}
\pb{\mathcal{G}}{\mathcal{H}}_{\mathfrak{G}_N^*}(\Ga)\coloneqq \ipp*{\comm{\ds\Gc[\Ga]}{\ds\Hc[\Ga]}_{\G_N},\Ga}_{\G_N-\G_N^*} = \sum_{k=1}^N \ipp*{\comm{\ds\mathcal{G}[\Ga]}{\ds\mathcal{H}[\Ga]}_{\G_N}^{(k)},\ga^{(k)}}_{\mathfrak{g}_k-\mathfrak{g}_k^*}.
\end{equation}
Then $(\G_N^*,\pb{\cdot}{\cdot}_{\G_N^*})$ is a reflexive, locally convex Lie-Poisson space in the sense of \cref{def:rlcps}. 
\end{restatable}

To show that the BBGKY hierarchy \eqref{eq:BBGKY} is a Hamiltonian equation on the Poisson vector space $(\G_\infty^*,\pb{\cdot}{\cdot}_{\G_\infty^*})$, we introduce the $N$-particle \emph{BBGKY Hamiltonian functional}
\begin{equation}\label{eq:BBGKYham}
\mathcal{H}_{BBGKY}(\Ga) \coloneqq \ipp*{\W_{BBGKY},\Ga}_{\mathfrak{G}_N-\mathfrak{G}_N^*}, \qquad \forall \Ga\in\G_\infty^*,
\end{equation}
where
\begin{equation}\label{eq:BBGKYW}
\W_{BBGKY} \coloneqq \left(\frac{1}{2} |v_1|^2 , \frac{(N-1)}{N}W(x_1 - x_2) + \frac{W(0)}{N}, 0 ,\dots, 0\right)  \in \mathfrak{G}_N,
\end{equation}
the dependence on $N$ being implicit. Here, $|v_1|^2$ and $W(x_1-x_2), W(0)$ are viewed as functions on $\Rd$ and $\Rd^2$, respectively. Note that $\W_{BBGKY}$ is indeed an element of $\G_N$ by the assumption that $W \in \Cc^\infty(\R^d)$. Tautologically, $\Hbb$ is linear, and it is continuous by the separate continuity of the duality pairing; hence, $\Hbb\in \Cc^\infty(\G_N^*)$. Interpreting the integrals as distributional pairings, we have
\begin{equation}
\Hbb(\Ga) =\frac{1}{2}\int_{\R^{2d}}d\ga^{(1)}(z_1)|v_1|^2 + \frac{1}{N}\int_{(\R^{2d})^2}d\ga^{(2)}(z_1,z_2)\paren*{(N-1)W(x_1-x_2) + W(0)}.
\end{equation}
The following theorem, our main result for the BBGKY hierarchy, is the classical counterpart to \cite[Theorem 2.3]{MNPRS2020} for the quantum BBGKY hierarchy.

\begin{restatable}{theorem}{bbgky}\label{thm:bbgky}
Let $I \subset \R$ be a compact interval and $N \in \N$. Then $\Ga \in \mathcal{C}^\infty(I, \mathfrak{G}_N^*)$ is a solution to the BBGKY hierarchy \eqref{eq:BBGKY} if and only if 
\begin{equation}
    \dot{\Ga} = X_{\mathcal{H}_{BBGKY}}(\Ga),
\end{equation}
where $X_{\mathcal{H}_{BBGKY}}$ is the unique Hamiltonian vector field generated by the Hamiltonian $\mathcal{H}_{BBGKY}$ with respect to the Poisson vector space $(\mathfrak{G}_N^*, \pb{\cdot}{\cdot}_{\G_N^*})$. 
\end{restatable}

Returning to the homomorphism $\iota_{\ep}$ from \eqref{eq:iotaepNdef}, we can take its dual $\iota_{\ep}^*:\g_N^*\rightarrow \G_N^*$. Analogous to the quantum setting (cf. \cite[Proposition 5.29]{MNPRS2020}), this dual map is nothing but the marginal map $\ga \mapsto (\ga^{(k)})_{k=1}^N$ and is Poisson morphism, facts shown in \cref{prop:marPo}.

\subsection{The Lie algebra $\G_\infty$ and Lie-Poisson space $\G_\infty^*$}\label{ssec:Outinf}
Having built up the necessary structure at the $N$-particle level, we transition to addressing the infinite-particle limit of our constructions. The natural inclusion map $\G_N\subset\G_M$ for any integers $M\geq N$ implies that one has the limiting topological vector space (a co-limit of topological spaces ordered by inclusion)
\begin{equation}\label{eq:Ginfdef}
\G_\infty \coloneqq \bigoplus_{k=1}^\infty \g_k.
\end{equation}
Elements of $\G_\infty$ are called \emph{observable $\infty$-hierarchies}, alternatively \emph{$\infty$-hierarchies of observables}. They take the form $F=(f^{(k)})_{k=1}^\infty$, where $f^{(k)} \in \g_k$ is the zero element for all $k\geq N+1$, for some $N\in\N$. Thus, given any $F,G\in\G_\infty$, by taking $N$ sufficiently large, it makes sense to consider the Lie bracket $\comm{F}{G}_{\G_N}$. Our next result computes the limit $\comm{F}{G}_{\G_\infty}$ of this expression as $N\rightarrow\infty$ and shows that $(\G_\infty, \comm{\cdot}{\cdot}_{\G_\infty})$ is indeed a Lie algebra. Notably, the Lie bracket $\comm{\cdot}{\cdot}_{\G_\infty}$ acquires a much simpler form than $\comm{\cdot}{\cdot}_{\G_N}$, as certain terms vanish as $N\rightarrow\infty$. In contrast to the quantum setting of \cite{MNPRS2020}, there are no technical difficulties involving compositions of distribution-valued operators to give meaning to $\comm{F}{G}_{\G_\infty}$. 

\begin{restatable}{theorem}{GinfLA}\label{prop:GinfLA}
Let $F = (f^{(\ell)})_{\ell=1}^\infty, G = (g^{(j)})_{j=1}^\infty \in \mathfrak{G}_{\infty}$. For each $k\in\N$, define
\begin{equation}\label{eq:GinfLB}
\comm{F}{G}_{\G_\infty}^{(k)} \coloneqq \lim_{N\rightarrow \infty} \comm{F}{G}_{\mathfrak{G}_N}^{(k)} = \sum_{\substack{\ell,j\geq 1 \\ \ell+j-1=k}} \mathrm{Sym}_k\Big(f^{(\ell)}\wedge_1 g^{(j)}\Big),
\end{equation}
where the limit is in the topology of $\mathfrak{G}_\infty$, the wedge product $\wedge_1$ is defined by
\begin{equation}
f^{(\ell)}\wedge_1 g^{(j)}(\uz_k) \coloneqq \ell j \paren*{\nabla_{x_1} f^{(\ell)}(\uz_\ell)\cdot\nabla_{v_1}g^{(j)}(z_1,\uz_{\ell+1;k}) - \nabla_{x_1} g^{(j)}(\uz_j)\cdot\nabla_{v_1} f^{(\ell)}(z_1,\uz_{j+1;k})},
\end{equation}
and the $k$-particle symmetrization operator $\Sym_k$ is defined by\footnote{By duality, $\Sym_k$ is also well-defined for distributions on $\Rd^k$.}
\begin{equation}\label{eq:Symdef}
\forall h^{(k)}\in \Cc^\infty(\Rd^k), \qquad \Sym_k(h^{(k)})\coloneqq \frac{1}{k!}\sum_{\pi\in\Ss_k} h^{(k)}\circ\pi.
\end{equation}
Moreover, $(\G_\infty,\comm{\cdot}{\cdot}_{\G_\infty})$ is a Lie algebra in the sense of \cref{def:LA}, and the bracket $\comm{\cdot}{\cdot}_{\G_\infty}$ is boundedly hypocontinuous.
\end{restatable}

As with the $N$-particle setting, the next step is the dual problem of constructing a Lie-Poisson space from $(\G_\infty,\comm{\cdot}{\cdot}_{\G_\infty})$. We define the real topological vector space
\begin{equation}\label{eq:Ginf*def}
\G_\infty^* \coloneqq \prod_{k=1}^\infty \g_k^*
\end{equation}
equipped with the usual product topology, which is the strong dual of $\G_\infty$. Elements of $\G_\infty^*$ are called \emph{state $\infty$-hierarchies}, alternatively \emph{$\infty$-hierarchies of states}. 

We want to construct a Lie-Poisson bracket over $\G_\infty^*$ similarly to as done in \cref{prop:GNLP}. Given any $\Fc,\Gc\in \Cc^\infty(\G_\infty^*)$ and $\Ga=(\ga^{(k)})_{k=1}^\infty \in\G_\infty^*$, the continuous linear functionals $\ds\Fc[\Ga],\ds\Gc[\Ga]$ may be identified as elements of $\G_\infty$ since $\G_\infty^{**} \cong \G_\infty$. Hence, $\comm{\ds\Fc[\Ga]}{\ds\Gc[\Ga]}_{\G_\infty}$ is an element of $\G_\infty$, in particular only finitely many of its components are nonzero, and we are justified in defining
\begin{equation}\label{eq:Ginf*LPdef}
\pb{\Fc}{\Gc}_{\G_\infty^*}(\Ga) \coloneqq \ipp{\comm{\ds\Fc[\Ga]}{\ds\Gc[\Ga]}_{\G_\infty}, \Ga}_{\G_\infty-\G_\infty^*} = \sum_{k=1}^\infty \ipp*{\comm{\ds\Fc[\Ga]}{\ds\Gc[\Ga]}_{\G_\infty}^{(k)},\ga^{(k)}}_{\g_k-\g_k^*}.
\end{equation}
Here, we come to one of the main technical difficulties of the paper: we are unable to prove that $\pb{\Fc}{\Gc}_{\G_\infty^*}\in \Cc^\infty(\G_\infty^*)$. An essentially equivalent issue is that while we are able to show that a Hamiltonian vector field $X_{\Gc}$ exists, we are unable to show it is $\Cc^\infty$ as a map $\G_\infty^*\rightarrow\G_\infty^*$. As remarked in \cref{ssec:intropf}, we cannot rely on \cref{thm:Glockner}, as done in the proof of \cref{prop:GNLP}, because we are unable to verify that $\G_\infty^*$ satisfies certain topological conditions, namely that it is a $k^\infty$ space (see \cref{def:kinfty}). Accordingly, we instead directly show that $\G_\infty^*$ admits a weak Lie-Poisson structure in the sense of \cref{def:WPVS}.

The key difference between a weak Lie-Poisson structure and Lie-Poisson structure is that in the former, one specifies a unital subalgebra (with respect to pointwise product) $\A_\infty\subset\Cc^\infty(\G_\infty^*)$, which must satisfy certain nondegeneracy conditions, as the ``admissible'' functionals, in contrast to working with all the functionals in $\Cc^\infty(\G_\infty^*)$.  To this end, we choose $\A_{\infty}\subset \Cc^\infty(\G_\infty^*)$ to be the algebra generated with respect to pointwise product by the set
\begin{equation}\label{eq:Ainfdef}
\{\Fc\in \Cc^\infty(\G_\infty^*) : \Fc(\cdot) =\ipp{F,\cdot}_{\G_\infty-\G_\infty^*} \, \enspace F\in \G_\infty\} \cup \{\Fc\in \Cc^\infty(\G_\infty^*) : \Fc(\cdot)\equiv C\in\R\}.
\end{equation}
Heuristically viewing the components of $\Ga=(\ga^{(k)})_{k=1}^\infty$ as measures on $\Rd^k$, we call functionals of the form $\Fc(\cdot)=\ipp{F,\cdot}_{\G_\infty-\G_\infty^*}$ \emph{expectations}. They are analogous to the ``trace functionals'' of \cite{MNPRS2020}. In other words, the subalgebra $\A_{\infty}$ is generated by expectations and the constant functionals. The work \cite{MNPRS2020} employs the notion of a weak Poisson vector space at both the $N$-particle level for $\G_N^*$ and the infinite-particle level for $\G_\infty^*$; while here, we only need this notion at the infinite-particle level. This is an advantage of the present work compared to \cite{MNPRS2020}. The motivation for this choice of algebra $\A_\infty$ is that expectation functionals have constant G\^ateaux derivatives (see \cref{rem:cgd} below). Since for fixed expectations $\Fc,\Gc$, the G\^ateaux derivatives $\ds\Fc[\Ga],\ds\Gc[\Ga]$ have only finitely many nonzero components as elements in $\G_\infty$, \emph{uniformly} in $\Ga$, this allows us then to directly check that the bracket $\pb{\Fc}{\Gc}_{\G_\infty^*}$ is $\Cc^\infty$, in fact it belongs to the subalgebra $\A_\infty$, and also show that the the vector field $X_{\Gc}$ is $\Cc^\infty$. This direct verification relies heavily on explicit formulae for the Poisson bracket $\pb{\cdot}{\cdot}_{\G_\infty^*}$ and for the Hamiltonian vector field with respect to the bracket $\pb{\cdot}{\cdot}_{\G_\infty^*}$ to show that these expressions reduce to finite sums of compositions of $\Cc^\infty$ maps.

\begin{remark}
\label{rem:AH_can}
Our definition of $\A_{\infty}$ is not canonical in the sense that one could, in principle, include functionals beyond those generated by expectations and constants. However, doing so comes at the cost of added complexity in verifying that $(\G_\infty^*,\A_{\infty},\pb{\cdot}{\cdot}_{\G_\infty^*})$ is a weak Poisson vector space, and therefore we will not do so in this work.
\end{remark}

\begin{remark}\label{rem:cgd}
By the bilinearity of the duality pairing and the definition of the G\^ateaux derivative, an expectation functional $\Fc$ has constant G\^ateaux derivative, that is $\ds\Fc[\Ga] = \ds\Fc[0]$ for all $\Ga\in \G_\infty^*$. Similarly, a constant functional has zero G\^ateaux derivative.
\end{remark}

\begin{restatable}{theorem}{GinfLP}\label{prop:GinfLP}
Let $\mathfrak{G}_\infty^*$ be the strong dual of $\mathfrak{G}_\infty$ as given in \eqref{eq:Ginf*def}. Define the bracket
\begin{equation}\label{eq:GinfPBdef}
\pb{\Fc}{\Gc}_{\mathfrak{G}_\infty^*}(\Gamma) \coloneqq \ipp*{\comm{\ds\Fc[\Ga]}{\ds\Gc[\Ga]}_{\G_\infty}, \Ga}_{\G_\infty-\G_\infty^*}, \qquad \forall \Fc,\Gc \in \Cc^\infty(\mathfrak{G}_\infty^*), \ \Ga \in \G_\infty^*,
\end{equation}
and let $\A_\infty$ be as in \eqref{eq:Ainfdef}. Then the triple $(\G_\infty^*,\A_\infty,\pb{\cdot}{\cdot}_{\G_\infty^*})$ is a weak Poisson vector space in the sense of \cref{def:WPVS}.
\end{restatable}

Having constructed a weak Poisson vector space for the infinite-particle setting, it makes sense to discuss Hamiltonian flows for $\infty$-hierarchies. Our final result of this subsection is that the Vlasov hierarchy \eqref{eq:VlH} is itself Hamiltonian, which is a new observation. The \emph{Vlasov hierarchy Hamiltonian functional} is the expectation (cf. \eqref{eq:BBGKYham}, \eqref{eq:BBGKYW} for the BBGKY Hamiltonian)
\begin{equation}\label{eq:VlHham}
\Hvlh(\Ga) = \ipp{\W_{VlH}, \Ga}_{\G_\infty-\G_\infty^*}
\end{equation}
generated by the observable $\infty$-hierarchy
\begin{equation}\label{eq:VlHW}
\W_{VlH}\coloneqq \paren*{\frac{1}{2}|v|^2,W(x_1-x_2),0,\ldots}.
\end{equation}
One immediately recognizes that $\W_{VlH}$ is the $N\rightarrow\infty$ limit of $\W_{BBGKY}$ in the topology of $\G_\infty$. Interpreting the integrals as distributional pairings, we can write, for $\Ga=(\ga^{(k)})_{k=1}^\infty$,
\begin{equation}
\Hvlh(\Ga) =\frac{1}{2}\int_{\R^{2d}}d\ga^{(1)}(z)|v|^2 + \int_{(\R^{2d})^2}d\ga^{(2)}(z_1,z_2) W(x_1-x_2).
\end{equation}
In particular, the functional $\Hc_{VlH}$ belongs to the admissible algebra $\A_\infty$ introduced in \eqref{eq:Ainfdef}. The next theorem asserts that the Vlasov hierarchy \eqref{eq:VlH} is a Hamiltonian flow on $(\G_\infty^*,\A_\infty,\pb{\cdot}{\cdot}_{\G_\infty^*})$, and it is the classical analogue of \cite[Theorem 2.10]{MNPRS2020} for the Gross-Pitaevskii hierarchy.

\begin{restatable}{theorem}{Vlh}\label{thm:Vlh}
Let $I \subset \R$ be a compact interval. Then $\Ga=(\ga^{(k)})_{k=1}^\infty \in \mathcal{C}^\infty(I, \mathfrak{G}_\infty^*)$ is a solution to the Vlasov hierarchy \eqref{eq:VlH} if and only if 
\begin{equation}
\dot{\Ga}= X_{\Hvlh}(\Ga),
\end{equation}
where $X_{\Hvlh}$ is the unique Hamiltonian vector field generated by the Hamiltonian $\Hvlh$ with respect to the weak Lie-Poisson space $(\mathfrak{G}_\infty^*, \A_\infty,\pb{\cdot}{\cdot}_{\G_\infty^*})$. 
\end{restatable}

\subsection{From Vlasov hierarchy to Vlasov equation}\label{ssec:OutVl}
Finally, we tie together the constituent results of the previous subsections to connect the Hamiltonian structure of the Vlasov hierarchy \eqref{eq:VlH} to the Vlasov equation \eqref{eq:Vl}. This necessitates elaborating on the rigorous formulation of the Hamiltonian structure of the Vlasov equation (cf. \cite[p. 329, 10.1(e)]{MR2013}). The \emph{Vlasov Hamiltonian functional} is
\begin{equation}\label{eq:HVl}
\Hvl(\gamma) \coloneqq \ipp*{\frac{1}{2}|v|^2, \ga}_{\g_1-\g_1^*} + \ipp*{W(x_1-x_2), \ga^{\otimes 2}}_{\g_2-\g_2^*}.
\end{equation}
In terms of ``integrals'' (as before, understood rigorously as distributional pairings),
\begin{equation}
\Hvl(\gamma) = \frac{1}{2}\int_{(\R^d)^2}d\ga(x,v)|v|^2+\int_{(\R^d)^2}d\rho^{\otimes 2}(x_1,x_2)W(x_1-x_2),
\end{equation}
where $\rho \coloneqq \int_{\R^d}d\ga(\cdot,v)$ is the density associated to $\ga$. Note that $\rho$ is well-defined as a distribution, since for any test function $f\in \Cc^\infty(\R^d)$, we can set
\begin{equation}
\ipp{f,\rho}_{\Cc^\infty(\R^d)-\mathcal{E}'(\R^d)}\coloneqq \ipp{f\otimes 1,\ga}_{\Cc^\infty(\R^{2d})-\mathcal{E}'(\R^{2d})},
\end{equation}
where $(f\otimes 1)(x,v) = f(x)$ for every $(x,v)\in (\R^d)^2$.  In contrast to the other Hamiltonian functionals we have seen so far, $\Hvl$ is \emph{nonlinear}, in fact quadratic, in the potential energy. Since $\Hvl$ is multilinear in its argument $\ga$ and continuous as a map from $\g_1^*\rightarrow\R$, it is straightforward to check that $\Hvl\in\Cc^\infty(\g_1^*)$.

\begin{restatable}{proposition}{Vlas}\label{prop:Vlas}
Let $I \subset \R$ be a compact interval. Then $\ga\in \mathcal{C}^\infty(I, \g_1^*)$ is a solution to the Vlasov equation \eqref{eq:Vl} if and only if 
\begin{equation}
\dot{\ga}= X_{\Hvl}(\ga),
\end{equation}
where $X_{\Hvl}$ is the unique Hamiltonian vector field generated by the Hamiltonian $\Hvl$ with respect to the Lie-Poisson space $(\mathfrak{g}_1^*,\pb{\cdot}{\cdot}_{\g_1^*})$. 
\end{restatable}

We connect the Vlasov hierarchy to the Vlasov equation, each as infinite-dimensional Hamiltonian systems, through the embedding
\begin{equation}\label{eq:iotadef}
\iota: \g_1^*\rightarrow \G_\infty^*, \qquad \iota(\ga) \coloneqq (\ga^{\otimes k})_{k=1}^\infty, \qquad \forall \ga\in\g_1^*.
\end{equation}
Here, $\ga^{\otimes k}$ denotes the usual $k$-fold tensor product of the distribution $\ga$. The geometric content of the map $\iota$, which we call the \emph{trivial embedding} or \emph{factorization map}, is that it preserves the Poisson structures on $\g_1^*$ and $\G_\infty^*$, i.e. it is a Poisson morphism in the sense of \cref{def:PM}.

\begin{restatable}{theorem}{PM}\label{thm:PM}
The map $\iota \in \Cc^\infty(\g_1^*,\G_\infty^*)$ is a morphism of the Lie-Poisson space $(\g_1^*,\pb{\cdot}{\cdot}_{\g_1^*})$ into the weak Lie-Poisson space $(\G_\infty^*,\A_\infty,\pb{\cdot}{\cdot}_{\G_\infty^*})$:
\begin{equation}
\forall \Fc,\Gc\in \A_\infty, \qquad \pb{\iota^*\Fc}{\iota^*\Gc}_{\g_1^*} = \iota^*\pb{\Fc}{\Gc}_{\G_\infty^*}.
\end{equation}
\end{restatable}

Let us now explain why the results of this section constitute a rigorous derivation of the Hamiltonian structure for the Vlasov equation, as claimed in the title of the paper. The reader may check that (see also \cref{rem:VlHVl})
\begin{equation}\label{eq:VHhampback}
\iota^* \Hvlh = \Hvl,
\end{equation}
i.e. the pullback of the Vlasov hierarchy Hamiltonian equals the Vlasov Hamiltonian. The identity \eqref{eq:VHhampback} together with \Cref{thm:PM,thm:Vlh} then show that the Hamiltonian functional and Poisson bracket for the Vlasov equation are obtained via the pullback under the trivial embedding $\iota$ of the Hamiltonian functional and Poisson bracket for the Vlasov hierarchy; moreover, $\iota$ sends solutions of the Vlasov equation to special factorized solutions of the Vlasov hierarchy. Combined with the results of \cref{ssec:OutN}, which provide a geometric correspondence between Newton's equations/Liouville equation and the BBGKY hierarchy, and \cref{prop:GinfLA}, which allows us to take the infinite-particle limit of our $N$-particle geometric constructions, we arrive at a rigorous derivation of the Hamiltonian structure of the Vlasov equation directly from the Hamiltonian formulation of Newtonian mechanics.

Finally, as mentioned in \cref{ssec:intromot}, there is another way to derive the Vlasov equation from the Newtonian $N$-body problem \eqref{eq:New} via the empirical measure. It is an interesting fact, which to our knowledge has not been previously observed, that the map $\iota_{EM}$ assigning a position-velocity configuration $\uz_N \in \Rd^N$ to its empirical measure on $\R^{2d}$ is, in fact, a Poisson morphism (see \cref{prop:EMpm} below). Since one also has $\iota_{EM}^*\Hvl=\Hc_{New}$ (see \cref{rem:EMVlHam}), this implies the previously mentioned fact that if $\uz_N^t$ is a solution to \eqref{eq:New}, then the associated empirical measure $\mu_N^t$ is a weak solution to the Vlasov equation.

\begin{restatable}{proposition}{EMpm}\label{prop:EMpm}
The map
\begin{equation}\label{eq:iotaEMdef}
\iota_{EM}: \Rd^N \rightarrow \g_1^*, \qquad \iota_{EM}(\uz_N) \coloneqq \frac{1}{N}\sum_{i=1}^N \d_{z_i}, \qquad \forall \uz_N \in\Rd^N
\end{equation}
belongs to $\Cc^\infty(\Rd^N, \g_1^*)$ and defines a morphism of Poisson vector spaces. 
\end{restatable}

\subsection{Organization of paper}\label{ssec:Outorg}
Let us close \cref{sec:Out} with some comments on the organization of the remaining body of the article.

\cref{sec:pre} contains background material on topological vector spaces, Lie algebras, and (weak) Lie-Poisson vector spaces. The reader may wish to skip this section upon first reading and instead consult it as necessary during the reading of \Cref{sec:Ngeom,sec:infgeom,sec:Ham}.

\cref{sec:Ngeom} contains the $N$-particle setting results. The section is divided into several subsections, each building upon the previous one. \cref{ssec:NgeomLio} concerns the setting of the Newtonian system \eqref{eq:New} and Liouville equation \eqref{eq:Lio}, proving \cref{prop:gkLA,prop:gkLP} for $\g_k$ and $\g_k^*$, respectively, \cref{prop:LioPM} for $\iota_{Lio}$, and \cref{prop:EMpm} for $\iota_{EM}$. \Cref{ssec:NgeomLie,ssec:NgeomLP} concern the setting of the BBGKY hierarchy \eqref{eq:BBGKY}, proving \Cref{prop:GNLA,prop:GNLP} for $\G_N,\G_N^*$, respectively. Finally, \cref{ssec:Ngeommor} concerns the operation of taking marginals, proving \cref{prop:marPo} for $\iota_{mar}$. 

\cref{sec:infgeom} contains the infinite-particle setting results. As with \cref{sec:Ngeom}, the section is divided into several subsections, each intended to build upon the previous one. \Cref{ssec:GinfLA,ssec:GinfLP} are devoted  to the proofs of \Cref{prop:GinfLA,prop:GinfLP} for $\G_\infty,\G_\infty^*$, respectively. \cref{ssec:infgeomPM} contains the proof of \cref{thm:PM} for the map $\iota$.

Lastly, \cref{sec:Ham} contains the proofs of the Hamiltonian flows results \cref{prop:Vlas} and \Cref{thm:bbgky,thm:Vlh}, which assert that that the Vlasov equation, BBGKY hierarchy, and Vlasov hierarchy, respectively, are Hamiltonian flows on their respective Lie-Poisson spaces given by \cref{prop:Vlas} and \Cref{prop:GNLP,prop:GinfLP}. The section is broken into three subsections with \cref{ssec:HamVl} corresponding to the Vlasov equation, \cref{ssec:HamBBGKY} to the BBGKY hierarchy, and \cref{ssec:HamVH} to the Vlasov hierarchy.

\section{Background material}\label{sec:pre}
The purpose of this section is to collect in one place all the necessary preliminary facts---some rather elementary---from functional analysis concerning topological vector spaces, function spaces and distributions, and Lie algebras and Lie-Poisson vector spaces. There is some overlap with \cite[Section 4, Appendices A-B]{MNPRS2020}, but this section also contains notions new to the present work, such as Gl\"ockner's aforementioned formalism of Poisson vector spaces. Moreover, our spaces of functions and distributions are not comparable to \cite{MNPRS2020}, as here we deal with test functions and distributions over $(\R^{2d})^k$, as opposed to operators between spaces of test functions and spaces of distributions. This difference is, of course, a reflection of the classical physics setting of the present work in contrast to the quantum setting of the cited work, as explained in \cref{ssec:intropf}.

\subsection{Some function analysis facts}
In this subsection, we review functional analytic notions which will be used throughout the rest of the paper. We begin by reviewing duality in topological vector spaces.

\begin{definition}
Let $X$ be a topological vector space. We define $X^*$ to be the set of continuous linear functionals on $X$, and endow it with the \emph{strong dual topology}, which is given as follows. Let $\mathcal{A}$ be the set of bounded subsets of $X$. For each $A \in \mathcal{A}$, we define the semi-norm 
\begin{equation}
    \rho_{A} : X^* \rightarrow [0,\infty), \qquad \rho_{A} (T) \coloneqq \sup_{f \in A} | T(f)|.
\end{equation}
Note that this is indeed a semi-norm, since continuous linear operators are bounded. We define the topology of $X^*$ to be the one generated by the above semi-norms. If the cannonical embedding 
\begin{equation} 
X \hookrightarrow (X^*)^* \eqqcolon X^{**}
\end{equation}
is an isomorphism between topological vector spaces, then we say that $X$ is \emph{reflexive}. 
\end{definition}

\begin{definition}
Let $X, Y$ be topological vector spaces, and let $F: X \rightarrow Y$ be a continuous linear map. We define the adjoint of $F$ to be $F^* : Y^* \rightarrow X^*$ with
\begin{equation}
    F^*(T)(x) \coloneqq (T \circ F)(x), \qquad \forall T \in Y^*, \ x\in X. 
\end{equation}
\end{definition}
\begin{proposition}\label{prop:contadj}
Let $X, Y$ be topological vector spaces, and let $F: X \rightarrow Y$ be a continuous linear map. Then $F^*: Y^* \rightarrow X^*$ is a continuous linear map.
\end{proposition}

We continue with the necessary background in functional analysis by reviewing the concepts of barrelled, Montel, and (DF) spaces following the presentation of \cite{Kothe1969I,Kothe1979II}.

\begin{definition}[Barrelled space]\label{def:barsp}
Let $X$ be a locally convex topological vector space. We say that $X$ is \emph{barrelled} if every closed absorbent, absolutely convex subset of $X$ is a neighborhood of $0 \in X$.
\end{definition}

In the above definition, a subset $M$ of $X$ is said to be \emph{absorbent} if for every $x \in X$, there exists a $\rho > 0$ such that $x \in \rho M$; it is said to be \emph{absolutely convex} if for every $x,y \in M$ and $\alpha,\be\in\R$ with $|\alpha| + |\beta| \leq 1$, the point $\alpha x + \beta y \in M$. For the following, we recall that a locally convex topological vector space is Fr\'echet if it is metrizable and complete. 

\begin{lemma}\label{lem:frbar}
Fr\'echet spaces are barrelled.
\end{lemma}
\begin{proof}
See \cite[\S 21.6 (3)]{Kothe1969I}.
\end{proof}

\begin{definition}[Montel space]\label{def:Monsp}
We say $X$ is a \emph{Montel space} if it is barrelled and every bounded subset of $X$ is relatively compact. 
\end{definition}

\begin{lemma}\label{lem:Monrfl} Montel spaces are reflexive, and the strong dual of a Montel space is Montel.
\end{lemma}
\begin{proof}
See \cite[\S 27.2 (1)-(2)]{Kothe1969I}.
\end{proof}

\begin{definition}[(DF) Spaces]\label{def:DFsp}
Let $X$ be a locally convex topological vector space. We say that $X$ is a \emph{dual Frech\'{e}t (DF) space} if the following conditions hold: 
\begin{enumerate}[(i)]
    \item The space $X$ has a fundamental sequence of bounded sets, i.e. there exists a countable sequence of bounded sets $\{ B_i \}_{i \in \N}$ such that any bounded set in $X$ is contained in some $B_i$.
    \item Every bounded subset of $X^*$ (in the strong topology) which is the countable union of equicontinuous sets is equicontinuous. 
\end{enumerate}
\end{definition}

\begin{lemma}\label{frechet dual is df}
The strong dual of a Fr\'echet space is a (DF) space.
\end{lemma}
\begin{proof}
See \cite[\S 29.3]{Kothe1969I}. 
\end{proof}

Next, we recall the notions of sequential spaces and $k$-spaces as presented in \cite[pp. 53, 152]{Engelking1989}.

\begin{definition}[Sequential Spaces]
Let $(X,\tau)$ be a topological space. We say a set $S \subset X$ is \emph{sequentially closed} if for any sequence $(x_i)_{i=1}^\infty$ in $S$ that converges to $x$ implies that $x \in S$. We say the space $X$ is a sequential space if every sequentially closed set is closed in $X$. 
\end{definition}

\begin{definition}[$k$-space]\label{def:ksp}
Let $(X,\tau)$ be a topological space. We say $X$ is a $k$-space if the following condition holds: for every set $A \subset X$, $K \cap A$ is closed in $A$, endowed with the subspace topology, for every compact $K$ if and only if $A$ is closed in $X$. 
\end{definition}

That a sequential space is \emph{a fortiori} a k-space is, perhaps well known. For the sake of completeness, we present a proof of this fact in the next proposition, which will be crucially used in \cref{ssec:NgeomLP}.

\begin{proposition}[Sequential $\Rightarrow$ $k$-space]\label{prop:seqk}
Let $(X,\tau)$ be a sequential space. Then $(X,\tau)$ is a $k$-space.
\end{proposition}
\begin{proof}
Assume that there exists a non-closed set $A \subset X$ which satisfies $K \cap A$ is closed in $A$ for every compact $K$. Since $A$ is not closed, and $X$ is a sequential space, we must have that $A$ is not sequentially closed. So, there exists some sequence $(x_i)_{i=1}^\infty$ in $A$ that converges to a point $x \in X \setminus A$. Note that the set $ \{ x_i : i \in \N \} \cup \{x\}$ is compact, and so 
\begin{equation}
A \cap (\{ x_i : i \in \N \} \cup \{x\}) =\{ x_i : i \in \N \}
\end{equation}
is closed in $A$. However, since closed sets are sequentially closed and the sequence $x_i$ converges to $x$, it must be the case that $x \in A$. This is a contradiction, so $X$ is a $k$-space. 
\end{proof}

We are now ready to state a result of Webb \cite{Webb1968} which gives sufficient conditions for a topological vector space to be a sequential space. 

\begin{theorem}[{\cite[Proposition 5.7]{Webb1968}}]\label{thm:Webb}
Let $X$ be an infinite-dimensional Montel (DF) space. Then $X$ is a sequential space. 
\end{theorem}

We close this subsection by stating the notions of derivative and smooth function for infinite-dimensional spaces used in this work, which is that of the G\^{a}teaux derivative. For more on calculus in the setting of topological vector spaces, we refer to the lecture notes of Milnor \cite{Milnor1984}.

\begin{definition}[G\^{a}teaux derivative]
Let $X,Y$ be topological vector spaces and let $f: X \rightarrow Y$. 
\begin{enumerate}
    \item The function $f$ is called $\mathcal{C}^0(X,Y)$ if it is continuous. 
    \item The function $f$ is called $\mathcal{C}^1(X,Y)$ if for every $x,x' \in X$, the limit 
\begin{equation}\label{eq:gddef}
    \ds f[x](x') \coloneqq \lim_{h \rightarrow 0}\frac{1}{h} \left[f(x + h x') - f(x)\right]
\end{equation}
exists in $Y$, and the mapping $\ds f: X \times X\rightarrow Y$ is continuous with respect to the product topology. The function $\ds f$ is called the G\^{a}teaux derivative of $f$. 
\item For $n\in\N$, the function $f$ is called $\mathcal{C}^n(X,Y)$ if $\ds^n f: X\times X^n  \rightarrow Y $ exists and is continuous. 
\item The function $f$ is called $\mathcal{C}^\infty(X,Y)$ if it is $\mathcal{C}^n(X,Y)$ for every $n \in \N$. 
\end{enumerate}
\end{definition}

In the remainder of the paper, we write $\mathcal{C}(X)$ (similarly, $\mathcal{C}^n(X),\mathcal{C}^\infty(X)$) when the codomain is $\R$, i.e. the maps are real-valued functionals.

\subsection{Lie algebras and Poisson vector spaces}\label{ssec:preLALP}
We start this subsection by giving a precise definition of Lie algebra and Poisson vector space that we use in this paper. With these definitions in hand, we then present a result due to Gl\"ockner \cite{Glockner2009} which allows one to canonically construct a Lie-Poisson vector space from a Lie algebra, assuming certain topological conditions are met, as mentioned in \cref{ssec:intropf}. The use of Gl\"ockner's machinery is new to the present work compared to \cite{MNPRS2020}. 

\begin{definition}[Lie algebra]\label{def:LA}
Let $\mathfrak{g}$ be a locally convex topological vector space over $\R$, and $\comm{\cdot}{\cdot}_{\mathfrak{g}} : \mathfrak{g} \times \mathfrak{g} \rightarrow \mathfrak{g}$. We say the pair $(\mathfrak{g}, \comm{\cdot}{\cdot}_{\mathfrak{g}})$ is a Lie algebra if the following conditions hold:
\begin{enumerate}[(L1)]
\item\label{LA1} The bracket $\comm{\cdot}{\cdot}_{\mathfrak{g}}$ is bilinear. 
    \item\label{LA2} For all $x,y\in \mathfrak{g}$, $\comm{x}{y}_{\mathfrak{g}} = -\comm{y}{x}_{\mathfrak{g}}$. 
    \item\label{LA3} For all $x,y,z \in \mathfrak{g}$, the Jacobi identity is satisfied:
    \begin{equation}
        \comm{x}{\comm{y}{z}_{\mathfrak{g}}}_{\mathfrak{g}} +  \comm{y}{\comm{z}{x}_{\mathfrak{g}}}_{\mathfrak{g}} + \comm{z}{\comm{x}{y}_{\mathfrak{g}}}_{\mathfrak{g}} =0.
    \end{equation}
\end{enumerate}
\end{definition}

\begin{remark}
Note that in this work, a continuity requirement is not assumed in \cref{def:LA}. This definition is consistent with the standard algebraic definition of a Lie algebra. In practice, all of our Lie brackets will be at a minimum separately continuous.
\end{remark}

The next definition introduces the notion of a possibly infinite-dimensional Poisson vector space, which is a natural extension of the finite-dimensional notion of a Poisson vector space, more generally Poisson manifold (e.g., see \cite{Weinstein1998}). Our usage is consistent with that of Gl\"ockner \cite[Definition 4.2]{Glockner2009}. For other possible notions of a infinite-dimensional Poisson vector spaces, which are not appropriate for our purposes due to being restricted to the Banach category, we refer to \cite{OR2003,OR2004}.

\begin{definition}[Poisson vector space]\label{def:PVS}
Let $X$ be a locally convex topological vector space, and 
\begin{equation} 
\pb{\cdot}{\cdot} : \mathcal{C}^\infty(X) \times \mathcal{C}^\infty(X) \rightarrow \mathcal{C}^\infty(X)
\end{equation} 
be a bilinear map. We say the pair $(X, \pb{\cdot}{\cdot})$ is a Poisson vector space if it satisfies the following properties:
\begin{enumerate}[(PVS1)]
\item $(\mathcal{C}^\infty(X),\pb{\cdot}{\cdot})$ is a Lie algebra in the sense of \cref{def:LA} obeying the Leibniz rule:
\begin{equation}\label{eq:LR}
\forall \mathcal{F},\mathcal{G},\mathcal{H} \in \mathcal{C}^\infty(X),\qquad  \pb{\mathcal{F}}{\mathcal{G}\mathcal{H}} = \mathcal{H}\pb{\mathcal{F}}{\mathcal{G}}  + \mathcal{G}\pb{\mathcal{F}}{\mathcal{H}}.
\end{equation}
\item For every $\mathcal{F} \in \mathcal{C}^\infty( X)$, there exists a smooth Hamiltonian vector field $X_\mathcal{F} : X \rightarrow X$ such that
\begin{equation}
\forall\Gc\in\Cc^\infty(X),\qquad    X_\mathcal{F}\mathcal{G} = \pb{\mathcal{G}}{\mathcal{F}}.
\end{equation}
\end{enumerate}
\end{definition}

We now state the theorem of Gl\"ockner \cite[Theorem 4.10]{Glockner2009} which will allow us to construct a Poisson vector space from a given Lie algebra in the $N$-particle setting (see \cref{ssec:NgeomLie}). For the purposes of this paper, the reader may view this theorem as a ``black box.'' But to use this black box,  certain topological conditions need to be satisfied. Namely, \cite{Glockner2009} works in the context of $k^\infty$ spaces, a class of topological vector spaces introduced in that paper. Accordingly, we shall start this portion of the exposition by recalling the definition of this class of spaces, as well as the notion of reflexive locally convex Poisson vector spaces which we also need.

\begin{definition}[$k^\infty$-spaces]\label{def:kinfty}
Let $(X,\tau)$ be a topological space. We say that $X$ is a $k^\infty$ space if, for every $n \in \N$, the space $X^n$ endowed with the product topology is a $k$-space (recall \cref{def:ksp}).
\end{definition}

\begin{definition}[Reflexive locally convex Poisson space] \label{def:rlcps}
A \emph{reflexive locally convex Poisson space} is a reflexive locally convex $k^\infty$ space $E$, together with a hypocontinuous\footnote{The condition of hypocontinuity is a weaker condition than continuity, but stronger condition than separate continuity. Here, hypocontinuity is always defined with respect to the set $\mathscr{A}$ of bounded subsets of $E$. See \cite[p. 155]{Kothe1979II} for a precise definition of hypocontinuity.} map $\comm{\cdot}{\cdot} : E ^* \times E^* \rightarrow E^*$ which makes $(E^*, \comm{\cdot}{\cdot})$ into a Lie algebra in the sense of \cref{def:LA}. 
\end{definition}

Equipped with \Cref{def:kinfty,def:rlcps}, we are now prepared to state the following result from \cite{Glockner2009}, the statement of which has been tailored to our setting.

\begin{theorem}[{\cite[Theorem 4.10]{Glockner2009}}]\label{thm:Glockner}
Let $E$ be a reflexive locally convex Poisson space in the sense of \cref{def:rlcps} such that its dual $E^*$ is equipped with a hypocontinuous bracket $\comm{\cdot}{\cdot}: E^* \times E^* \rightarrow E^*$. For $\mathcal{F}, \mathcal{G} \in \mathcal{C}^\infty( E)$, the \emph{Lie-Poisson bracket} $\pb{\mathcal{F}}{\mathcal{G}} : E \rightarrow \R$ is defined by the expression 
\begin{equation}\label{eq:LPdef}
\pb{\mathcal{F}}{\mathcal{G}}(\Gamma) \coloneqq \ipp{\comm{\ds\mathcal{F}[\Gamma]}{\ds\mathcal{G}[\Gamma]}, \Gamma}_{E^* - E}, \qquad \forall \Gamma \in E,
\end{equation}
where $\ipp{\cdot,\cdot}_{E^*-E}$ denotes the duality pairing. The pair $(E, \pb{\cdot}{\cdot})$, called a \emph{Lie-Poisson space}, is a Poisson vector space in the sense of \cref{def:PVS}.
\end{theorem}

\begin{remark}
The work \cite{Glockner2009} does not specifically use the term \emph{Lie-Poisson space}; however, we feel this bit of terminology is appropriate to emphasize that the bracket as defined in \eqref{eq:LPdef} is a Lie-Poisson construction, while in general a Poisson bracket---and therefore Poisson vector space---need not be of Lie-Poisson type. 
\end{remark}

\begin{remark}
For our purposes, we will apply \cref{thm:Glockner} with $E = \g_k^*, \mathfrak{G}_N^*$ (defined in \eqref{eq:gk*def} and \eqref{eq:GN*def}, respectively), which requires our proving that $\g_k^*,\mathfrak{G}_N^*$ satisfy the assumptions of the theorem. This will be shown in \Cref{ssec:NgeomLio,ssec:NgeomLP}. Note that since $\g_k^*,\G_N^*$ are reflexive, $E^*$ is identifiable with the Lie algebras $\g_k,\G_N$, respectively.
\end{remark}

The space $\mathfrak{G}_\infty^*$ is a non-trivial countably infinite product of (DF) spaces, and as such is not a (DF) space itself (see \cite[p. 196]{SW1999}). Therefore, \cref{thm:Webb} is not applicable,  which renders verification of the assumptions of \cref{thm:Glockner} out of reach. To overcome this obstacle, we need a weaker notion of a Poisson vector space than assumed in \cref{thm:Glockner}. Namely, we need to restrict to a proper subalgebra $\mathcal{A}$ of functionals in $\mathcal{C}^\infty(E)$ for which smoothness of the Poisson bracket and Hamiltonian vector field can be verified. To this end, we use, as in the previous work \cite{MNPRS2020}, the framework of weak Poisson vector spaces due to Neeb \emph{et al.} \cite{NST2014}. 

\begin{definition}[Weak Poisson vector space]\label{def:WPVS}
Let $X$ be a locally convex topological vector space, and let $\mathcal{A} \subset \mathcal{C}^\infty( X)$ be an unital subalgebra. We say the triple $(X, \mathcal{A}, \pb{\cdot}{\cdot})$ is a \emph{weak Poisson vector space} if the following properties hold:
\begin{enumerate}[(WP1)]
    \item\label{assWP1}
    The pair $(\mathcal{A} , \pb{\cdot}{\cdot})$ is a Lie algebra in the sense of \cref{def:LA} obeying the Leibniz rule:
\begin{equation}
    \forall \Fc,\Gc,\Hc\in\A, \qquad \pb{\mathcal{F}}{\mathcal{G}\mathcal{H}} = \mathcal{H}\pb{\mathcal{F}}{\mathcal{G}}  + \mathcal{G}\pb{\mathcal{F}}{\mathcal{H}}.
\end{equation}
\item\label{assWP2}
For each $x,v\in X$, if $\ds\mathcal{F}[x](v) = 0$ for every $\mathcal{F}\in \mathcal{A}$, then $v = 0$. 
\item\label{assWP3}
For every $\mathcal{F} \in \mathcal{A}$, there exists a $\Cc^\infty$ Hamiltonian vector field $X_\mathcal{F} : X \rightarrow X$ such that
\begin{equation}\label{eq:HamVFuniq}
\forall \Gc\in\A,\qquad   X_\mathcal{F}\mathcal{G} = \pb{\mathcal{G}}{\mathcal{F}}.
\end{equation}
\end{enumerate}
\end{definition}

\begin{remark}
When the Poisson bracket $\pb{\cdot}{\cdot}$ in \cref{def:WPVS} is of Lie-Poisson type, as in \eqref{eq:LPdef}, we shall use the terminology \emph{weak Lie-Poisson space}.
\end{remark}

\begin{remark}
As alluded to in the paragraph preceding \cref{def:WPVS}, a Poisson vector space in the sense of \cref{def:PVS} is \emph{a fortiori} a weak Poisson vector space.
\end{remark}

\begin{remark}\label{rem:HamVFuniq}
The property \eqref{eq:HamVFuniq} uniquely characterizes the Hamiltonian vector field. Indeed, if $X_{\Fc},\tl{X}_{\Fc}$ are two $\Cc^\infty$ vector fields obeying \eqref{eq:HamVFuniq}, then given $x\in X$,
\begin{equation}
\forall \Gc\in\A, \qquad \paren*{\paren*{X_{\Fc}-\tl{X}_{\Fc}}(\Gc)}(x) = \ds\Gc[x]\paren*{X_{\Fc}(x)-\tl{X}_{\Fc}(x)}.
\end{equation}
Applying \ref{assWP2} with $v= X_{\Fc}(x)-\tl{X}_{\Fc}(x)$, we conclude that $X_{\Fc}(x)=\tl{X}_{\Fc}(x)$.
\end{remark}

Finally, we need the notion of a morphism between (weak) Poisson vector spaces.

\begin{definition}\label{def:PM}
Let $(E_1,\pb{\cdot}{\cdot}_{E_1}), (E_2,\pb{\cdot}{\cdot}_{E_2})$ be Poisson vector spaces in the sense of \cref{def:PVS}. We say that a $\Cc^\infty$ map $T: E_1\rightarrow E_2$ is a \emph{morphism of Poisson vector spaces} if
\begin{equation}\label{eq:PM}
\forall \Fc,\Gc\in\Cc^\infty(E_2), \qquad \pb{T^*\Fc}{T^*\Gc}_{E_1} = T^*\pb{\Fc}{\Gc}_{E_2},
\end{equation}
where $T^*$ denotes the pullback under $T$. Suppose now that $(E_1,\A_1,\pb{\cdot}{\cdot}_{E_1}), (E_2,\A_2,\pb{\cdot}{\cdot}_{E_2})$ are weak Poisson vector spaces in the sense of \cref{def:WPVS}. We say that a $\Cc^\infty$ map $T:E_1\rightarrow E_2$ is a \emph{morphism of weak Poisson vector spaces} if for any $\Fc,\Gc\in \A_2$, $T^*\Fc, T^*\Gc\in \A_1$ and \eqref{eq:PM} holds with $\Cc^\infty(E_2)$ replaced by $\A_2$.
\end{definition}

\section{$N$-particle geometric structure}\label{sec:Ngeom}
In this section, we present the proofs of the results stated in \cref{ssec:OutNew,ssec:OutN}.

\subsection{$N$-particle Newton/Liouville equations}\label{ssec:NgeomLio}
The goal of this subsection is to establish the Hamiltonian structure of the Newtonian system \eqref{eq:New} and the Liouville equation \eqref{eq:Lio}, as well as to connect the two structures through a Poisson morphism. Since the Newtonian system is classical, we leave the proofs of the statements concerning it in \cref{ssec:OutNew} as simple exercises for the reader.

We recall from \eqref{eq:gkdef} and \eqref{eq:gkLBdef} the definitions of the space $\g_k$ and the bracket $\comm{\cdot}{\cdot}_{\g_k}$. Our first task is to prove \cref{prop:gkLA} asserting that $(\g_k,\comm{\cdot}{\cdot}_{\g_k})$ is a Lie algebra.

\begin{proof}[Proof of \cref{prop:gkLA}]
Since our bracket is defined as a scalar multiple of the standard Poisson bracket, the algebraic properties \ref{LA1}-\ref{LA3} are satisfied, so it only remains to check continuity. It suffices to show that the multiplication and differentiation maps
\begin{equation}
    M: \mathfrak{g}_k \times \mathfrak{g}_k \rightarrow \mathfrak{g}_k, \qquad (f,g) \mapsto fg
\end{equation}
and
\begin{equation}
    \partial^\alpha: \mathfrak{g}_k \rightarrow \mathfrak{g}_k, \qquad f \mapsto \partial^\alpha f
\end{equation}
are continuous, for each multi-index $\alpha \in \N^{2dk}$. But this follows because $\comm{\cdot}{\cdot}_{\g_k}$ is just a linear combination of compositions of $M,\partial^\alpha$. 

We first show that $M$ is continuous. Since the spaces $\mathfrak{g}_k$ are Fr\'echet, it suffices to show that $M$ is sequentially continuous. To this end, let $(f_j, g_j) \rightarrow (f,g)$ be a convergent sequence in $\mathfrak{g}_k \times \mathfrak{g}_k$. Note that for any compact set $K \subset \Rd^k$ and $n \in \N$, we have 
\begin{equation}\label{bounded sequence}
\max\paren*{\sup_{j\in \N} \rho_{K,n} (f_j), \sup_{j\in \N} \rho_{K,n} (g_j)} \leq C_{K,n},
\end{equation}
where the constant $C_{K,n}$ depends only on the set $K$ and the index $n$. Using the Leibniz rule and triangle inequality, we now estimate 
\begin{align}
    \rho_{K,n} (f_jg_j - fg)  & \leq \rho_{K,n} ( f_j(g_j-g)) + \rho_{K,n}( (f_j -f)g ) \nn\\
    & = \sup_{|\alpha|\leq n} \Vert \partial^\alpha ( f_j(g_j-g))\Vert_{L^\infty(K)} +\sup_{|\alpha|\leq n} \Vert \partial^\alpha( (f_j -f)g )\Vert_{L^\infty(K)} \nn\\
    & \leq \sup_{|\alpha|\leq n} \sum_{\beta \leq \alpha} \binom{\alpha}{\beta} \Vert \partial^\beta  f_j\Vert_{L^\infty(K)} \Vert \partial^{\alpha - \beta} (g_j-g)\Vert_{L^\infty(K)} \nn\\
    &\phantom{=}+\sup_{|\alpha|\leq n} \sum_{\beta \leq \alpha} \binom{\alpha}{\beta} \Vert \partial^\beta(f_j -f)\Vert_{L^\infty(K)} \Vert \partial^{\alpha -\beta} g \Vert_{L^\infty(K)} \nn\\
    & \leq C_{d,n} \left(\rho_{K,n}(f_j) \rho_{K,n}(g_j - g) +  \rho_{K,n}(f_j-f) \rho_{K,n}(g) \right)\nn\\
    & \leq C_{d,n,K} \paren*{\rho_{K,n}(g_j - g) + \rho_{K,n}(f_j - f)},
\end{align}
where in the last line we have used the bound \eqref{bounded sequence}. Since the last line converges to $0$ as $j \rightarrow \infty$ and $K,n$ were arbitrary, we have shown that $f_j g_j \rightarrow fg$ in $\mathfrak{g}_k$. Thus, $M$ is continuous.

Fix a multi-index $\alpha$. To show that $\partial^\alpha$ is continuous, let $f_j \rightarrow f$ in $\mathfrak{g}_k$ and calculate 
\begin{equation}
    \rho_{K,n}(\partial^\alpha (f_j - f)) = \sup_{|\gamma|\leq n} \Vert \partial^\gamma \partial^\alpha (f_j - f) \Vert_{L^\infty(K)} \leq \sup_{|\beta| \leq n + |\alpha|} \Vert \partial^\beta( f_j -f) \Vert_{L^\infty(k)} = \rho_{K,n+|\alpha|} (f_j -f).
\end{equation}
The right-hand side converges to $0$ as $j \rightarrow \infty$, which shows that the operator $\partial^\alpha$ is continuous. 
\end{proof}

Now recall the definitions of the space $\g_k^*$ and the bracket $\pb{\cdot}{\cdot}_{\g_k^*}$ from \eqref{eq:gk*def} and \eqref{eq:gk*LPdef}, respectively. Our next task is to prove \cref{prop:gkLP}, asserting that $(\g_k^*,\pb{\cdot}{\cdot}_{\g_k^*})$ is a Lie-Poisson space. To this end, we need the following technical lemma alluded to in \cref{ssec:OutNew}.

\begin{lemma}\label{lem:gkremon}
For each $k \in \N$, the space $\mathfrak{g}_k$ is reflexive, and $\mathfrak{g}_k^*$ is a (DF) Montel space.
\end{lemma}
\begin{proof}
We first prove that $\mathfrak{g}_k$ is Montel. The proof is an adaptation to symmetric functions of the argument that the space $\mathcal{C}^\infty( \R^n)$, for any $n\in\N$, is a Montel space (see \cite[Theorem VII, \S 2, Chapter 3]{Schwartz1966}). We reproduce it here for the reader's convenience. 

First, we fix an equivalent sequence of semi-norms on $\mathfrak{g}_k$ which give the same topology. Let $K_n$ be a compact exhaustion of $\Rd^k$, i.e. let $\{K_n\}_{n=1}^\infty$ be an increasing sequence of compact sets such that $\bigcup_{n=1}^\infty K_n = \Rd^k$. Then define the semi-norms 
\begin{equation}
\tilde{\rho}_n(f) \coloneqq \sup_{ |\alpha|\leq n} \Vert \partial^\alpha f \Vert_{L^\infty(K_n)}.
\end{equation}
These semi-norms are equivalent to those given in \eqref{seminorm def}, as the reader may check. This implies that $\mathfrak{g}_k$ is indeed a Fr\'echet space and hence a barrelled space by \cref{lem:frbar}. We will now show that $\mathfrak{g}_k$ satisfies the Heine-Borel property, i.e. that bounded closed subsets are compact. Note that since $\mathfrak{g}_k$ is a metric space, it suffices to show that bounded closed sets are sequentially compact. 
To this end, let $B \subset \mathfrak{g}_k$ be a bounded, closed set, and let $\{f_k \}_{k=1}^\infty \subset B$. Then by definition of bounded, there exist constants $C_n > 0$ such that 
\begin{equation}\label{bounded set in gk}
   \sup_{k\in\N} \Vert \partial^\alpha f_k \Vert_{L^\infty(K_n)} \leq C_n, \quad \forall |\alpha |\leq n.
\end{equation}
We will be using the convention that  subsequences of $\{f_k\}$ are still denoted by $\{f_k\}$. We take subsequences and diagonalize in the following way:
\begin{enumerate}
    \item Apply Arzel\`a–Ascoli and diagonalize with respect to $K_n$ to get a subsequence $f_k \rightarrow f$ locally uniformly in $L^\infty$.
    \item Apply Arzel\`a–Ascoli again and diagonalize with respect to each $|\alpha| \leq n$ to get a further subsequence $\partial^\alpha f_k \rightarrow \partial^\alpha f$ locally uniformly in $L^\infty$.
    \item Now we can conclude that $f_k \rightarrow f$ in $\mathfrak{g}_k$. Hence, since $B$ is closed in $\mathfrak{g}_k$, we have proved that $f \in B$.
\end{enumerate}
Hence, the space $\mathfrak{g}_k$ satisfies the Heine-Borel property. Since we noted above that $\mathfrak{g}_k$ is also barrelled, we conclude that it is a Montel space (recall \cref{def:Monsp}). 

Finally, we are ready to conclude the proof of our lemma. By invoking \cref{lem:Monrfl}, we now have that $\mathfrak{g}_k^*$ is Montel and that $\mathfrak{g}_k$ is reflexive. The fact that  $\mathfrak{g}_k^*$ is $(DF)$ follows from Lemma \ref{frechet dual is df} since $\mathfrak{g}_k$ is a Fr\'echet space.

\end{proof}

We now have the necessary ingredients to prove \cref{prop:gkLP}.

\begin{proof}[Proof of \cref{prop:gkLP}]
The proof is an application of \cref{thm:Glockner} with $E=\g_k^*$. Indeed, \cref{prop:gkLA} tells us that $\comm{\cdot}{\cdot}_{\g_k}$ is continuous and using the canonical isomorphism $(\g_k^*)^* \cong \g_k$ given by \cref{lem:gkremon}, there is a continuous, \textit{a fortiori} hypocontinuous, bracket $\comm{\cdot}{\cdot}_{\g_k}: E^*\times E^*\rightarrow E^*$. It remains to check that $\g_k^*$ is a $k^\infty$-space.

To show this property, we need to check that for any $n\in\N$, the product $(\g_k^*)^n$ is a $k$-space. $(\g_k^*)^n$ is (DF) Montel, since $\g_k^*$ is (DF) Montel by \cref{lem:gkremon} and a finite product of (DF) Montel spaces is still (DF) Montel (see \cite[pp. 370, 403]{Kothe1969I}. By \cref{thm:Webb}, $(\g_k^*)^n$ is sequential, hence a $k$-space by \cref{prop:seqk}. This completes the proof of the proposition.
\end{proof}

We next turn to proving \cref{prop:LioPM}, asserting that the Liouville map $\iota_{Lio}:\Rd^N\rightarrow \g_N^*$ is a morphism of Poisson vector spaces. To do this, we need the following technical lemma computing the G\^ateaux derivatives of $\iota_{Lio}$. 

\begin{lemma}\label{lem:diotaL}
It holds that $\iota_{Lio}\in \Cc^\infty(\Rd^N, \g_N^*)$ and for every $n\in\N$, $\uz_N,\ul{w}_N^1,\ldots,\ul{w}_N^n\in \Rd^N$,
\begin{equation}\label{eq:diotaL}
\ds^n\iota_{Lio}[\uz_N](\ul{w}_N^1,\ldots,\ul{w}_N^n) = (-1)^n\frac{1}{N!}\sum_{\pi\in\Ss_N}\sum_{\substack{0\leq n_1,\ldots,n_N\leq n \\ n_1+\cdots+n_N = n}} \sum_{\mathcal{I}} \bigotimes_{j=1}^N \paren*{\nabla^{\otimes n_j}\d_{z_{\pi(j)}} :  \bigotimes_{k=1}^{n} w_{\pi(j)}^{i_k^j}}.
\end{equation}
where the summation $\sum_{\mathcal{I}}$ is over all tuples
\begin{equation}\label{eq:Itup}
\mathcal{I} = (\mathbf{i}^1,\ldots,\mathbf{i}^N), \qquad \mathbf{i}^j := (i_1^j,\ldots,i_n^j) \in \{0,1\}^n \ \text{with} \ i_1^j+\cdots+i_n^j = n_j
\end{equation}
and $w_{\pi(j)}^0$ denotes the factor in the tensor product is vacuous. Here, $\nabla^{\otimes n_j}\d_{z_{\pi(j)}} : \bigotimes_{k=1}^n w_{\pi(j)}^{i_k^j}$ is the distribution in $\Ec'(\R^{2d})$ defined
\begin{equation}
\forall \varphi\in \Cc^\infty(\R^{2d}), \qquad \ipp*{\varphi,\nabla^{\otimes n_j}\d_{z_{\pi(j)}} : \bigotimes_{k=1}^n w_{\pi(j)}^{i_k^j} } = (-1)^{n_j}\nabla^{\otimes n_j}\varphi(z_{\pi(j)}) : \bigotimes_{k=1}^n w_{\pi(j)}^{i_k^j},
\end{equation}
with $\nabla^{\otimes n_j}\varphi = (\p_{x^{\al_1}}\p_{v^{\be_1}}\cdots\p_{x^{\al_{n_j}}}\p_{v^{\be_{n_j}}}\varphi)_{\al_1,\be_1,\ldots,\al_{n_j},\be_{n_j}=1}^d$ and $:$ denoting the tensor inner product.
\end{lemma}
\begin{proof}
The proof follows from the multlinearity of the tensor product and Taylor's theorem.
\end{proof}

\begin{remark}
Specializing the identity \eqref{eq:diotaL} to $n=1$, we obtain
\begin{equation}\label{eq:d1iotaL}
\ds\iota_{Lio}[\uz_N](\ul{w}_N) = -\frac{1}{N!}\sum_{\pi\in\Ss_N}\sum_{i=1}^N \d_{z_{\pi(1)}} \otimes \cdots \otimes \d_{z_{\pi(i-1)}} \otimes \paren*{\nabla\d_{z_{\pi(i)}}\cdot w_{\pi(i)}} \otimes \d_{z_{\pi(i+1)}} \otimes \cdots \otimes \d_{z_{\pi(N)}}.
\end{equation}
\end{remark}

\begin{proof}[Proof of \cref{prop:LioPM}]
Let $\Fc,\Gc\in \Cc^\infty(\g_N^*)$, and set $F\coloneqq \Fc\circ\iota_{Lio}, G\coloneqq \Gc\circ\iota_{Lio}\in \Cc^\infty(\Rd^N)$. By the chain rule, we have the identity
\begin{equation}\label{eq:Liostrt}
\sum_{i=1}^N \nabla_{z_i} F(\uz_N) \cdot w_i =\ds F[\uz_N](\uw_N)=\ds\Fc[\iota_{Lio}(\uz_N)]\Big(\ds\iota_{Lio}[\uz_N](\uw_N)\Big), \qquad \forall \uw_N \in \Rd^N.
\end{equation}
Identifying $\ds\Fc[\iota_{Lio}(\uz_N)]$ as an element of $\g_N$ and using \eqref{eq:d1iotaL}, the preceding right-hand side equals
\begin{multline}
\ipp*{\ds\Fc[\iota_{Lio}(\uz_N)],\ds\iota_{Lio}[\uz_N](\uw_N)}_{\g_N-\g_N^*} \\
= \frac{1}{N!}\sum_{\pi\in\Ss_N}\sum_{i=1}^N w_{\pi(i)}\cdot\nabla_{z_\pi(i)}\ds\Fc[\iota_{Lio}(\uz_N)](z_{\pi(1)},\ldots,z_{\pi(N)}).
\end{multline}
Since $\ds\Fc[\iota_{Lio}(\uz_N)]$ is symmetric with respect to exchange of particle labels, the right-hand side simplifies to
\begin{equation}
\sum_{i=1}^N w_i\cdot\nabla_{z_i}\ds\Fc[\iota_{Lio}(\uz_N)](z_1,\ldots,z_N).
\end{equation}
Returning to our starting identity \eqref{eq:Liostrt}, the arbitrariness of $\uw_N$ and the uniqueness of the gradient field imply that
\begin{equation}
\nabla_{z_i}F(z_1,\ldots,z_N) = \nabla_{z_i}\ds\Fc[\iota_{Lio}(\uz_N)](z_1,\ldots,z_N), \qquad \forall 1\leq i\leq N.
\end{equation}
With this identity, we compute
\begin{align}
\pb{F}{G}_{N}(\uz_N) &= N\sum_{i=1}^N \Big(\nabla_{x_i}F\cdot\nabla_{v_i}G - \nabla_{v_i}F\cdot\nabla_{x_i}G\Big)(\uz_N)\nn\\
&=N\sum_{i=1}^N \Big(\nabla_{x_i}\ds\Fc[\iota_{Lio}(\uz_N)]\cdot\nabla_{v_i}\ds\Gc[\iota_{Lio}(\uz_N)] -\nabla_{v_i}\ds\Fc[\iota_{Lio}(\uz_N)]\cdot\nabla_{x_i}\ds\Gc[\iota_{Lio}(\uz_N)]\Big)(\uz_N)\nn\\
&=\pb{\ds\Fc[\iota_{Lio}(\uz_N)]}{\ds\Gc[\iota_{Lio}(\uz_N)]}_{N}(\uz_N).
\end{align}
Since $\pb{\ds\Fc[\iota_{Lio}(\uz_N)]}{\ds\Gc[\iota_{Lio}(\uz_N)]}_{N}$ is symmetric with respect to exchange of particle labels, the last line may be rewritten as
\begin{align}
&\frac{1}{N!}\sum_{\pi\in\Ss_N}\ipp*{\pb{\ds\Fc[\iota_{Lio}(\uz_N)]}{\ds\Gc[\iota_{Lio}(\uz_N)]}_{N},\d_{z_{\pi(1)}}\otimes\cdots\otimes\d_{z_{\pi(N)}} }_{\g_N-\g_N^*} \nn\\
&=\ipp*{\pb{\ds\Fc[\iota_{Lio}(\uz_N)]}{\ds\Gc[\iota_{Lio}(\uz_N)]}_{N},\iota_{Lio}(\uz_N)}_{\g_N-\g_N^*} \nn\\
&=\pb{\Fc}{\Gc}_{\g_N^*}(\iota_{Lio}(\uz_N)),
\end{align}
which is exactly what we needed to show.
\end{proof}

Using a similar argument, we can also prove that the empirical measure map $\iota_{EM}$ from \eqref{eq:iotaEMdef} is also a Poisson morphism. This then proves \cref{prop:EMpm}. First, a technical lemma, analogous to \cref{lem:diotaL}, for the G\^ateaux derivatives of $\iota_{EM}$. We leave the proof to the reader.

\begin{lemma}\label{lem:diotaEM}
It holds that $\iota_{EM} \in \Cc^\infty(\Rd^N, \g_1^*)$ and for every $n\in\N$, $\uz_N, \uw_N^1,\ldots,\uw_N^n\in\Rd^N$,
\begin{equation}\label{eq:diotaEM}
\ds^n\iota_{EM}[\uz_N](\uw_N^1,\ldots,\uw_N^n) = \frac1N \sum_{j=1}^N   (-1)^n\nabla^{\otimes n}\d_{z_j} : \bigotimes_{k=1}^{n} w_{j}^{k}.
\end{equation}
\end{lemma}

\begin{proof}[Proof of \cref{prop:EMpm}]
Let $\Fc,\Gc\in \Cc^\infty(\g_1^*)$, and set $F \coloneqq \Fc\circ\iota_{EM}$ and $G\coloneqq\Gc\circ\iota_{EM}$, which belong to $\Cc^\infty(\Rd^N)$. By the chain rule, we have the identity
\begin{equation}\label{eq:EMstrt}
\sum_{i=1}^N \nabla_{z_i} F(\uz_N) \cdot w_i =\ds\Fc[\iota_{EM}(\uz_N)]\Big(\ds\iota_{EM}[\uz_N](\uw_N)\Big), \qquad \forall \uw_N \in \Rd^N.
\end{equation}
Identifying $\ds\Fc[\iota_{EM}(\uz_N)]$ as an element of $\g_1$ and using the identity \eqref{eq:diotaEM} specialized to $n=1$, the preceding right-hand side equals
\begin{equation}
\ipp*{\ds\Fc[\iota_{EM}(\uz_N)],\ds\iota_{EM}[\uz_N](\uw_N)}_{\g_1-\g_1^*} = \frac{1}{N}\sum_{j=1}^N w_{j}\cdot\Big(\nabla_{z}\ds\Fc[\iota_{EM}(\uz_N)]\Big)(z_j).
\end{equation}
Returning to our starting identity \eqref{eq:EMstrt}, the arbitrariness of $\uw_N$ and the uniqueness of the gradient field imply
\begin{equation}
\nabla_{z_i}F(z_1,\ldots,z_N) = \frac1N\Big(\nabla_{z}\ds\Fc[\iota_{EM}(\uz_N)]\Big)(z_i), \qquad \forall 1\leq i\leq N.
\end{equation}
With this identity, we compute
\begin{align}
\pb{F}{G}_{N}(\uz_N) &= N\sum_{i=1}^N \Big(\nabla_{x_i}F\cdot\nabla_{v_i}G - \nabla_{v_i}F\cdot\nabla_{x_i}G\Big)(\uz_N)\nn\\
&=N\sum_{i=1}^N \Bigg(\frac1N\Big(\nabla_{x}\ds\Fc[\iota_{EM}(\uz_N)]\Big)(z_i)\cdot \frac1N\Big(\nabla_{v}\ds\Gc[\iota_{EM}(\uz_N)]\Big)(z_i) \nn\\
&\qquad -\frac1N\Big(\nabla_{v}\ds\Fc[\iota_{EM}(\uz_N)]\Big)(z_i)\cdot\frac1N\Big(\nabla_{x}\ds\Gc[\iota_{EM}(\uz_N)]\Big)(z_i)\Bigg)\nn\\
&= \frac1N\sum_{i=1}^N \Bigg(\Big(\nabla_{x}\ds\Fc[\iota_{EM}(\uz_N)]\Big)(z_i)\cdot \Big(\nabla_{v}\ds\Gc[\iota_{EM}(\uz_N)]\Big)(z_i)\nn\\
&\qquad - \Big(\nabla_{v}\ds\Fc[\iota_{EM}(\uz_N)]\Big)(z_i)\cdot\Big(\nabla_{x}\ds\Gc[\iota_{EM}(\uz_N)]\Big)(z_i)\Bigg). \label{eq:EMfinRHS}
\end{align}
But
\begin{align}
\pb{\Fc}{\Gc}_{\g_1^*}(\iota_{EM}(\uz_N)) &= \ipp*{\pb{\ds\Fc[\iota_{EM}(\uz_N)]}{\ds\Gc[\iota_{EM}(\uz_N)]}_{\Rd}, \frac1N\sum_{i=1}^N\d_{z_i}}_{\g_1-\g_1^*} \nn\\
&=\frac1N\sum_{i=1}^N\pb{\ds\Fc[\iota_{EM}(\uz_N)]}{\ds\Gc[\iota_{EM}(\uz_N)]}_{\Rd}(z_i),
\end{align}
which equals the final two lines of \eqref{eq:EMfinRHS}. Thus, the proof is complete.
\end{proof}

We close this subsection by proving the Hamiltonian formulation of the Liouville equation as given by \cref{prop:LioHam}. The reader will recall from \eqref{eq:LioHam} the definition of the Liouville Hamiltonian functional $\Hc_{Lio}$.

\begin{proof}[Proof of \cref{prop:LioHam}]
We remark that since $\Hcl$ is a linear functional, it is trivial that we have the identification $\ds\Hcl[\ga] = \Hc_{New}\in\g_N$ for every $\ga\in\g_N^*$. To find a formula for the Hamiltonian vector field $X_{\Hcl}$ with respect to the Poisson bracket $\pb{\cdot}{\cdot}_{\g_N^*}$, we compute for any $\Fc\in \Cc^\infty(\g_N^*)$ and $\ga\in \g_N^*$,
\begin{align}
\pb{\Fc}{\Hcl}_{\g_N^*}(\ga) &= \ipp*{\comm{\ds\Fc[\ga]}{\ds\Hcl[\ga]}_{\g_N},\ga}_{\g_N-\g_N^*} \nn\\
&= \ipp{\comm{\ds\Fc[\ga]}{\Hc_{New}}_{\g_N},\ga}_{\g_N-\g_N^*}  \nn\\
&= \ipp*{\ds\Fc[\ga], -N\sum_{i=1}^N\Bigg(\div_{x_i}\Big(\nabla_{v_i}\Hc_{New} \ga\Big) - \div_{v_i}\Big(\nabla_{x_i}\Hc_{New}\ga\Big)\Bigg)}_{\g_N-\g_N^*},
\end{align}
where the ultimate line follows from unpacking the definition of $\comm{\cdot}{\cdot}_{\g_N}$ and integration by parts. Since the second entry of the pairing $\ipp{\cdot,\cdot}_{\g_N-\g_N^*}$ in the last line satisfies the characterizing property of the Hamiltonian vector field (recall that $\Fc$ was arbitrary), the uniqueness of the vector field implies that
\begin{align}
X_{\Hcl}(\ga) &= -N\sum_{i=1}^N\Bigg(\div_{x_i}\Big(\nabla_{v_i}\Hc_{New} \ga\Big) - \div_{v_i}\Big(\nabla_{x_i}\Hc_{New}\ga\Big)\Bigg) \nn\\
&= -\sum_{i=1}^N \Bigg(v_i\cdot \nabla_{x_i}\ga - \frac{2}{N}\sum_{\substack{1\leq j\leq N}}\nabla W(x_i-x_j)\cdot\nabla_{v_i} \ga\Bigg),
\end{align}
where the second line follows from the product rule and the fact that $\nabla_{x_i}\nabla_{v_i}\Hc_{New} = \nabla_{v_i}\nabla_{x_i}\Hc_{New} = 0$. Thus, we have shown that equation \eqref{eq:Lio} is equivalent to 
\begin{equation}
\dot{\ga} = X_{\Hcl}(\ga),
\end{equation}
exactly as desired.
\end{proof}

\begin{remark}\label{rem:NewLio}
Together, \Cref{prop:EMpm,prop:LioHam} imply that the Liouville map $\iota_{Lio}$ sends solutions of the Newtonian $N$-particle system \eqref{eq:New} to solutions of the $N$-particle Liouville equation \eqref{eq:Lio}. Indeed, if $\uz_N \in \Cc^\infty(I; \Rd^N)$ is a solution to \eqref{eq:New} on some interval $I$, define $\ga^t \coloneqq \iota_{Lio}(\uz_N^t)$ for every $t\in I$. Using that $\iota_{Lio}^* \Hc_{Lio} = \Hc_{New}$, which is easy to check from the symmetry of $\Hc_{New}$, \cref{prop:LioPM} implies
\begin{equation}
\forall \Fc\in \Cc^\infty(\g_N^*), \qquad \pb{\Hc_{New}}{\iota_{Lio}^*\Fc}_{N}(\uz_N^t) = \pb{\Hc_{Lio}}{\Fc}_{\g_N^*}(\ga^t).
\end{equation}
Since $\iota_{Lio}^*\Fc\in \Cc^\infty(\Rd^N)$ by \cref{lem:diotaL} and the chain rule, one has that
\begin{equation}
\frac{d}{dt}\Fc(\ga^t) = \frac{d}{dt}(\iota_{Lio}^*\Fc)(\uz_N^t) = \pb{\Hc_{New}}{\iota_{Lio}^*\Fc}_{N}(\uz_N^t) = \pb{\Hc_{Lio}}{\Fc}_{\g_N^*}(\ga^t).
\end{equation}
Since $\Fc\in\Cc^\infty(\g_N^*)$ was arbitrary, the claim follows.
\end{remark}

\subsection{Lie algebra $\G_N$ of $N$ particle observables}\label{ssec:NgeomLie}
In this subsection, we transition to discussing $N$-hierarchies, with the goal of proving \cref{prop:GNLA}, which asserts that $(\G_N,\comm{\cdot}{\cdot}_{\G_N})$ is a Lie algebra in the sense of \cref{def:LA}. As sketched in \cref{ssec:OutN}, we accomplish this task through a series of lemmas.

The starting point is the introduction of the maps $\epsilon_{k,N}:\mathfrak{g}_k \rightarrow \mathfrak{g}_N$ for $N \geq k\geq 1$, which in turn will be used to define a Lie bracket on the space of $N$-particle hierarchies of observables. For the reader's benefit, we recall from \cref{ssec:OutN} the definition of $\ep_{k,N}$:
\begin{equation}\label{eq:epdef}
\forall \uz_N \in\Rd^N, \qquad \epsilon_{k,N}(f^{(k)})(\uz_N) \coloneqq  \frac{1}{|P_k^N|}\sum_{(j_1, \dots, j_k) \in P_k^N} f^{(k)}(\uz_{(j_1, \dots, j_k)}),
\end{equation}
where $\uz_{(j_1,\ldots,j_k)} \coloneqq (z_{j_1},\ldots,z_{j_k})$ and
\begin{equation}
P_k^N \coloneqq \{ (j_1, \dots, j_k) : \text{$1 \leq j_i \leq N$ and $j_i$ distinct} \}.
\end{equation}
For example, if $k =1$, $N =2$, and $f^{(1)} \in \mathcal{C}^\infty(\R^{2d})$, then as a function we have
\begin{equation}\label{example epsilon}
\epsilon_{1,2}(f^{(1)})(z_1,z_2) = \frac{1}{2} \left(f^{(1)}(z_1) + f^{(1)}(z_2)\right).
\end{equation}
In the sequel, it will be convenient to use the tuple shorthand $\mathbf{j}_{k} = (j_1,\ldots,j_k)\in P_k^N$, similarly $f_{\bm{j}_k}^{(k)}$ and $\uz_{\bm{j}_k}$.

The next two lemmas show that each map $\ep_{k,N}$ is continuous, linear, and injective (cf. \cite[Lemmas 5.3 and 5.4]{MNPRS2020}). In particular, the first two properties imply that $\ep_{k,N} \in \Cc^\infty(\g_k,\g_N)$.

\begin{lemma}[$\epsilon_{k,N}$ are continuous]\label{lem:epcont}
The maps $\epsilon_{k,N} : \mathfrak{g}_k \rightarrow \mathfrak{g}_N$ are continuous and linear.
\end{lemma}
\begin{proof}
Linearity follows directly from the definition. For continuity, note that the spaces $\mathfrak{g}_k, \mathfrak{g}_N$ are Fr\'echet, so it suffices to show that for any sequence $(f_j)_{j=1}^\infty \subset \mathfrak{g}_k$ with $f_j \rightarrow f \in \mathfrak{g}_k$, we have $\epsilon_{k,N} (f_j) \rightarrow \epsilon_{k,N}(f)$ in $\mathfrak{g}_N$. By linearity, we may assume that $f_j \rightarrow 0$. Now, for any compact set $K \subset \Rd^N$ and $j \in \N$, we estimate  using triangle inequality
\begin{equation}
    \rho_{K,n} (\epsilon_{k,N}(f_j)) \leq \frac{1}{|P_k^N|} \sum_{ \bm{p}_k \in P_k^N} \sup_{|\alpha| \leq n} \|\partial^\alpha (f_{j})_{ \bm{p}_k}\|_{L^\infty(K)} \leq  \rho_{K,n}(f_j),
\end{equation}
which converges to $0$ as $j \rightarrow \infty$. Since $K, n$ were arbitrary, we have that $\epsilon_{k,N} (f_j) \rightarrow 0$ in $\mathfrak{g}_N$. 
\end{proof}

\begin{lemma}($\epsilon_{k,N}$ is injective)\label{lem:epinj}
The maps $\epsilon_{k,N} : \mathfrak{g}_k \rightarrow \mathfrak{g}_N$ are injective, and hence have well defined inverses on their images. 
\end{lemma}
\begin{proof}
Fix $1\leq k\leq N$. To prove injectivity, we will show the contrapositive statement: if $f^{(k)} \neq 0$, then $\ep_{k,N}(f^{(k)}) \neq 0$. The argument presented below is a ``classical version'' of the argument used to prove \cite[Lemma 5.4]{MNPRS2020}.

We introduce a parameter $n\in\N_0$ with $n<k$. We say that $f^{(k)}$ has \emph{property $\mathbf{P}_n$} if the following holds: if $n=0$, then there exists a point $z_0\in\R^{2d}$ such that
\begin{equation}
f^{(k)}(z_0,\ldots,z_0) \neq 0
\end{equation}
and if $n\geq 1$, then there exist points $z_0,\ldots,z_n \in \R^{2d}$ such that
\begin{equation}
f^{(k)}(z_0^{\times k-n},z_1,\ldots,z_n) \neq 0,
\end{equation}
where $z_0^{\times k-n} \coloneqq (z_0,\ldots,z_0)\in \Rd^{k-n}$ and the Cartesian product is understood as vacuous when $n=k$. Observe that $f^{(k)}$ always has property $\mathbf{P}_{k-1}$, since $f^{(k)}$ is nonzero by assumption, therefore there exist points $z_0,\ldots,z_{k-1}\in\R^{2d}$ such that $f^{(k)}(z_0,\ldots,z_{k-1})\neq 0$. We define the integer $n_{\min}$ by
\begin{equation}
n_{\min} \coloneqq \min\{0\leq n<k : f^{(k)} \ \text{has property} \ \mathbf{P}_n\}.
\end{equation}

To avoid confusion over notation, we first dispense with the trivial case $n_{\min} = 0$. The definition of $\mathbf{P}_0$ implies that there exists a point $z_0\in \R^{2d}$ such that $f^{(k)}(z_0,\ldots,z_0) \neq 0$. It then follows from the definition of $f_{(j_1,\ldots,j_k)}^{(k)}$ that $f_{(j_1,\ldots,j_k)}^{(k)}(z_0^{\times N}) = f^{(k)}(z_0^{\times k})$ for each tuple $(j_1,\ldots,j_k)\in P_k^N$. Hence, 
\begin{equation}
\ep_{k,N}(f^{(k)})(z_0^{\times N}) = f^{(k)}(z_0^{\times k}) \neq 0.
\end{equation}

We next consider the case $1\leq n_{\min}<k$. The definition of $\mathbf{P}_{n_{\min}}$ implies that there exist points $z_0,\ldots,z_{n_{\min}} \in \R^{2d}$ such that
\begin{equation}
f^{(k)}(z_0^{\times k-n_{\min}}, z_1,\ldots,z_{n_{\min}}) \neq 0.
\end{equation}
We claim that $\ep_{k,N}(f^{(k)})(z_0^{\times N-n_{\min}},z_1,\ldots,z_{n_{\min}}) \neq 0$. To see this, we observe from unpacking the definition of $\ep_{k,N}(f^{(k)})$ that
\begin{equation}
\ep_{k,N}(f^{(k)})(z_0^{\times N-n_{\min}},z_1,\ldots,z_{n_{\min}}) = C_{k,N}\sum_{\mathbf{j}_k \in P_{k}^N} f_{(j_1,\ldots,j_k)}^{(k)}(z_0^{\times N-n_{\min}}, z_1,\ldots,z_{n_{\min}}).
\end{equation}
For each $\mathbf{j}_k=(j_1,\ldots,j_k) \in P_k^N$, we can use the symmetry of $f^{(k)}$ to write
\begin{equation}
f_{(j_1,\ldots,j_k)}^{(k)}(z_0^{\times N-n_{\min}}, z_1,\ldots,z_{n_{\min}}) = f^{(k)}(w_1,\ldots,w_k),
\end{equation}
where for some $r\leq k$, $w_1=\cdots=w_r = z_0$ and for $r<i\leq k$, the $w_i$ are distinct elements of $\{z_1,\ldots,z_{n_{\min}}\}$. If $r> k-n_{\min}$, then by definition of $n_{\min}$, $f^{(k)}(w_1,\ldots,w_k) = 0$. Since only $n_{\min}$ coordinates of $(z_0^{\times N-n_{\min}}, z_1,\ldots,z_{n_{\min}})$ are not equal to $z_0$, we must have $r=k-n_{\min}$. But in this case, the symmetry of $f^{(k)}$ implies
\begin{equation}
f^{(k)}(w_1,\ldots,w_k) = f^{(k)}(z_0^{\times k-n_{\min}},z_1,\ldots,z_{n_{\min}}) \neq 0,
\end{equation}
by choice of the points $z_0,z_1,\ldots,z_{n_{\min}}$. Therefore, we have shown that
\begin{align}
C_{k,N}\sum_{\mathbf{j}_k \in P_{k}^N} f_{(j_1,\ldots,j_k)}^{(k)}(z_0^{\times N-n_{\min}}, z_1,\ldots,z_{n_{\min}}) &= C_{k,N}' f^{(k)}(w_1,\ldots,w_k) \nonumber\\
&= C_{k,N}'f^{(k)}(z_0^{\times k-n_{\min}},z_1,\ldots,z_{n_{\min}}) \nonumber\\
&\neq 0,
\end{align}
where $C_{k,N}'$ is some other combinatorial factor depending on $k,N$. This then implies $\ep_{k,N}(f^{(k)}) \neq 0$, completing the proof of injectivity. 
\end{proof}

We now present a technical lemma which will be applied to prove \cref{lem:epfil} for the filtration property. It shows the commutativity of the following diagram.
\begin{center}
\begin{tikzcd}[column sep=3cm, row sep =3cm]
 \g_a \arrow[dr, dashrightarrow, "\ep_{b,N}\circ\ep_{a,b}", marking]  \arrow[r, "\ep_{a,b}"]
 	 & \g_b \arrow[d, "\ep_{b,N}"] \\
 	 & \g_N
\end{tikzcd}
\end{center}

\begin{lemma}\label{lem:epabbN}
Let $ 1 \leq a \leq b \leq N$. Then we have that $\epsilon_{a,N} = \epsilon_{b,N} \circ \epsilon_{a,b}$.
\end{lemma}
\begin{proof}
Let $f \in \mathfrak{g}_a$. Then by definition of $\epsilon_{a,b}$ and $\epsilon_{b,N}$,
\begin{align}
   \epsilon_{b,N} (\epsilon_{a,b}(f)) (\uz_N) 
   & = \frac{1}{|P_b^N|} \sum_{\mathbf{m}_b \in P_b^N} \epsilon_{a,b}(f)({\uz}_{(m_1,m_2,\dots,m_b)}) \nonumber\\
   & = \frac{1}{|P_b^N||P_a^b|} \sum_{\mathbf{m}_b \in P_b^N}
   \sum_{\mathbf{n}_a \in P_a^b}
   f(\uz_{(m_{n_1},m_{n_2},\dots,m_{n_a})}).\label{eq:epabepbN}
\end{align}
Fix an $a$-tuple $\mathbf{l}_a = (l_1,\ldots,l_a) \in P_a^N$. Let $\mathscr{A}_{\mathbf{l}_a}$ denote the set of $b$-tuples $\mathbf{m}_b\in P_b^N$ such that $\{l_1,\ldots,l_a\}\subset \{m_1,\ldots,m_b\}$. Any element in $\mathbf{m}_b\in \mathscr{A}_{\mathbf{l}_a}$ is a permutation of $(l_1,\ldots,l_a,j_1,\ldots,j_{b-a})$ for some choice $1\leq j_1<\cdots < j_{b-a}\leq N$ not in $\{l_1,\ldots,l_a\}$. There are ${N-a \choose {b-a}}$ such choices $(j_1,\ldots,j_{b-a})$, and $b!$ permutations of $b$ letters, hence
\begin{equation}\label{eq:Alcard}
|\mathscr{A}_{\mathbf{l}_a}| = {N-a\choose b-a}b! = \frac{b!(N-a)!}{(N-b)!(b-a)!}.
\end{equation}
We also see that given $\mathbf{m}_b \in \mathscr{A}_{\mathbf{l}_a}$, there is a unique choice of indices $n_1,\ldots,n_a\in\{1,\ldots,b\}$ such that $(m_{n_1},\ldots,m_{n_a}) = (l_1,\ldots,l_a)$. Let us denote this unique choice by $\mathbf{n}_{a,\mathbf{l}_a}$. We write
\begin{equation}
\sum_{\mathbf{m}_b \in P_b^N}\sum_{\mathbf{n}_a \in P_a^b} (\cdots) = \sum_{\mathbf{l}_a \in P_a^N} \sum_{\substack{(\mathbf{m}_b,\mathbf{n}_a) \in P_b^N\times P_a^b \\ (m_{n_1},\ldots,m_{n_a}) = (l_1,\ldots,l_a)}} (\cdots) = \sum_{\mathbf{l}_a \in P_a^N} \sum_{(\mathbf{m}_b,\mathbf{n}_a) \in \mathscr{A}_{\mathbf{l}_a} \times \{\mathbf{n}_{a,\mathbf{l}_a}\}}(\cdots).
\end{equation}
For any $(\mathbf{m}_b,\mathbf{n}_a) \in \mathscr{A}_{\mathbf{l}_a} \times \{\mathbf{n}_{a,\mathbf{l}_a}\}$, we have $f(\uz_{(m_{n_1},\ldots,m_{n_a})})=f(\uz_{\mathbf{l}_a})$. Returning to \eqref{eq:epabepbN}, this identity and \eqref{eq:Alcard} imply
\begin{align}
\epsilon_{b,N} (\epsilon_{a,b}(f)) (\uz_N)  &= \frac{1}{|P_b^N||P_a^b|} \frac{b!(N-a)!}{(N-b)!(b-a)!}\sum_{\mathbf{l}_a \in P_a^N}  f(\uz_{\mathbf{l}_a}) \nn \\
&= \frac{(N-a)!}{N!}\sum_{\mathbf{l}_a \in P_a^N}  f(\uz_{\mathbf{l}_a}) \nn\\
&=\ep_{a,N}(f)(\uz_N),
\end{align}
where the penultimate line follows from simplification of the combinatorial factor in the first line and the ultimate line is tautological. Hence, the lemma is proved. 
\end{proof}

\begin{lemma}($\epsilon_{k,N}$ filters)\label{lem:epfil}
Let $N \in \N$, and let $1 \leq \ell, j \leq N$. Then for any $f^{(\ell)} \in \mathfrak{g}_\ell$ and $g^{(j)} \in \mathfrak{g}_j$, there exists a unique $h^{(k)} \in \mathfrak{g}_k$ such that 
\begin{equation}\label{eq:epfil}
\comm{\epsilon_{\ell,N}(f^{(\ell)})}{\epsilon_{j,N}(g^{(j)})}_{\mathfrak{g}_N} = \epsilon_{k,N} (h^{(k)}), \qquad \text{ with $k \coloneqq \min ( \ell + j -1 , N)$},
\end{equation}
given by
\begin{equation}
h^{(k)} = \sum_{r = r_0}^{\min( \ell,j)} \frac{(N-\ell)!(N-j)!}{(N-1)! (N-\ell -j +r)!} \epsilon_{\ell+j - r, k}\paren*{\Sym_{\ell+j-r}\paren*{f^{(\ell)} \wedge_r g^{(j)}}},
\end{equation}
where $r_0 \coloneqq \max (1 ,\ell+j-N)$ and for $1\leq r\leq \min(\ell, j)$,  $f^{(\ell)}\wedge_r g^{(j)} \in \Cc^\infty(\Rd^{\ell+j-r})$ is defined by
\begin{multline}\label{r wedge}
f^{(\ell)} \wedge_r g^{(j)}(\uz_{\ell+j-r}) \coloneqq \binom{\ell}{r}\binom{j}{r} r!\sum_{i=1}^r \Big(\nabla_{x_{i}} f^{(\ell)}(\uz_{\ell}) \cdot \nabla_{v_{i}} g^{(j)}(\uz_{(1,\ldots,r, \ell+1,\ldots,\ell+j-r)}) \\
-  \nabla_{x_i} g^{(j)}(\uz_{j})\cdot\nabla_{v_i} f^{(\ell)}(\uz_{(1,\ldots,r, j+1, \ldots , \ell+j-r)})
    \Big).
\end{multline}
\end{lemma}

\begin{remark}
In the right-hand side of \eqref{r wedge}, the ranges in the tuple $(1,\ldots,r,\ell+1,\ldots,\ell+j-r)$ or $(1,\ldots,r,j+1,\ldots,j+\ell-r)$ are understood as vacuous whenever the lower and upper bounds do not make sense. For example, if $j=1$, then $r=1$ and $(1,\ldots,r,\ell+1,\ldots,\ell+j-r)$ is understood as just $(1)$. Singling out these exceptional cases would be tedious, and therefore we will not do so.
\end{remark}

\begin{proof}[Proof of \cref{lem:epfil}]
Injectivity of the $\epsilon_{k,N}$ operators shows uniqueness. The case where $\ell + j -1 \geq N$ is trivial, since the operator $\epsilon_{N,N}$ is the identity on $\mathfrak{g}_N$. So, assume that $ \ell + j -1 < N$, in which case $k=\ell+j-1$. From the definition \eqref{eq:epdef} of $\ep_{k,N}$, we obtain the following equality of functions on $\Rd^N$:

\begin{align}
     \comm{\epsilon_{\ell,N}(f^{(\ell)})}{\epsilon_{j,N}(g^{(j)})}_{\mathfrak{g}_N} & = N \pb{\epsilon_{\ell,N}(f^{(\ell)})}{\epsilon_{j,N}(g^{(j)})}_{\Rd^N}\nonumber \\
     & = \frac{N}{|P_\ell^N| |P_j^N|} \sum_{i =1}^N
     \sum_{\bm{m}_\ell \in P_\ell^N} \sum_{\bm{n}_j \in P_j^N}
     \left( \nabla_{x_i} f_{\bm{m}_\ell}^{(\ell)} \cdot\nabla_{v_i} g_{\bm{n}_j}^{(j)} - \nabla_{x_i} g_{\bm{n}_j}^{(j)} \cdot\nabla_{v_i} f_{\bm{m}_\ell}^{(\ell)}\right) \nn \\
     & =\frac{N}{|P_\ell^N| |P_j^N|} \sum_{r = 1}^{\min( \ell,j)}  
     \sum_{\substack{(\bm{m}_\ell,\bm{n}_j) \in P_\ell^N\times P_j^N \\  |\{m_1,\dots,m_\ell\} \cap \{n_1,\dots n_j\}|=r}} \sum_{i=1}^N
     \Big( \nabla_{x_i} f_{\bm{m}_\ell}^{(\ell)} \cdot\nabla_{v_i} g_{\bm{n}_j}^{(j)} \nonumber\\
     &\quad - \nabla_{x_i} g_{\bm{n}_j}^{(j)}  \cdot\nabla_{v_i} f_{\bm{m}_\ell}^{(\ell)}\Big). \label{bracket-eps}
\end{align}
To avoid any confusion over notation, we remind the reader that, here, $f_{\bm{m}_\ell}^{(\ell)}, g_{\bm{n}_j}^{(j)}$ are regarded as functions on $\Rd^N$, and therefore, for instance,
\begin{align}
\nabla_{x_i}f_{\bm{m}_\ell}^{(\ell)}(\uz_N) &= \lim_{h\rightarrow 0} \frac{f_{\bm{m}_\ell}^{(\ell)}(\uz_N + h e_{i,x}) - f_{\bm{m}_\ell}^{(\ell)}(\uz_N)}{h},
\end{align}
where $e_{i,x}$ denotes the basis vector in $\Rd^N$ for the $x_i$ variable. Note that if $i\notin \{m_1,\ldots,m_\ell\}\cap \{n_1,\ldots,n_j\}$, then $\nabla_{z_i}f_{\bm{m}_\ell}^{(\ell)} = 0$ or $\nabla_{z_i}g_{\bm{n}_j}^{(j)}=0$. Thus,
\begin{equation}
\sum_{i=1}^N (\cdots) = \sum_{ \substack{ i \in \{m_1,\dots,m_\ell\} \cap \{n_1,\dots n_j\}}} (\cdots).
\end{equation}
In particular, if the cardinality of the intersection equals $r$, then there are only $r$ indices $i$ for which the expression inside the parentheses is possibly nonzero.

Let $ 1 \leq r \leq \min(\ell, j)$. Let us count the number of elements in the set
\begin{equation}\label{eq:PlNPjN}
\{(\bm{m}_\ell,\bm{n}_j)\in P_{\ell}^N \times P_j^N : |\{m_1,\ldots,m_\ell\}\cap \{n_1,\ldots,n_j\}| = r\}.
\end{equation}
For a positive integer $q$, let $\mathscr{P}_{q}^N$ denote the subsets of $\{1,\ldots,N\}$ with cardinality $q$. As the reader may check, there is a bijection between the set \eqref{eq:PlNPjN} and the set of tuples
\begin{multline}
\mathcal{T}=\paren*{A_{o,r}, A_{opos,m}, A_{opos,n}, A_{no,m},A_{no,n}, \pi_{r,m}, \pi_{r,n}, \tau_{\ell-r,m}, \tau_{j-r,n}} \\
\in \mathscr{P}_r^N \times \mathscr{P}_{r}^\ell \times \mathscr{P}_{r}^j \times\mathscr{P}_{\ell-r}^{N-r} \times \mathscr{P}_{j-r}^{N-r} \times \Ss_{r}^2 \times \Ss_{\ell-r}\times\Ss_{j-r}.
\end{multline}
In words, $A_{o,r}$ is the set of $r$ overlaps between the elements of $\bm{m}_\ell$ and $\bm{n}_j$; $A_{opos,m}, A_{opos,n}$ are the sets of indices corresponding to the placements of the overlaps in $\bm{m}_\ell,\bm{n}_j$, respectively; $A_{no,m},A_{no,n}$ are the remaining sets of elements in $\bm{m}_\ell,\bm{n}_j$, respectively, which do not overlap (and therefore which are disjoint from $A_{o,r}$); $\pi_{r,m},\pi_{r,n}$ are permutations of the sets $A_{opos,m}, A_{opos,n}$; and $\tau_{\ell-r,m},\tau_{j-r,n}$ are permutations of the sets $\{1,\ldots,N\}\setminus A_{opos,m}, \{1,\ldots,N\}\setminus A_{opos,n}$, respectively. We note that
\begin{align}
(A_{o,r}, A_{opos,m}, A_{no,m}, \pi_{r,m},\tau_{\ell-r,m}) \mapsto \bm{m}_\ell \\
(A_{o,r}, A_{opos,n}, A_{no,n}, \pi_{r,n}, \tau_{j-r,n}) \mapsto \bm{n}_j
\end{align}
define bijections. 

Fix an overlap set $A_{o,r}$ and fix nonoverlap sets $A_{no,m},A_{no,n}$. If
\begin{align}
\bm{m}_\ell \leftrightarrow (A_{o,r},A_{opos,m}, A_{no,m},\pi_{r,m},\tau_{\ell-r,m}),\\
\bm{m}_\ell' \leftrightarrow (A_{o,r},A_{opos,m}', A_{no,m},\pi_{r,m}',\tau_{\ell-r,m}'),
\end{align}
where $\leftrightarrow$ denotes the bijection, then the symmetry of $f^{(\ell)}$ implies $f_{\bm{m}_\ell}^{(\ell)} = f_{\bm{m}_\ell'}^{(\ell)}$. Similarly, if
\begin{align}
\bm{n}_j  \leftrightarrow (A_{o,r}, A_{opos,n},A_{no,n},\pi_{r,n},\tau_{j-r,n}),\\
\bm{n}_j' \leftrightarrow (A_{o,r},A_{opos,n}',A_{no,n},\pi_{r,n}',\tau_{j-r,n}'),
\end{align}
then the symmetry of $g^{(j)}$ implies $g_{\bm{n}_j}^{(j)} = g_{\bm{n}_j'}^{(j)}$. Since $|\mathscr{P}_{r}^\ell| = {\ell\choose r}$, $|\mathscr{P}_{r}^j| = {j\choose r}$, and $|\Ss_{r}|=r!$, it follows that
\begin{multline}
\sum_{\substack{(\bm{m}_\ell,\bm{n}_j) \in P_\ell^N\times P_j^N \\  |\{m_1,\dots,m_\ell\} \cap \{n_1,\dots n_j\}|=r}} \sum_{i\in \{m_1,\ldots,m_\ell\}\cap \{n_1,\ldots,n_j\}}
     \Big( \nabla_{x_i} f_{\bm{m}_\ell}^{(\ell)} \cdot\nabla_{v_i} g_{\bm{n}_j}^{(j)}      - \nabla_{x_i} g_{\bm{n}_j}^{(j)}  \cdot\nabla_{v_i} f_{\bm{m}_\ell}^{(\ell)}\Big)\\
     = {j\choose r}{\ell\choose r}r!\sum_{\substack{\mathcal{T}\\ A_{opos,m}=A_{opos,n} = \{1,\ldots,r\} \\ \pi_{r,n}=\pi_{r,m}}} \sum_{i=1}^r\Big( \nabla_{x_i} f_{\bm{m}_\ell}^{(\ell)} \cdot\nabla_{v_i} g_{\bm{n}_j}^{(j)}      - \nabla_{x_i} g_{\bm{n}_j}^{(j)}  \cdot\nabla_{v_i} f_{\bm{m}_\ell}^{(\ell)}\Big).
     \end{multline}
Since there is a bijection between tuples $\bm{p}_{\ell+j-r} \in P_{\ell+j-r}^N$ and tuples
\begin{equation}
\paren*{A_{o,r},A_{no,m}, A_{no,n},\pi_{r,m},\tau_{\ell-r,m},\tau_{j-r,n}} \in \mathscr{P}_r^N \times \mathscr{P}_{\ell-r}^{N-r} \times \mathscr{P}_{j-r}^{N-\ell} \times \Ss_{r} \times \Ss_{\ell-r}\times\Ss_{j-r},
\end{equation}
we conclude upon relabeling that \eqref{bracket-eps} equals
\begin{multline}\label{eq:rsumvac}
\frac{N}{|P_\ell^N| |P_j^N|} \sum_{r = 1}^{\min( \ell,j)} {j\choose r}{\ell\choose r}r! \sum_{\bm{p}_{\ell+j-r}} \sum_{i=1}^r \Big(\nabla_{x_{p_i}} f_{\bm{p}_\ell}^{(\ell)} \cdot\nabla_{v_{p_i}} g_{(\bm{p}_r,\bm{p}_{\ell+1;\ell+j-r})}^{(j)} -  
\nabla_{x_{p_i}} g_{\bm{p}_j}^{(j)} \cdot\nabla_{v_{p_i}} f_{(\bm{p}_r,\bm{p}_{j+1;j+\ell-r})}^{(\ell)}\Big),
\end{multline}
where we have used the shorthand $(\bm{p}_r,\bm{p}_{\ell+1;\ell+j-r}) = (p_1,\ldots,p_r,p_{\ell+1},\ldots,p_{\ell+j-r})$ (similarly when $\ell$ and $j$ are swapped). Note that for $1 \leq r < \max (1 , \ell+j-N) \eqqcolon r_0$, the preceding sum is vacuous. 

Using the $\wedge_r$ wedge product notation and the fact that the sum over $\mathbf{p}_{\ell+j-r}\in P_{\ell+j-r}^N$ is invariant under the $\Ss_{\ell+j-r}$ action, we can rewrite the expression \eqref{eq:rsumvac} as
\begin{align}
    &\frac{N}{|P_\ell^N| |P_j^N|} \sum_{r = r_0}^{\min( \ell,j)} \sum_{\bm{p}_{\ell +j -r} \in P_{\ell +j -r}^N} \paren*{ f^{(\ell)} \wedge_r g^{(j)}}_{\bm{p}_{\ell+j-r}} \nn\\
    & = \frac{N}{|P_\ell^N| |P_j^N|} \sum_{r = r_0}^{\min( \ell,j)} |P^N_{\ell+j -r}| \epsilon_{\ell+j - r,N}\paren*{{\Sym_{\ell+j-r}\paren*{f^{(\ell)} \wedge_r g^{(j)}}}} \nn\\
   & = \epsilon_{\ell+j - 1,N} \paren*{\frac{N}{|P_\ell^N| |P_j^N|} \sum_{r = r_0}^{\min( \ell,j)} \frac{N!}{(N - \ell - j +r)!}  \epsilon_{\ell+j - r, \ell + j -1}\paren*{{\Sym_{\ell+j-r}\paren*{f^{(\ell)} \wedge_r g^{(j)}}}}}.
\end{align}
where to obtain the final line, we have used \cref{lem:epabbN} with $a = \ell + j -r, b = \ell + j -1$ and have used the linearity of $\ep_{\ell+j-1,N}$. Observing
\begin{equation}
\frac{N}{|P_\ell^N| |P_j^N|} \frac{N!}{(N-\ell-j+r)!}  = \frac{(N-\ell)!(N-j)!}{(N-1)!(N-\ell-j+r)!},
\end{equation}
we arrive at the stated assertion of the lemma. This completes the proof in all cases of $\ell,j$.
\end{proof}

With \cref{lem:epfil} in hand, we can show that the expression \eqref{eq:GNLBdef} for $\comm{F}{G}_{\G_N}$ is well-defined. Indeed, fix $1\leq k<N$, given $F=(f^{(\ell)})_{\ell=1}^N, G=(g^{(j)})_{j=1}^N\in\G_N$, \cref{lem:epfil} yields, for any $1\leq \ell,j\leq N$ satisfying $\ell+j-1 = k$,
\begin{equation}
\comm{\epsilon_{\ell,N}(f^{(\ell)})}{\epsilon_{j,N}(g^{(j)})}_{\mathfrak{g}_N}^{(k)} =\epsilon_{k,N}\paren*{\sum_{r = r_0}^{\min( \ell,j)} \frac{(N-\ell)!(N-j)!}{(N-1)! (N-\ell -j +r)!} \epsilon_{\ell+j - r, k}\paren*{\Sym_{\ell+j-r}\paren*{f^{(\ell)} \wedge_r g^{(j)}}}}.
\end{equation}
Summing over $\ell,j$ such that $\ell+j-1=k$ and applying $\epsilon_{k,N}^{-1}$ to both sides, we arrive at \eqref{eq:GNLBdef}. In fact, we have also obtained an explicit formula for $\comm{F}{G}_{\G_N}$:
\begin{equation}
\comm{F}{G}_{\G_N}^{(k)} = \sum_{\substack{1\leq \ell,j\leq N \\ \ell+j-1=k}}\sum_{r = r_0}^{\min( \ell,j)} \frac{(N-\ell)!(N-j)!}{(N-1)! (N-\ell -j +r)!} \epsilon_{\ell+j - r, k}\paren*{\Sym_{\ell+j-r}\paren*{f^{(\ell)} \wedge_r g^{(j)}}}.
\end{equation}
If $k=N$, then we should modify the preceding reasoning to sum over $\ell+j-1\geq N$. In all cases of $k$, we obtain
\begin{equation}\label{eq:GNLBform}
\comm{F}{G}_{\G_N}^{(k)} = \sum_{\substack{1\leq \ell,j\leq N \\ \min(\ell+j-1,N) = k}}\sum_{r = r_0}^{\min( \ell,j)} \frac{(N-\ell)!(N-j)!}{(N-1)! (N-\ell -j +r)!} \epsilon_{\ell+j - r, k}\paren*{\Sym_{\ell+j-r}\paren*{f^{(\ell)} \wedge_r g^{(j)}}}
\end{equation}

\begin{remark}\label{rem:Nconlim}
Note that the constant 
\begin{equation}
C_{j\ell N r} \coloneqq \frac{(N-\ell)!(N-j)!}{(N-1)! (N-\ell -j +r)!}
\end{equation}
satisfies $\lim_{N \rightarrow \infty }C_{j \ell N r}= \lim_{N \rightarrow \infty }N^{1-r} = \indic_{r = 1}$. This observation will be used in \Cref{ssec:GinfLA,ssec:GinfLP} to evaluate $N\rightarrow\infty$ limits of $N$-particle Lie and Lie-Poisson brackets.
\end{remark}

We now have all the ingredients to prove \cref{prop:GNLA}.

\begin{proof}[Proof of \cref{prop:GNLA}]
We first consider the algebraic part, which amounts to checking properties \ref{LA1}-\ref{LA3} from \cref{def:LA}. This part is similar to the algebraic portion of the proof of \cite[Proposition 2.1]{MNPRS2020}, but we present again the computations in an effort to make the present article self-contained. 

The first two properties are a ready consequence of the definition of $\comm{\cdot}{\cdot}_{\mathfrak{G}_N}$. For the third property, let $F,G,H\in\mathfrak{G}_N$. We need to show that
\begin{equation}
\comm{F}{\comm{G}{H}_{\G_N}}_{\G_N}+ \comm{H}{\comm{F}{G}_{\G_N}}_{\G_N}+ \comm{G}{\comm{H}{F}_{\G_N}}_{\G_N} =0.
\end{equation}
Since $\epsilon_{k,N}$ is injective, it suffices to show that $\epsilon_{k,N}$ applied to the $k$-th component of the left-hand side of the preceding identity equals the zero element of $\g_N$. We only consider the case $1\leq k<N$ and leave the $k=N$ case as an exercise for the reader. Using the definition of the Lie bracket and bilinearity, we have the identities
\begin{align}
\ep_{k,N}\left(\comm{F}{\comm{G}{H}_{\G_N}}_{\G_N}^{(k)}\right) &= \sum_{j_1+j_2-1=k} \comm{\ep_{j_1,N}(F^{(j_1)})}{\ep_{j_2,N}(\comm{G}{H}_{\G_N}^{(j_2)})}
_{\g_N} \nonumber\\
&=\sum_{j_1+j_2-1=k} \sum_{j_3+j_4-1=j_2} \comm{\ep_{j_1,N}(F^{(j_1)})}{\comm{\ep_{j_3,N}(G^{(j_3)})}{\ep_{j_4,N}(H^{(j_4)})}_{\g_N}}_{\g_N} \nonumber \\
&= \sum_{\ell_1+\ell_2+\ell_3=k+2} \comm{\ep_{\ell_1,N}(F^{(\ell_1)})}{\comm{\ep_{\ell_2,N}(G^{(\ell_2)})}{ 
\ep_{\ell_3,N}(H^{(\ell_3)})}_{\g_N}}_{\g_N},
\end{align}
\begin{align}
\ep_{k,N}\left(\comm{H}{\comm{F}{G}_{\G_N}}_{\G_N}^{(k)}\right) &= \sum_{j_1+j_2-1=k} \comm{\ep_{j_1,N}(H^{(j_1)})}{\ep_{j_2,N}(\comm{F}{G}_{\G_N}^{(j_2)})}_{\g_N} \nonumber\\
&=\sum_{j_1+j_2-1=k} \sum_{j_3+j_4-1=j_2} \comm{\ep_{j_1,N}(H^{(j_1)})}{\comm{\ep_{j_3,N}(F^{(j_3)})}{\ep_{j_4,N}(G^{(j_4)})}_{\g_N}}_{\g_N} \nonumber \\
&=\sum_{\ell_1+\ell_2+\ell_3=k+2} \comm{\ep_{\ell_3,N}(H^{(\ell_3)})}{\comm{\ep_{\ell_1,N}(F^{(\ell_1)})}{\ep_{\ell_2,N}(G^{(\ell_2)})}_{\g_N}}_{\g_N},
\end{align}
\begin{align}
\ep_{k,N}\left(\comm{G}{\comm{H}{F}_{\G_N}}_{\G_N}^{(k)}\right) &= \sum_{j_1+j_2-1=k} \comm{\ep_{j_1,N}(G^{(j_1)})}{\ep_{j_2,N}(\comm{H}{F}_{\G_N}^{(j_2)})}_{\g_N} \nonumber\\
&=\sum_{j_1+j_2-1=k} \sum_{j_3+j_4-1=j_2} \comm{\ep_{j_1,N}(G^{(j_1)})}{\comm{\ep_{j_3,N}(H^{(j_3)})}{\ep_{j_4,N}(F^{(j_4)})}_{\g_N}}_{\g_N} \nonumber \\
&=\sum_{\ell_1+\ell_2+\ell_3=k+2} \comm{\ep_{\ell_2,N}(G^{(\ell_2)})}{\comm{\ep_{\ell_3,N}(H^{(\ell_3)})}{\ep_{\ell_1,N}(F^{(\ell_1)})}_{\g_N}}_{\g_N}.
\end{align}
Since $[\cdot,\cdot]_{\g_N}$ is a Lie bracket and therefore satisfies the Jacobi identity, it follows that for fixed integers $1\leq \ell_1,\ell_2,\ell_3\leq N$,
\begin{equation}
\begin{split}
0 &=\comm{\ep_{\ell_1,N}(F^{(\ell_1)})}{\comm{\ep_{\ell_2,N}(G^{(\ell_2)})}{\ep_{\ell_3,N}(H^{(\ell_3)})}_{\g_N}}_{\g_N} \\
&\phantom{=}+ \comm{\ep_{\ell_3,N}(H^{(\ell_3)})}{\comm{\ep_{\ell_1,N}(F^{(\ell_1)})}{\ep_{\ell_2,N}(G^{(\ell_2)})}_{\g_N}}_{\g_N} \\
&\phantom{=}\phantom{=} + \comm{\ep_{\ell_2,N}(G^{(\ell_2)})}{\comm{\ep_{\ell_3,N}(H^{(\ell_3)})}{\ep_{\ell_1,N}(
F^{(\ell_1)})}_{\g_N}}_{\g_N}.
\end{split}
\end{equation}
Hence,
\begin{equation}
\ep_{k,N}\left(\comm{F}{\comm{G}{H}_{\G_N}}_{\G_N}^{(k)} + \comm{H}{\comm{F}{G}_{\G_N}}_{\G_N}^{(k)} + \comm{G}{\comm{H}{F}_{\G_N}}_{\G_N}^{(k)}\right) = 0\in\g_N.
\end{equation}

We now consider the analytic part, which amounts to checking the continuity of $[\cdot,\cdot]_{\G_N}$.  We wish to show that
\begin{equation}
\G_N \times \G_N \rightarrow \G_N, \qquad (F,G)\mapsto \comm{F}{G}_{\G_N}
\end{equation}
is continuous, for which it suffices to show that for each $k\in \{1,\ldots,N\}$, the map
\[
\G_N \times \G_N \rightarrow \g_k, \qquad (F, G)\mapsto \comm{F}{G}_{\G_N}^{(k)}
\]
is continuous. By the $\G_N$ Lie bracket formula given in \eqref{eq:GNLBform} and the continuity of $\epsilon_{k,N}$ from \cref{lem:epcont}, the continuity of $(F,G)\mapsto [F,G]_{\G_N}$ is then reduced to proving continuity of the map
\begin{equation}
\Cc^\infty(\Rd^\ell) \times \Cc^\infty(\Rd^j) \rightarrow \Cc^\infty(\Rd^{\ell+j-r}), \qquad (f^{(\ell)},g^{(j)})\mapsto f^{(\ell)} \wedge_r g^{(j)},
\end{equation}
where $\wedge_r$ is defined in \eqref{r wedge}. Since $\wedge_r$ is a linear combination of products of derivatives, it is continuous by a similar argument to the analytic part of the proof of \cref{prop:gkLA}.
\end{proof}

We conclude this subsection with what we believe is an interesting fact relating the mappings $\epsilon_{k,N}$ with the notion of taking the marginal of a distribution, the latter notion being important to obtaining Hamiltonian vector field formulae in \Cref{ssec:NgeomLP,ssec:GinfLP}. For each $1 \leq k < N$, define the $k$-particle marginal mapping
\begin{equation}
\int_{\Rd^{N-k}} d\uz_{k+1;N} : \mathfrak{g}_N^* \rightarrow \mathfrak{g}_k^*,
\end{equation}
where, for $\ga\in \g_N^*$, $\int_{\Rd^{N-k}}d\uz_{k+1;N}\ga$ is the unique element in $\g_k^*$ satisfying
\begin{equation}\label{eq:mardef}
\forall\phi \in\g_k, \qquad \ipp*{\phi,\int_{\Rd^{N-k}}d\uz_{k+1;N}\ga}_{\g_k-\g_k^*} \coloneqq \ipp*{\Sym_N\paren*{\phi\otimes 1^{\otimes N-k}}, \ga}_{\g_N-\g_N^*}.
\end{equation}
We additionally define the mapping $\int dz_{N+1, N} : \mathfrak{g}_N^* \rightarrow \mathfrak{g}_N^*$ to be the identity map. Sometimes, we use the alternative notation $\int_{\Rd^{N-k}}d\ga(\cdot,\uz_{k+1;N})$. Our duality result for the maps $\ep_{k,N}$ and $\int_{\Rd^{N-k}}d\uz_{k+1;N}$ is the following proposition.

\begin{proposition}\label{prop:eptrdual}
For $1 \leq k \leq N$, the map $\int dz_{k+1,N}: \g_N^*\rightarrow \g_{k}^*$ is a continuous linear operator and we have the operator equality
\begin{equation}
\epsilon_{k,N}^*  = \int d\uz_{k+1;N}.
\end{equation}
\end{proposition}
\begin{proof}
Let $\ga \in \mathfrak{g}_N^*$ and $\phi \in \mathfrak{g}_k$. We calculate 
\begin{equation}\label{eq:epstar}
\ipp*{\phi, \epsilon_{k,N}^*(\ga)}_{\g_k-\g_k^*} = \ipp*{\ep_{k,N}(\phi), \ga}_{\g_N-\g_N^*} = \ipp*{\frac{1}{|P_k^N|} \sum_{\mathbf{m}_k \in P_k^N} \phi_{\mathbf{m}_k}, \ga}_{\g_N-\g_N^*}
\end{equation}
and
\begin{align}
\ipp*{\phi,\int_{\Rd^{N-k}}d\uz_{k+1;N}\ga}_{\g_k-\g_k^*} &=\ipp*{\Sym_N\paren*{\phi \otimes 1^{\otimes N-k}}, \ga}_{\g_N-\g_N^*} \nn\\
&= \ipp*{\frac{1}{N!} \sum_{\sigma \in\Ss_N} (\phi \otimes 1^{\otimes (N-k)})_{\sigma}, \ga}_{\g_N-\g_N^*},\label{eq:pteval}
\end{align}
where $(\phi\otimes 1^{\otimes N-k})_{\sigma}(\uz_N) \coloneqq (\phi\otimes 1^{\otimes N-k})(z_{\sigma(1)},\ldots,z_{\sigma(N)})$. For each $k$-tuple $\mathbf{m}_k \in P_k^N$, define the set 
\begin{equation}
        A(\bm{m}_k) \coloneqq \{ \sigma \in \Ss_N : (\sigma(1),\ldots,\sigma(k)) = \mathbf{m}_k\},
    \end{equation}
which has cardinality $(N-k)!$. If $\sigma \in A(\mathbf{m}_k) \cap A(\mathbf{m}_k')$, then $\mathbf{m}_k' = (\sigma(1),\ldots,\sigma(k)) = \mathbf{m}_k$, which implies that the sets $A(\mathbf{m}_k)$ are pairwise disjoint. Given a permutation $\sigma \in \Ss_N$, set $\mathbf{m}_k = (\sigma(1),\ldots,\sigma(k))$. Then trivially, $\sigma \in A(\bm{m}_k)$ and we have shown the partition 
\begin{equation}
    \bigsqcup_{\mathbf{m}_k \in P_k^N} A(\mathbf{m}_k)  = \Ss_N. 
\end{equation}
Hence, we have 
\begin{align}
\sum_{\sigma \in \Ss_N } (\phi \otimes 1^{\otimes (N-k)})_\sigma  &= \sum_{\mathbf{m}_k \in P_k^N} \sum_{\sigma \in A(\mathbf{m}_k)}(\phi \otimes 1^{\otimes (N-k)})_\sigma \nn\\
&=  \sum_{\mathbf{m}_k \in P_k^N} \sum_{\sigma \in A(\mathbf{m}_k)} \phi_{\mathbf{m}_k} \nn\\
&=(N-k)! \sum_{\mathbf{m}_k \in P_k^N}  \phi_{\mathbf{m}_k}.
\end{align}
Recalling that $|P_k^N| = N!/(N-k)!$, we have shown that \eqref{eq:epstar} equals \eqref{eq:pteval}. This equality of operators proves the continuity of $\int_{\Rd^{N-k}} d\uz_{k+1;N}$ by \cref{prop:contadj}  and \cref{lem:epcont}.
\end{proof}

\begin{remark}
One may also obtain \cref{prop:eptrdual} stated below as a corollary of \Cref{prop:LAhomsum,prop:marPo}; but we feel the argument presented above is more direct.
\end{remark}

\subsection{Lie-Poisson space $\G_N^*$ of $N$-hierarchies of states}\label{ssec:NgeomLP}
We turn to proving \cref{prop:GNLP}, which asserts that there is a well-defined Lie-Poisson structure on the dual space $\G_N^*$ of $N$-hierarchies of states. Later, in \cref{ssec:HamBBGKY}, we will use this Lie-Poisson structure to demonstrate a Hamiltonian formulation of the BBGKY hierarchy.

We start with a technical lemma, which is a straightforward consequence of \cref{lem:gkremon}.

\begin{lemma}\label{lem:GNrefl}
For each $N \in \N$, the space $\mathfrak{G}_N$ is reflexive.
\end{lemma}

\begin{proof}
Since each $\mathfrak{g}_k$ is reflexive by \cref{lem:gkremon}, $\mathfrak{G}_N$ is also reflexive since, by using once again  \cite[Proposition 2 \S 14, Chapter 3]{Kothe1969I}, we have the chain of isomorphisms 
\begin{equation}
\mathfrak{G}_N^{**} = \left(\bigoplus_{k=1}^N \mathfrak{g}_k \right)^{**} \cong \left( \prod_{k=1}^N \mathfrak{g}_k^* \right)^{*} \cong \bigoplus_{k=1}^N \mathfrak{g}_k^{**} \cong \bigoplus_{k=1}^N \mathfrak{g}_k \cong \mathfrak{G}_N.
\end{equation}
\end{proof}

\begin{proof}[Proof of \cref{prop:GNLP}]
The proof is similar to that of \cref{prop:gkLP} for $\g_k^*$, except now we will apply \cref{thm:Glockner} with $E = \mathfrak{G}_N^*$. First, note that by \cref{lem:GNrefl}, \cref{prop:GNLA}, and \eqref{eq:brid}, the bracket on $E^* = \mathfrak{G}_N^{**} $ is continuous, \textit{a fortiori} hypocontinuous. It therefore suffices to show that $E$ is a $k^\infty$-space in the sense of \cref{def:kinfty}. To show this, note that for $n \in \N$,
\begin{equation}
    (\mathfrak{G}_N^*)^n \cong \prod_{j=1}^n \prod_{k=1}^N \mathfrak{g}_k^*,
\end{equation}
which shows that $(\mathfrak{G}_N^*)^n$ is a finite product of (DF) Montel spaces $\g_k^*$ (recall \cref{lem:gkremon}), hence $(DF)$ Montel itself. Therefore, \cref{thm:Webb} implies that $(\mathfrak{G}_N^*)^n$ is a sequential space, which in turn implies that $(\mathfrak{G}_N^*)^n$ is a $k$-space by \cref{prop:seqk}. Since $n \in \N$ was arbitrary, it follows that $\mathfrak{G}_N^*$ is a $k^\infty$-space. 
\end{proof}

Abstractly, \cref{prop:GNLP} tells us that for any functional $\Gc\in \Cc^\infty(\G_N^*)$, there exists a unique Hamiltonian vector field
$X_{\Gc} \in \Cc^\infty(\G_N^*, \G_N^*)$ characterized by the property that
\begin{equation}
\forall \Fc\in \Cc^\infty(\G_N^*), \qquad \Big(X_{\Gc}(\Fc)\Big)(\Ga_N)  = \ds\Fc[\Ga_N](X_{\Gc}(\Ga_N)) = \pb{\Gc}{\Fc}_{\G_N^*}(\Ga_N).
\end{equation}
For applications, in particular as it pertains to the BBGKY and Vlasov hierarchies (see \Cref{ssec:HamBBGKY,ssec:HamVH}) and evaluating limits as $N\rightarrow\infty$ (see \cref{ssec:GinfLA}), it is useful to have an explicit formula for the Hamiltonian vector field $X_{\Gc}$. We provide such a formula with the next proposition. 

\begin{proposition}\label{prop:NhamVF}
If $\Gc\in \mathcal{C}^\infty(\G_N^*)$, then we have the following formula for the Hamiltonian vector field $X_{\Gc}$ with respect to the bracket $\pb{\cdot}{\cdot}_{\G_N^*}$: for $1\leq \ell\leq N$ and any $\Ga=(\ga^{(k)})_{k=1}^N\in \G_N^*$,
\begin{equation}\label{eq:NhamVF}
X_{\Gc}(\Ga)^{(\ell)} = \sum_{j=1}^N \sum_{r=r_0}^{\min(\ell,j)}C_{\ell j Nr}{j\choose r}\int_{(\R^{2d})^{k-\ell}}d\uz_{\ell+1;k}\pb{\sum_{\ab_r\in P_r^\ell} \ds\Gc[\Ga]_{(\ab_r,\ell+1,\ldots,\ell+j-r)}^{(j)}}{\gamma^{(k)}}_{(\R^{2d})^k},
\end{equation}
where $C_{\ell j Nr}\coloneqq  \frac{(N-\ell)!(N-j)!}{(N-1)!(N-\ell-j+r)!}$,  $k \coloneqq \min ( \ell + j -1 , N)$, and $r_0 \coloneqq \max (1, \ell + j -N)$.
\end{proposition}

Before proceeding to the proof of \cref{prop:NhamVF}, some explanation regarding the well-definedness of the expression \eqref{eq:NhamVF} is in order. First, thanks to the identification $\g_j\cong \g_j^{**}$, each $\ds\Gc[\Ga]^{(j)} \in\g_j$ is a symmetric element of the test function space $\mathcal{C}^\infty((\R^{2d})^j)$. Thus, the symmetrized function $\sum_{\ab_r\in P_r^\ell} \ds\Gc[\Ga]_{(\ab_r,\ell+1,\ldots,\ell+j-r)}^{(j)}$ is an element of $\g_k$, i.e. a symmetric element of $\mathcal{C}^\infty((\R^{2d})^k)$. Although $\ga^{(k)}\in \g_k^*$, i.e. a symmetric distribution on $\mathcal{C}^\infty((\R^{2d})^k)$, is not a function, the usual rules for multiplication of a distribution by a test function and differentiation of a distribution show that the Poisson bracket above is again a well-defined element of $\g_k^*$. Hence by \cref{prop:eptrdual}, we can take its $\ell$-particle marginal, which is symbolically denoted in \eqref{eq:NhamVF} by the integration over the last $k-\ell$ coordinates, to obtain an element of $\g_\ell^*$. Thus, the right-hand side in \eqref{eq:NhamVF} is nothing but a linear combination of elements in $\g_\ell^*$, hence itself an element of $\g_\ell^*$. With these clarifications, we turn to the proof of \cref{prop:NhamVF}.

\begin{remark}
There are some exceptional cases concerning our notation in the right-hand side of \eqref{eq:NhamVF}, which, out of convenience, we do not separate out. When $k=\ell$ the integration is vacuous. When $j=1$ (and therefore $r_0=1$), the tuple $(\ab_r,\ell+1,\ldots,\ell+j-r)$ should be replaced by $\ab_1$. When $j=2$, the tuple $(\ab_r,\ell+1,\ldots,\ell+j-r)$ should be replaced by $(\ab_1,\ell+1)$ if $r=1$ and $\ab_2$ if $r=2$.
\end{remark}

\begin{proof}[Proof of \cref{prop:NhamVF}]
To increase the transparency of our computations, it is convenient to use integral notation, instead of distributional pairings, throughout the proof. By definition of the $\G_N^*$ Poisson bracket \eqref{eq:GN*LPdef},
\begin{equation}
\pb{\Fc}{\Gc}_{\G_N^*}(\Ga) = \sum_{k=1}^N \int_{(\R^{2d})^k}d\ga^{(k)}(\uz_k) \comm{\ds\Fc[\Ga]}{\ds\Gc[\Ga]}_{\G_N}^{(k)}(\uz_k).
\end{equation}
By the formula \eqref{eq:GNLBform} for the $\G_N$ Lie bracket,
\begin{equation}
\comm{\ds\Fc[\Ga]}{\ds\Gc[\Ga]}_{\G_N}^{(k)} = \sum_{\substack{1\leq \ell,j\leq N \\ \min(\ell+j-1,N)=k}} \sum_{r=r_0}^{\min(\ell,j)} C_{\ell j Nr} \ep_{\ell+j-r,k}\paren*{\Sym_{\ell+j-r}\paren*{\ds\Fc[\Ga]^{(\ell)} \wedge_r \ds\Gc[\Ga]^{(j)}}},
\end{equation}
where $C_{\ell jNr}$ is as above. To compactify the notation, let us set $f^{(\ell)} \coloneqq \ds\Fc[\Ga]^{(\ell)}$ and $g^{(j)}\coloneqq \ds\Gc[\Ga]^{(j)}$, the dependence on $\Ga$ being implicit. By definition of $\ep_{\ell+j-r,k}$ and using the invariance of the summation $\sum_{\bm{p}_{\ell+j-r}\in P_{\ell+j-r}^k}$ under the $\Ss_{\ell+j-r}$ action, we have that
\begin{align}
&\ep_{\ell+j-r,k}\paren*{\Sym_{\ell+j-r}\paren*{f^{(\ell)}\wedge_r g^{(j)}}} \nonumber\\
&= \frac{1}{|P_{\ell+j-r}^{k}|} \sum_{\bp_{\ell+j-r}\in P_{\ell+j-r}^{k}} \Big(f^{(\ell)}\wedge_r g^{(j)}\Big)_{\bm{p}_{\ell+j-r}} \nonumber\\
&=\frac{1}{|P_{\ell+j-r}^{k}|} \sum_{\bp_{\ell+j-r}\in P_{\ell+j-r}^{k}}{\ell\choose r}{j\choose r}{r!}\sum_{i=1}^r \Bigg(\nabla_{x_{p_i}}f_{\bp_\ell}^{(\ell)}\cdot \nabla_{v_{p_i}}g^{(j)}_{(\bp_r,\bp_{\ell+1;\ell+j-r})}  \nonumber\\
&\quad - \nabla_{x_{p_i}}g_{\bp_j}^{(j)}\cdot\nabla_{v_{p_i}}f_{(\bp_r,\bp_{j+1;j+\ell-r})}^{(\ell)}\Bigg),
\end{align}
where the ultimate line follows from the definition \eqref{r wedge} of $\wedge_r$. Thus, setting $C_{\ell j N rk} \coloneqq \frac{C_{\ell j Nr}{\ell\choose r}{j\choose r}{r!}}{|P_{\ell+j-r}^{k}|}$, we arrive at the identity
\begin{multline}\label{eq:rhssumk}
\pb{\Fc}{\Gc}_{\G_N^*}(\Ga) = \sum_{k=1}^N \sum_{\substack{1\leq \ell,j\leq N \\ \min(\ell+j-1,N)=k}}\sum_{r=r_0}^{\min(\ell,j)}C_{\ell jNrk}\sum_{\bp_{\ell+j-r}\in P_{\ell+j-r}^{k}}\sum_{i=1}^r\Bigg(\int_{(\R^{2d})^k}d\ga^{(k)}(\uz_k) \\
\Big(\nabla_{x_{p_i}}f_{\bp_\ell}^{(\ell)}(\uz_{k})\cdot \nabla_{v_{p_i}}g_{(\bp_r,\bp_{\ell+1;\ell+j-r})}^{(j)}(\uz_k) -\nabla_{x_{p_i}}g_{\bp_j}^{(j)}(\uz_k)\cdot \nabla_{v_{p_i}}f_{(\bp_r,\bp_{j+1;j+\ell-r})}^{(\ell)}(\uz_k) \Big)\Bigg).
\end{multline}
By introducing in the second term of the last line the relabeling $\bm{q}_{\ell+j-r}$ of $\bp_{\ell+j-r}$ according to
\begin{equation}
\bm{q}_r \coloneqq \bp_r, \quad \bm{q}_{r+1;\ell} \coloneqq \bp_{j+1;j+\ell-r}, \quad \bm{q}_{\ell+1;\ell+j-r} \coloneqq \bp_{r+1;j},
\end{equation}
we see that
\begin{multline}
\sum_{\bp_{\ell+j-r}\in P_{\ell+j-r}^{k}}\sum_{i=1}^r\int_{(\R^{2d})^k}d\ga^{(k)}(\uz_k) \Big(\nabla_{x_{p_i}}f_{\bp_\ell}^{(\ell)}(\uz_{k})\cdot \nabla_{v_{p_i}}g_{(\bp_r,\bp_{\ell+1;\ell+j-r})}^{(j)}(\uz_k) \\
-\nabla_{x_{p_i}}g_{\bp_j}^{(j)}(\uz_k)\cdot \nabla_{v_{p_i}}f_{(\bp_r,\bp_{j+1;j+\ell-r})}^{(\ell)}(\uz_k) \Big)\\
= \sum_{\bp_{\ell+j-r}\in P_{\ell+j-r}^{k}}\sum_{i=1}^r\int_{(\R^{2d})^k}d\ga^{(k)}(\uz_k) \Big(\nabla_{x_{p_i}}f_{\bp_\ell}^{(\ell)}(\uz_{k})\cdot \nabla_{v_{p_i}}g_{(\bp_r,\bp_{\ell+1;\ell+j-r})}^{(j)}(\uz_k) \\
-\nabla_{x_{p_i}}g_{(\bp_r,\bp_{\ell+1;\ell+j-r})}^{(j)}(\uz_k)\cdot \nabla_{v_{p_i}}f_{\bp_\ell}^{(\ell)}(\uz_k)\Big).
\end{multline}

Integrating by parts (in the distributional sense) the $\nabla_{x_{p_i}}$, we have that
\begin{multline}
\int_{(\R^{2d})^k} d\ga^{(k)}(\uz_k)\nabla_{x_{p_i}}f_{\bp_\ell}^{(\ell)}(\uz_k)\cdot \nabla_{v_{p_i}}g_{(\bp_r,\bp_{\ell+1;\ell+j-r})}^{(j)}(\uz_k) \\
= - \int_{(\R^{2d})^k}d\ga^{(k)}(\uz_k) f_{\bp_\ell}^{(\ell)}(\uz_k) \div_{x_{p_i}}\nabla_{v_{p_i}}g_{(\bp_r,\bp_{\ell+1;\ell+j-r})}^{(j)}(\uz_k) \\
-\int_{(\R^{2d})^k}d\nabla_{x_{p_i}}\ga^{(k)}(\uz_k)\cdot f_{\bp_\ell}^{(\ell)}(\uz_k)\nabla_{v_{p_i}} g_{(\bp_r,\bp_{\ell+1;\ell+j-r})}^{(j)}(\uz_k).
\end{multline}
Similarly integrating by parts the $\nabla_{v_{p_i}}$,
\begin{multline}
-\int_{(\R^{2d})^k} d\ga^{(k)}(\uz_k)\nabla_{v_{p_i}}f_{\bp_\ell}^{(\ell)}(\uz_k)\cdot \nabla_{x_{p_i}}g_{(\bp_r,\bp_{\ell+1;\ell+j-r})}^{(j)}(\uz_k) \\
= \int_{(\R^{2d})^k} d\ga^{(k)}(\uz_k)f_{\bp_\ell}^{(\ell)}(\uz_k)\div_{v_{p_i}}\nabla_{x_{p_i}}g_{(\bp_r,\bp_{\ell+1;\ell+j-r})}^{(j)}(\uz_k)\\
+ \int_{(\R^{2d})^k} d\nabla_{v_{p_i}}\ga^{(k)}(\uz_k) \cdot f_{\bp_\ell}^{(\ell)}(\uz_k) \nabla_{x_{p_i}}g_{(\bp_r,\bp_{\ell+1;\ell+j-r})}^{(j)}(\uz_k).
\end{multline}
Since $g^{(j)}$ is locally $\Cc^\infty$, we have the equality $\div_{x_{p_i}}\nabla_{v_{p_i}}g_{(\bp_r,\bp_{\ell+1;\ell+j-r})}^{(j)} = \div_{v_{p_i}}\nabla_{x_{p_i}}g_{(\bp_r,\bp_{\ell+1;\ell+j-r})}^{(j)}$. Therefore,
\begin{multline}
\int_{(\R^{2d})^k}d\ga^{(k)}(\uz_k) \Big(\nabla_{x_{p_i}}f_{\bp_\ell}^{(\ell)}(\uz_k)\cdot \nabla_{v_{p_i}}g_{(\bp_r,\bp_{\ell+1;\ell+j-r})}^{(j)}(\uz_k) - \nabla_{v_{p_i}}f_{\bp_\ell}^{(\ell)}(\uz_k)\cdot \nabla_{x_{p_i}}g_{(\bp_r,\bp_{\ell+1;\ell+j-r})}^{(j)}(\uz_k)\Big)\\
=-\int_{(\R^{2d})^k}d\nabla_{x_{p_i}}\ga^{(k)}(\uz_k)\cdot f_{\bp_\ell}^{(\ell)}(\uz_k)\nabla_{v_{p_i}} g_{(\bp_r,\bp_{\ell+1;\ell+j-r})}^{(j)}(\uz_k) \\
+ \int_{(\R^{2d})^k} d\nabla_{v_{p_i}}\ga^{(k)}(\uz_k) \cdot f_{\bp_\ell}^{(\ell)}(\uz_k) \cdot \nabla_{x_{p_i}}g_{(\bp_r,\bp_{\ell+1;\ell+j-r})}^{(j)}(\uz_k).
\end{multline}
Let us make the change of variable $\uz_k\mapsto \ul{w}_k$ according to
\begin{equation}
\ul{w}_k = (w_1,\ldots,w_k) = (z_{p_1},\ldots,z_{p_{\ell+j-r}}, z_{m_1},\ldots,z_{m_{k-(\ell+j-r)}}),
\end{equation}
where $m_1<\cdots<m_{k-(\ell+j-r)}$ is the increasing ordering of the set $\{1,\ldots,k\} \setminus\{p_1,\ldots,p_{\ell+j-r}\}$. Write $w_i=(y_i,u_i)$. Since $\ga^{(k)}$ is symmetric with respect to exchange of particle labels,
\begin{multline}
-\int_{(\R^{2d})^k}d\nabla_{x_{p_i}}\ga^{(k)}(\uz_k)\cdot f_{\bp_\ell}^{(\ell)}(\uz_k)\nabla_{v_{p_i}} g_{(\bp_r,\bp_{\ell+1;\ell+j-r})}^{(j)}(\uz_k) \\
+ \int_{(\R^{2d})^k} d\nabla_{v_{p_i}}\ga^{(k)}(\uz_k) \cdot f_{\bp_\ell}^{(\ell)}(\uz_k) \nabla_{x_{p_i}}g_{(\bp_r,\bp_{\ell+1;\ell+j-r})}^{(j)}(\uz_k) \\
= - \int_{(\R^{2d})^k}d\nabla_{y_i}\ga^{(k)}(\ul{w}_k)\cdot f_{1;\ell}^{(\ell)}(\ul{w}_k)\nabla_{u_i}g_{(1;r, \ell+1;\ell+j-r)}^{(j)}(\ul{w}_k) \\
+ \int_{(\R^{2d})^k} d\nabla_{u_i}\ga^{(k)}(\ul{w}_k)\cdot f_{1;\ell}^{(\ell)}(\ul{w}_k)\nabla_{y_i}g_{(1;r,\ell+1;\ell+j-r)}^{(j)}(\ul{w}_k).
\end{multline}

Next, since we have the partition
\begin{equation}
\bigsqcup_{k=1}^N \{(\ell,j)\in\{1,\ldots,N\}^2 : \min(\ell+j-1,N) = k\} = \{1,\ldots,N\}^2,
\end{equation}
we can interchange the order of summation over $k$ and summation over $\ell,j$ in the right-hand side of \eqref{eq:rhssumk} to obtain
\begin{multline}
\pb{\Fc}{\Gc}_{\G_N^*}(\Ga) = \sum_{\ell=1}^N \sum_{j=1}^N C_{\ell j Nr}\sum_{r=r_0}^{\min(\ell,j)}{\ell\choose r}{j\choose r}r!\sum_{i=1}^r\Bigg(\\
- \int_{(\R^{2d})^{k}}d\nabla_{y_i}\ga^{(k)}(\ul{w}_k)\cdot f_{1;\ell}^{(\ell)}(\ul{w}_k)\nabla_{u_i}g_{(1;r,\ell+1;\ell+j-r)}^{(j)}(\ul{w}_k) \\
+ \int_{(\R^{2d})^{k}} d\nabla_{u_i}\ga^{(k)}(\uz_k)\cdot f_{1;\ell}^{(\ell)}(\ul{w}_k)\nabla_{y_i}g_{(1;r,\ell+1;\ell+j-r)}^{(j)}(\ul{w}_k)\Bigg).
\end{multline}
By the distributional Fubini-Tonelli theorem, the sum of the integrals in the right-hand side may be re-expressed as
\begin{multline}
\int_{(\R^{2d})^{\ell}}d\ul{w}_\ell f^{(\ell)}(\ul{w}_\ell)\Bigg(\int_{(\R^{2d})^{k-\ell}} d\ul{w}_{\ell+1;k}\Big(\nabla_{u_i}\ga^{(k)}(\uw_k)\cdot\nabla_{y_i}g_{(1;r,\ell+1;\ell+j-r)}^{(j)}(\ul{w}_k)\\
-\nabla_{y_i}\ga^{(k)}(\ul{w}_k)\cdot \nabla_{u_i}g_{(1;r,\ell+1;\ell+j-r)}^{(j)}(\ul{w}_k)\Big) \Bigg),
\end{multline}
where the inner integral should be understood as the $\ell$-particle marginal (recall \cref{prop:eptrdual}) of the distribution given by the integrand. Going forward, we return to the original $z,x,v$ notation.

We claim that we can rewrite the preceding right-hand side more concisely in terms of the canonical Poisson bracket $\pb{\cdot}{\cdot}_{\Rd^k}$ on $(\R^{2d})^k$. Indeed, consider the expression
\begin{multline}
\int_{(\R^{2d})^{k-\ell}}d\uz_{\ell+1;k}\pb{\sum_{\ab_r\in P_r^\ell} g_{(\ab_r,\ell+1,\ldots,\ell+j-r)}^{(j)}}{\gamma^{(k)}}_{(\R^{2d})^k} \\
=\sum_{\ab_r\in P_r^\ell}\sum_{\beta=1}^k \int_{(\R^{2d})^{k-\ell}}d\uz_{\ell+1;k}\Big(\nabla_{x_\beta}g_{(\ab_r,\ell+1,\ldots,\ell+j-r)}^{(j)}\cdot\nabla_{v_\beta}\gamma^{(k)} - \nabla_{v_\beta}g_{(\ab_r,\ell+1,\ldots,\ell+j-r)}^{(j)}\cdot\nabla_{x_\beta}\gamma^{(k)}\Big).
\end{multline}
If $\beta\in\{\ell+1,\ldots,k\}$, then integrating by parts to swap $\nabla_{x_\beta}$ and $\nabla_{v_{\beta}}$, we see that the resulting summand is zero. Therefore, only coordinates $\beta\in\{1,\ldots,\ell\}$ produce a nonzero contribution. Similarly, for $\beta\in\{1,\ldots,\ell\} \setminus\{a_1,\ldots,a_r\}$,
\begin{equation}
\nabla_{x_\beta}g_{(\ab_r,\ell+1,\ldots,\ell+j-r)}^{(j)} = \nabla_{v_\beta}g_{(\ab_r,\ell+1,\ldots,\ell+j-r)}^{(j)} = 0.
\end{equation}
Therefore, by relabeling the sum over $\beta$, we have the equality
\begin{multline}
\sum_{\beta=1}^k \int_{(\R^{2d})^{k-\ell}}d\uz_{\ell+1;k}\Big(\nabla_{x_\beta}g_{(\ab_r,\ell+1,\ldots,\ell+j-r)}^{(j)}\cdot\nabla_{v_\beta}\gamma^{(k)} - \nabla_{v_\beta}g_{(\ab_r,\ell+1,\ldots,\ell+j-r)}^{(j)}\cdot\nabla_{x_\beta}\gamma^{(k)}\Big)\\
=\sum_{i=1}^r \int_{(\R^{2d})^{k-\ell}}d\uz_{\ell+1;k}\Big(\nabla_{x_{a_i}}g_{(\ab_r,\ell+1,\ldots,\ell+j-r)}^{(j)}\cdot\nabla_{v_{a_i}}\gamma^{(k)} - \nabla_{v_{a_i}}g_{(\ab_r,\ell+1,\ldots,\ell+j-r)}^{(j)}\cdot\nabla_{x_{a_i}}\gamma^{(k)}\Big).
\end{multline}
We claim that
\begin{multline}\label{eq:CoVpre}
\sum_{\ab_r\in P_r^\ell}\sum_{i=1}^r \int_{(\R^{2d})^\ell}d\uz_\ell f^{(\ell)}(\uz_{\ell})\int_{(\R^{2d})^{k-\ell}}d\uz_{\ell+1;k}\Big(\nabla_{x_{a_i}}g_{(\ab_r,\ell+1,\ldots,\ell+j-r)}^{(j)}\cdot\nabla_{v_{a_i}}\gamma^{(k)} \\
- \nabla_{v_{a_i}}g_{(\ab_r,\ell+1,\ldots,\ell+j-r)}^{(j)}\cdot\nabla_{x_{a_i}}\gamma^{(k)}\Big)(\uz_k)\\
={\ell\choose r}r! \sum_{i=1}^r\int_{(\R^{2d})^\ell}d\uz_\ell f^{(\ell)}(\uz_{\ell})\int_{(\R^{2d})^{k-\ell}}d\uz_{\ell+1;k}\Big(\nabla_{x_{i}}g_{(1,\ldots,r,\ell+1,\ldots,\ell+j-r)}^{(j)}\cdot\nabla_{v_{i}}\gamma^{(k)}\\
 - \nabla_{v_{i}}g_{(1,\ldots,r,\ell+1,\ldots,\ell+j-r)}^{(j)}\cdot\nabla_{x_{i}}\gamma^{(k)}\Big)(\uz_k).
\end{multline}
Indeed, let $m_1<\cdots<m_{\ell-r}$ denote the increasing ordering of the set $\{1,\ldots,\ell\}\setminus\{a_1,\ldots,a_r\}$, and consider the change of variable
\begin{equation}
\ul{w}_k = (w_1,\ldots,w_k) = (z_{a_1},\ldots,z_{a_r}, z_{m_1},\ldots,z_{m_{\ell-r}},z_{\ell+1},\ldots,z_{k}).
\end{equation}
Then since $f^{(\ell)}$ is symmetric with respect to permutation of particle labels,
\begin{equation}
f^{(\ell)}(z_1,\ldots,z_\ell) = f^{(\ell)}(w_1,\ldots,w_\ell).
\end{equation}
Similarly,
\begin{align}
\nabla_{v_{a_i}}\gamma^{(k)}(\uz_k) &= \nabla_{u_i}\gamma^{(k)}(\ul{w}_k),\\
\nabla_{x_{a_i}}\gamma^{(k)}(\uz_k) &= \nabla_{y_i}\gamma^{(k)}(\ul{w}_k), \\
\nabla_{x_{a_i}}g_{(\ab_r,\ell+1,\ldots,\ell+j-r)}^{(j)}(\uz_k) &= \nabla_{y_i}g_{(1,\ldots,r,\ell+1,\ldots,\ell+j-r)}^{(j)}(\ul{w}_k),\\
\nabla_{v_{a_i}}g_{(\ab_r,\ell+1,\ldots,\ell+j-r)}^{(j)}(\uz_k) &= \nabla_{u_i}g_{(1,\ldots,r,\ell+1,\ldots,\ell+j-r)}^{(j)}(\ul{w}_k).
\end{align}
Making the change of variable $\uz_k\mapsto \ul{w}_k$ in the first two lines of \eqref{eq:CoVpre} and recalling that $|P_{r}^\ell| = {\ell\choose r}r!$, we arrive at the desired conclusion. Therefore, we have shown that
\begin{multline}
\pb{\Fc}{\Gc}_{\G_N^*}(\Ga) = \sum_{\ell=1}^N \sum_{j=1}^N C_{\ell j Nr}\sum_{r=r_0}^{\min(\ell,j)}{j\choose r}\int_{(\R^{2d})^\ell}d\uz_\ell f^{(\ell)}(\uz_\ell)\\
\int_{(\R^{2d})^{k-\ell}}d\uz_{\ell+1;k}\pb{\sum_{\ab_r\in P_r^\ell} g_{(\ab_r,\ell+1,\ldots,\ell+j-r)}^{(j)}}{\gamma^{(k)}}_{(\R^{2d})^k}(\uz_k),
\end{multline}
which, upon recalling that $f^{(\ell)} = \ds\Fc[\Ga]^{(\ell)}$ and $g^{(j)} = \ds\Gc[\Ga]^{(j)}$, implies that
\begin{multline}
X_{\Gc}(\Ga)^{(\ell)} = \sum_{j=1}^N C_{\ell j Nr}\sum_{r=r_0}^{\min(\ell,j)}{j\choose r}\int_{(\R^{2d})^{k-\ell}}d\uz_{\ell+1;k}\pb{\sum_{\ab_r\in P_r^\ell} \ds\Gc[\Ga]_{(\ab_r,\ell+1,\ldots,\ell+j-r)}^{(j)}}{\gamma^{(k)}}_{(\R^{2d})^k},
\end{multline}
completing the proof of the lemma.
\end{proof}

\subsection{Marginals}\label{ssec:Ngeommor}
We close this section with the observation that the operation of forming an $N$-hierarchy of marginals from an $N$-particle distribution defines a map which is a Poisson morphism. The material in this subsection closely parallels that of \cite[Section 5.3]{MNPRS2020} for the quantum setting, as the underlying algebraic structure is the same. Therefore, we shall be brief in our remarks.

\cref{prop:LAhomsum} below states that there is a linear homomorphism of Lie algebras $\G_{N}\rightarrow \g_{N}$ induced by the embeddings $\{\epsilon_{k,N}\}_{k=1}^N$. The proof carries over \textit{verbatim} from \cite[Proposition 5.28]{MNPRS2020}.

\begin{proposition}\label{prop:LAhomsum}
For any $N\in \N$, the map $\iota_{\ep}: \G_{N} \rightarrow \g_{N}$ defined by
\begin{equation}
\forall F=(f^{(k)})_{k=1}^N\in\G_N, \qquad \iota_{\ep}(F) \coloneqq \sum_{k=1}^{N} \epsilon_{k,N}(f^{(k)}),
\end{equation}
is a continuous linear homomorphism of Lie algebras.
\end{proposition}

The dual of a Lie algebra homomorphism is automatically a Poisson morphism between the induced Lie-Poisson structures \cite[Proposition 10.7.2]{MR2013}.\footnote{The dual of a Lie algebra homomorphism is, in fact, a \emph{momentum map} (also called \emph{moment map}) and therefore is a Poisson morphism (see \cite{GS1980,MRW1984}).} Therefore, by showing that the map $\iota_{mar} : \g_N^* \rightarrow \G_N^*$ defined by
\begin{equation}\label{homsum formula}
\forall \ga\in\g_N^*, \ k\in\{1,\ldots,N\}, \qquad \iota_{mar}(\ga)^{(k)} \coloneqq \int_{(\R^{2d})^{N-k}}d\uz_{k+1;N}\ga
\end{equation}
is the dual of the map $\iota_{\ep}$, it follows that $\iota_{mar}$ is a Poisson morphism. The proof is completely analogous to that of \cite[Proposition 5.29]{MNPRS2020}, replacing the trace pairing $i\Tr_{1,\ldots,N}(\cdot)$ by the duality pairing $\ipp{\cdot,\cdot}_{\G_N-\G_N^*}$, therefore we omit the details. 

\begin{proposition}\label{prop:marPo}
The map $\iota_{mar}: \g_{N}^{*} \rightarrow \G_{N}^{*}$ defined above is a morphism of Poisson vector spaces in the sense of \cref{def:PM}.
\end{proposition}

\section{$\infty$-particle geometric structure}\label{sec:infgeom}
We now turn to the geometric structure at the infinite-particle level and proving the results announced in \cref{ssec:Outinf}.

\subsection{The Lie algebra $\G_\infty$ of observable $\infty$-hierarchies}\label{ssec:GinfLA}
We recall from \eqref{eq:Ginfdef} that $\G_\infty= \bigoplus_{k=1}^\infty \g_k$ equipped with the locally convex direct sum topology, and we can identify $\G_N$ as a subspace via inclusion. We also recall that any element $F\in \G_\infty$ belongs to $\G_N$ for all $N$ sufficiently large. The goal of this subsection is to prove \cref{prop:GinfLA}, asserting that for any $F,G\in \G_\infty$, $\lim_{N\rightarrow\infty}\comm{F}{G}_{\G_N} \eqqcolon \comm{F}{G}_{\G_\infty}$ exists and defines a Lie bracket for $\G_\infty$. In contrast to the quantum setting (see \cite[Section 6.2]{MNPRS2020}), the limit $\G_\infty$ of the spaces $\G_N$ is large enough to contain all $\infty$-hierarchies of observables of interest. This is a technical advantage of the classical setting vs. the quantum setting.

\begin{proof}[Proof of \cref{prop:GinfLA}]
We first show the limit \eqref{eq:GinfLB}. Fix $F=(f^{(\ell)})_{\ell=1}^\infty, G=(g^{(j)})_{j=1}^\infty\in\G_\infty$ and $k\in\N$. Let $M_0 \in \N$ be such that $f^{(\ell)}=g^{(j)}=0$ for $\min(\ell,j)>M_0$. Note that for $k=\ell+j-1\geq 2M_0$, we must have $\max(\ell,j)>M_0$, implying
\begin{equation}
f^{(\ell)}\wedge_r g^{(j)} = 0, \qquad 1\leq r\leq \min(\ell,j).
\end{equation}
For any $N\geq M_0$, $F,G$ can be identified as elements of $\G_N$ by projection onto the first $N$ components, without any loss of information. For $1\leq k\leq N$, the formula \eqref{eq:GNLBform} gives
\begin{align}\label{eq:FGGN}
\comm{F}{G}_{\G_N}^{(k)} &=\sum_{\substack{\ell,j\geq 1 \\ \ell+j-1=k}}\sum_{r = r_0}^{\min( \ell,j)} C_{\ell jNr} \epsilon_{\ell+j - r, k}\paren*{\Sym_{\ell+j-r}\Big(f^{(\ell)} \wedge_r g^{(j)}\Big)},
\end{align}
where we recall that $C_{\ell jNr} = \frac{(N-\ell)!(N-j)!}{(N-1)! (N-\ell -j +r)!}$ and $r_0 =\max(1,\ell+j-N)$. Suppose $N\geq 2M_0$. Then if $\ell+j-N>1$, we have $\max(\ell,j)>M_0+1$, implying either $f^{(\ell)}=0$ or $g^{(j)}=0$, hence $\Sym_{\ell+j-r}(f^{(\ell)}\wedge_r g^{(j)})=0$. So there is no harm in starting the summation at $r=1$ in the right-hand side of \eqref{eq:FGGN}. More importantly, we note that for any $N\geq 2M_0$ and $2M_0\leq k\leq N$, we have $\comm{F}{G}_{\G_N}^{(k)} = 0$. By \cref{rem:Nconlim}, we have $\lim_{N\rightarrow\infty} C_{\ell j Nr} = \indic_{r=1}$. Using that there are only finitely many terms in the right-hand side of \eqref{eq:FGGN} independent of $N$, we compute, for fixed $k$, 
\begin{equation}
\lim_{N\rightarrow\infty} \comm{F}{G}_{\G_N}^{(k)} = \sum_{\substack{\ell,j\geq 1 \\ \ell+j-1=k}} \epsilon_{\ell+j-1,k}\paren*{\Sym_{\ell+j-1}\Big(f^{(\ell)} \wedge_1 g^{(j)}\Big)}.
\end{equation}
Observing that $\ell+j-1=k$ and therefore $\ep_{\ell+j-1,k}$ is the identity map $\g_k\rightarrow\g_k$, we arrive at \eqref{eq:GinfLB}.
 
Next, we turn to showing that $(\G_\infty,\comm{\cdot}{\cdot}_{\G_\infty})$ is a Lie algebra in the sense of \cref{def:LA}. This step, which is algebraic, consists of the verification that the bracket $\comm{\cdot}{\cdot}_{\G_\infty}$ satisfies properties \ref{LA1}-\ref{LA3}. The argument is essentially identical to that for the quantum case (cf. \cite[proof of Proposition 2.7]{MNPRS2020}), therefore we omit the details.

Finally, we turn to the analytic matter of showing that $\comm{\cdot}{\cdot}_{\G_\infty}$ is boundedly hypocontinuous, which is unique to this work. First, note that the sum defining $\comm{F}{G}_{\mathfrak{G}_\infty}^{(k)}$ is finite: for $k$ fixed, there are $k$ pairs $(\ell,j) \in \N^2$ satisfying $\ell+j-1=k$. Furthermore, our remarks at the beginning of the proof give that $\comm{F}{G}_{\G_\infty}^{(k)}=0$ for $k>2M_0$. Let $\mathcal{B} \subset \mathfrak{G}_\infty$ be bounded, and let $\mathsf{P}_{\ell}$ be the projection onto the $\ell$-th component of $\mathfrak{G}_{\infty}$. Then there exists an $N \in \N$ such that $\mathsf{P}_\ell (\mathcal{B}) = \{ 0 \}$ for $\ell > N$ and $\mathsf{P}_\ell(\mathcal{B})$ is bounded in $\mathfrak{g}_\ell$ for every $1\leq \ell \leq N$ (see \cite[(4), p. 213]{Kothe1969I}) . Equicontinuity of $\{ \comm{\cdot}{G}_{\mathfrak{G}_\infty} : G \in \mathcal{B} \}$ now follows from the the equicontinuity of $ \wedge_1$.
\end{proof}

\subsection{The weak Lie-Poisson space $\G_\infty^*$ of state $\infty$-hierarchies}\label{ssec:GinfLP}
The objective of this subsection is to prove \cref{prop:GinfLP}, asserting that the Lie bracket $\comm{\cdot}{\cdot}_{\G_\infty}$ constructed in \cref{prop:GinfLA} induces a well-defined weak Lie-Poisson structure on $\G_\infty^*$ in the sense of \cref{def:WPVS}, if we choose \eqref{eq:Ainfdef} as our unital subalgebra $\A_\infty\subset\Cc^\infty(\G_\infty^*)$,.

The reader will recall \cref{rem:cgd} that any expectation in $\A_\infty$ has constant G\^ateaux derivative. We record below the following observation on the structure of elements of $\A_\infty$, which will be crucial to verification of the weak Poisson properties \ref{assWP1}-\ref{assWP3} from \cref{def:WPVS}.

\begin{remark}\label{rem:Ainfform}
By definition, any element $\Fc\in \A_\infty$ takes the form
\begin{equation}\label{eq:Ainfform}
\Fc = \sum_{m=0}^\infty C_m\sum_{a=0}^{n_m}\Fc_{m1a}\cdots \Fc_{mma}
\end{equation}
where, for each $m\in\N_0$, $n_m\in\N_0$, $\Fc_{m1a} = \ipp{F_{m1a},\cdot}_{\G_\infty-\G_\infty^*},\ldots,\Fc_{mma}=\ipp{F_{mma},\cdot}_{\G_\infty-\G_\infty^*}$ for some $F_{1ma},\ldots,F_{mma}\in\G_\infty$ and $(C_m)_{m=0}^\infty$ are real coefficients such that there exists $M\in\N$ for which $C_m = 0$ if $m>M$. In words, $\Fc$ is a linear combination of finite products of expectations. This observation will be quite useful in the sequel, as invocation of some form of linearity will allow us to verify certain identities under the assumption that $\Fc$ is just a finite product of expectations. 
\end{remark}

We break the proof of \cref{prop:GinfLP}, which entails the verification of the properties \ref{assWP1}-\ref{assWP3}, into a series of lemmas. We begin with the following technical lemma for the G\^ateaux derivative of $\pb{\cdot}{\cdot}_{\G_\infty^*}$.

\begin{lemma}\label{lem:trgen}
If $\Gc_1=\Gc_{1,1}\cdots\Gc_{1,n_1}$ and $\Gc_2 = \Gc_{2,1}\cdots\Gc_{2,n_2}$ are the product of $n_1$ and $n_2$ expectations in $\A_{\infty}$, respectively, then through the isomorphism $\G_\infty^{**} \cong \G_\infty$, the G\^ateaux derivative $\ds\pb{\Gc_1}{\Gc_2}_{\G_\infty^*}[\Gamma]$ at the point $\Ga\in\G_\infty^*$ may be identified with
\begin{equation}
\sum_{i_1=1}^{n_1}\sum_{i_2=1}^{n_2} \Big(\prod_{\substack{1\leq q\leq n_1 \\ q\neq i_1}} \Gc_{1,q}(\Ga)\Big)\Big(\prod_{\substack{1\leq q\leq n_2 \\ q\neq i_2}} \Gc_{2,q}(\Ga)\Big)\comm{\ds\Gc_{1,i_1}[0]}{\ds\Gc_{2,i_2}[0]}_{\G_\infty}^{(k)} \in\G_\infty.
\end{equation}
In particular, if $\Gc_1,\Gc_2$ are expectations, that is $\Gc_q(\Gamma) = \ipp{\ds\Gc_q[0],\Ga}_{\G_\infty-\G_\infty^*}$, then $\ds\pb{\Gc_1}{\Gc_2}_{\G_\infty^*}[\Gamma]$ may be identifed with the element
\begin{equation}
\comm{\ds\Gc_1[0]}{\ds\Gc_2[0]}_{\G_\infty} \in \G_\infty.
\end{equation}
\end{lemma}
\begin{proof}
Observe from the Leibniz rule for the G\^ateaux derivative and the bilinearity of the wedge product $\wedge_1$ (remember that for fixed $\Ga$, $\Gc_{1,q}(\Ga),\Gc_{2,q}(\Ga)\in\R$),
\begin{equation}
\ds\Gc_{1}[\Gamma]^{(\ell)} \wedge_1 \ds\Gc_{2}[\Gamma]^{(j)} = \sum_{i_1=1}^{n_1}\sum_{i_2=1}^{n_2} \Big(\prod_{\substack{1\leq q\leq n_1 \\ q\neq i_1}} \Gc_{1,q}(\Ga)\Big)\Big(\prod_{\substack{1\leq q\leq n_2 \\ q\neq i_2}} \Gc_{2,q}(\Ga)\Big) \Big(\ds\Gc_{1,i_1}[0]^{(\ell)} \wedge_1 \ds\Gc_{2,i_2}[0]^{(j)}\Big).
\end{equation}
Hence using \cref{prop:GinfLA} and the linearity of the $\Sym_k$ operator, we find that
\begin{align}
&\comm{\ds\Gc_{1}[\Gamma]}{\ds\Gc_{2}[\Gamma]}_{\G_\infty}^{(k)} \nn\\
&= \sum_{\substack{1\leq \ell,j\leq N \\ \min(\ell+j-1,N)=k}} \Sym_k\Big(\ds\Gc_1[\Gamma]^{(\ell)}\wedge_1 \ds\Gc_2[\Gamma]^{(j)}\Big)\nonumber\\
&=\sum_{i_1=1}^{n_1}\sum_{i_2=1}^{n_2} \Big(\prod_{\substack{1\leq q\leq n_1 \\ q\neq i_1}} \Gc_{1,q}(\Ga)\Big)\Big(\prod_{\substack{1\leq q\leq n_2 \\ q\neq i_2}} \Gc_{2,q}(\Ga)\Big)\sum_{\substack{1\leq \ell,j\leq N \\ \min(\ell+j-1,N)=k}} \Sym_k\Big(\ds\Gc_{1,i_1}[0]^{(\ell)} \wedge_1 \ds\Gc_{2,i_2}[0]^{(j)}\Big)  \nonumber\\
&=\sum_{i_1=1}^{n_1}\sum_{i_2=1}^{n_2} \Big(\prod_{\substack{1\leq q\leq n_1 \\ q\neq i_1}} \Gc_{1,q}(\Ga)\Big)\Big(\prod_{\substack{1\leq q\leq n_2 \\ q\neq i_2}} \Gc_{2,q}(\Ga)\Big)\comm{\ds\Gc_{1,i_1}[0]}{\ds\Gc_{2,i_2}[0]}_{\G_\infty}^{(k)},
\end{align}
where the ultimate equality follows from the definition of the $\G_\infty$ Lie bracket. This completes the proof.
\end{proof}

Recall from the previous subsection that if $M_0$ is the maximal nonzero component index for $F,G\in\G_\infty$, then $\comm{F}{G}_{\G_\infty}^{(k)} = 0$ for $k>2M_0$. Moreover, there are only $k$ terms in the sum defining $\comm{F}{G}_{\G_\infty}^{(k)}$. Since for any $\Fc,\Gc\in\Cc^\infty(\G_\infty^*)$, $\ds\Fc[\Ga],\ds\Gc[\Ga]\in \G_\infty^{**}\cong \G_\infty$, it follows that the Poisson bracket $\pb{\Fc}{\Gc}_{\G_\infty^*}$ is well-defined pointwise. The next lemma shows that for $\Fc,\Gc\in\A_\infty$, the bracket $\pb{\Fc}{\Gc}_{\G_\infty^*}$ in fact belongs to $\A_\infty$ and the pair $(\A_\infty,\pb{\cdot}{\cdot}_{\G_\infty^*})$ is a Lie algebra obeying the Leibniz rule.

\begin{lemma}\label{lem:GinfWP1}
The formula \eqref{eq:GinfPBdef} defines a map $\A_{\infty}\times\A_{\infty}\rightarrow \A_{\infty}$ which satisfies property \ref{assWP1} in \cref{def:WPVS}.
\end{lemma}
\begin{proof}
We first show that for $\Fc,\Gc\in\A_{\infty}$, one has $\pb{\Fc}{\Gc}_{\G_\infty^{*}}\in\A_{\infty}$. Recalling \cref{rem:Ainfform}, the Leibniz rule, bilinearity of the bracket $\comm{\cdot}{\cdot}_{\G_\infty}$, and the bilinearity of the duality pairing $\ipp{\cdot,\cdot}_{\G_\infty-\G_\infty^*}$ allow us to consider only the case where $\Fc=\Fc_{1}\cdots\Fc_{n},\Gc=\Gc_{1}\cdots\Gc_m$ are both finite products of expectations. Unpacking the definition of the Poisson bracket $\pb{\cdot}{\cdot}_{\G_\infty^*}$ and appealing to \cref{lem:trgen}, we find
\begin{multline}
\forall \Ga\in\G_\infty^*, \qquad \pb{\Fc}{\Gc}_{\G_\infty^{*}}(\Gamma) \\
= \sum_{i_1=1}^n\sum_{i_2=1}^m \Big(\prod_{\substack{1\leq j\leq n\\ j\neq i_1}}\Fc_j(\Ga) \Big)\Big(\prod_{\substack{1\leq j\leq m\\ j\neq i_2}}\Gc_j(\Ga) \Big) \ipp*{\comm{\ds\Fc_{i_1}[0]}{\ds\Gc_{i_2}[0]}_{\G_\infty}, \Ga}_{\G_\infty-\G_\infty^*}.
\end{multline}
So, we only need to show that for each pair of indices $1\leq i_1\leq n$ and $1\leq i_2\leq m$, the functional
\begin{equation}\label{eq:wtsexp}
\Ga\mapsto \ipp*{\comm{\ds\Fc_{i_1}[0]}{\ds\Gc_{i_2}[0]}_{\G_\infty},\Ga}_{\G_\infty-\G_\infty^*}
\end{equation}
defines an element of $\A_\infty$. But this follows from the fact that $\ds\Fc_{i_1}[0]$ and $\ds\Gc_{i_2}[0]$ are both identifiable as elements of $\G_\infty$, and therefore $\comm{\ds\Fc_{i_1}[0]}{\ds\Gc_{i_2}[0]}_{\G_\infty}\in\G_\infty$, implying \eqref{eq:wtsexp} is the expectation of $\comm{\ds\Fc_{i_1}[0]}{\ds\Gc_{i_2}[0]}_{\G_\infty}$.

Bilinearity and anti-symmetry of $\pb{\cdot}{\cdot}_{\G_\infty^*}$ are immediate from the bilinearity and anti-symmetry of $\comm{\cdot}{\cdot}_{\G_\infty}$ and the bilinearity of the duality pairing $\ipp{\cdot,\cdot}_{\G_\infty-\G_\infty^*}$, so it remains to verify the Jacobi identity. Let $\Fc,\Gc,\Hc\in \A_{\infty}$. As we argued above, it suffices to consider the case where
\begin{equation}
\Fc = \Fc_{1}\cdots\Fc_n, \quad \Gc = \Gc_1\cdots\Gc_m, \quad \Hc=\Hc_1\cdots\Hc_q
\end{equation}
are all finite products of expectations. By multiplying by the constant functional 1, we may assume without loss of generality that $n=m=q$.
Thus, we need to show that for every $\Gamma\in \G_\infty^{*}$,
\begin{align}\label{eq:JIwts}
0&=\pb{\Fc}{\pb{\Gc}{\Hc}_{\G_\infty^{*}}}_{\G_\infty^{*}}(\Gamma) + \pb{\Gc}{\pb{\Hc}{\Fc}_{\G_\infty^{*}}}_{\G_\infty^{*}}(\Gamma) + \pb{\Hc}{\pb{\Fc}{\Gc}_{\G_\infty^{*}}}_{\G_\infty^{*}}(\Gamma) \nonumber\\
&=\ipp*{\comm{\ds\Fc[\Gamma]}{\ds\pb{\Gc}{\Hc}_{\G_\infty^{*}}[\Gamma]}_{\G_\infty},\Ga}_{\G_\infty-\G_\infty^*} + \ipp*{\comm{\ds\Gc[\Gamma]}{\ds\pb{\Hc}{\Fc}_{\G_\infty^{*}}[\Gamma]}_{\G_\infty},\Gamma}_{\G_\infty-\G_\infty^*} \nonumber\\
&\phantom{=} +\ipp*{\comm{\ds\Hc[\Gamma]}{\ds\pb{\Fc}{\Gc}_{\G_\infty^{*}}[\Gamma]}_{\G_\infty},\Gamma}_{\G_\infty-\G_\infty^*},
\end{align}
which we do by direct computation.

First, since
\begin{equation}\label{eq:dsFc}
\ds\Fc[\Ga] = \sum_{i=1}^n \Big(\prod_{\substack{1\leq j\leq n \\ j\neq i}} \Fc_{j}(\Ga) \Big)\ds\Fc_{i}[0]
\end{equation}
(here we are implicitly using \cref{rem:cgd} for $\Fc_{i}$), it follows from the bilinearity of the duality pairing that
\begin{multline}\label{eq:FGHpre}
\ipp*{\comm{\ds\Fc[\Gamma]}{\ds\pb{\Gc}{\Hc}_{\G_\infty^{*}}[\Gamma]}_{\G_\infty},\Gamma}_{\G_\infty-\G_\infty^*}  \\
=  \sum_{i_1=1}^n \Big(\prod_{\substack{1\leq j\leq n \\ j\neq i_1}} \Fc_{j}(\Ga) \Big) \ipp*{ \comm{\ds\Fc_{i_1}[0]}{\ds\pb{\Gc}{\Hc}_{\G_\infty^*}[\Ga]}_{\G_\infty}, \Ga}_{\G_\infty-\G_\infty^*}.
\end{multline}
Since the identity \eqref{eq:dsFc} also holds with $\Fc$ replaced by $\Gc$ or $\Hc$ in both sides, we also have by \cref{lem:trgen} that
\begin{equation}\label{eq:pbGcHc}
\pb{\Gc}{\Hc}_{\G_\infty^*}(\Ga) = \sum_{i_2,i_3=1}^n \Big(\prod_{\substack{1\leq j\leq n \\ j\neq i_2}}\Gc_j(\Ga) \Big)\Big(\prod_{\substack{1\leq j\leq n \\ j\neq i_3}}\Hc_j(\Ga)\Big) \ipp*{\comm{\ds\Gc_{i_2}[0]}{\ds\Hc_{i_3}[0]}_{\G_\infty},\Ga}.
\end{equation}
For each pair $1\leq i_2,i_3\leq n$, define the expectation $\Ec_{i_2i_3}^{\Gc\Hc}(\Ga)\coloneqq \ipp*{\comm{\ds\Gc_{i_2}[0]}{\ds\Hc_{i_3}[0]}_{\G_\infty},\Ga}_{\G_\infty-\G_\infty^*}$. Evidently, $\ds\Ec_{i_2i_3}^{\Gc\Hc}[\Ga] = \comm{\ds\Gc_{i_2}[0]}{\ds\Hc_{i_3}[0]}_{\G_\infty}$. Therefore, it follows from an application of the Leibniz rule to the right-hand side of \eqref{eq:pbGcHc}, and using that the $\Gc_j,\Hc_j$ are expectations, that
\begin{multline}
\forall\Ga\in\G_\infty^*, \qquad \ds\pb{\Gc}{\Hc}_{\G_\infty^*}[\Ga] =\sum_{i_2,i_3=1}^n \sum_{\substack{1\leq p\leq n \\ p \neq i_2}} \Big(\prod_{\substack{1\leq j\leq n \\ j\neq i_2,p}} \Gc_j(\Ga)\Big)\Big(\prod_{\substack{1\leq j\leq n \\ j\neq i_3}} \Hc_j(\Ga)\Big) \Ec_{i_2i_3}^{\Gc\Hc}(\Ga)\ds\Gc_{p}[0] \\
+ \sum_{i_2,i_3=1}^n \sum_{\substack{1\leq p\leq n \\ p \neq i_3}} \Big(\prod_{\substack{1\leq j\leq n \\ j\neq i_2}} \Gc_j(\Ga)\Big)\Big(\prod_{\substack{1\leq j\leq n \\ j\neq i_3,p}} \Hc_j(\Ga)\Big) \Ec_{i_2i_3}^{\Gc\Hc}(\Ga)\ds\Hc_{p}[0] \\
+ \sum_{i_2,i_3=1}^n\Big(\prod_{\substack{1\leq j\leq n \\ j\neq i_2}} \Gc_j(\Ga)\Big)\Big(\prod_{\substack{1\leq j\leq n \\ j\neq i_3}} \Hc_j(\Ga)\Big)\comm{\ds\Gc_{i_2}[0]}{\ds\Hc_{i_3}[0]}_{\G_\infty}.
\end{multline}
Substituting this identity into the right-hand side of \eqref{eq:FGHpre} and using the bilinearity of the Lie bracket and the duality pairing, we obtain
\begin{multline}
\ipp*{\comm{\ds\Fc[\Gamma]}{\ds\pb{\Gc}{\Hc}_{\G_\infty^{*}}[\Gamma]}_{\G_\infty},\Gamma}_{\G_\infty-\G_\infty^*} \\
=\sum_{i_1,i_2,i_3=1}^n\sum_{\substack{1\leq p\leq n \\ p \neq i_2}}  \Big(\prod_{\substack{1\leq j\leq n \\ j\neq i_1}}\Fc_{j}(\Ga)\Big)\Big(\prod_{\substack{1\leq j\leq n \\ j\neq i_2,p}} \Gc_j(\Ga)\Big)\Big(\prod_{\substack{1\leq j\leq n \\ j\neq i_3}} \Hc_j(\Ga)\Big) \Ec_{i_2i_3}^{\Gc\Hc}(\Ga)\ipp*{\comm{\ds\Fc_{i_1}[0]}{\ds\Gc_{p}[0]}_{\G_\infty},\Ga}_{\G_\infty-\G_\infty^*} \\
+ \sum_{i_1,i_2,i_3=1}^n\sum_{\substack{1\leq p\leq n \\ p \neq i_3}} \Big(\prod_{\substack{1\leq j\leq n \\ j\neq i_1}}\Fc_{j}(\Ga)\Big) \Big(\prod_{\substack{1\leq j\leq n \\ j\neq i_2}} \Gc_j(\Ga)\Big)\Big(\prod_{\substack{1\leq j\leq n \\ j\neq i_3,p}} \Hc_j(\Ga)\Big) \Ec_{i_2i_3}^{\Gc\Hc}(\Ga)\ipp*{\comm{\ds\Fc_{i_1}[0]}{\ds\Hc_{p}[0] }_{\G_\infty},\Ga}_{\G_\infty-\G_\infty^*}\\
+ \sum_{i_1,i_2,i_3=1}^n\Big(\prod_{\substack{1\leq j\leq n \\ j\neq i_1}}\Fc_{j}(\Ga)\Big)\Big(\prod_{\substack{1\leq j\leq n \\ j\neq i_2}} \Gc_j(\Ga)\Big)\Big(\prod_{\substack{1\leq j\leq n \\ j\neq i_3}} \Hc_j(\Ga)\Big)\ipp*{\comm{\ds\Fc_{i_1}[0]}{\comm{\ds\Gc_{i_2}[0]}{\ds\Hc_{i_3}[0]}_{\G_\infty}}_{\G_\infty},\Ga}_{\G_\infty-\G_\infty^*}.
\end{multline}
To further compactify the notation, let us set $\Ec_{i_1p}^{\Fc\Gc}(\Ga)\coloneqq \ipp{\comm{\ds\Fc_{i_1}[0]}{\ds\Gc_{p}[0]}_{\G_\infty},\Ga}_{\G_\infty-\G_\infty^*}$ and similarly for $\Ec_{i_1p}^{\Fc\Hc}(\Ga)$. By the same reasoning, we also obtain
\begin{multline}
\ipp*{\comm{\ds\Gc[\Gamma]}{\ds\pb{\Hc}{\Fc}_{\G_\infty^{*}}[\Gamma]}_{\G_\infty},\Gamma}_{\G_\infty-\G_\infty^*}\\
=\sum_{i_1,i_2,i_3=1}^n \sum_{\substack{1\leq p\leq n \\ p \neq i_3}} \Big(\prod_{\substack{1\leq j\leq n \\ j\neq i_2}}\Gc_{j}(\Ga)\Big) \Big(\prod_{\substack{1\leq j\leq n \\ j\neq i_3,p}} \Hc_j(\Ga)\Big)\Big(\prod_{\substack{1\leq j\leq n \\ j\neq i_1}} \Fc_j(\Ga)\Big) \Ec_{i_3i_1}^{\Hc\Fc}(\Ga)\Ec_{i_2p}^{\Gc\Hc}(\Ga)\\
+\sum_{i_1,i_2,i_3=1}^n \sum_{\substack{1\leq p\leq n \\ p \neq i_1}} \Big(\prod_{\substack{1\leq j\leq n \\ j\neq i_2}}\Gc_{j}(\Ga)\Big) \Big(\prod_{\substack{1\leq j\leq n \\ j\neq i_3}} \Hc_j(\Ga)\Big)\Big(\prod_{\substack{1\leq j\leq n \\ j\neq i_1,p}} \Fc_j(\Ga)\Big) \Ec_{i_3i_1}^{\Hc\Fc}(\Ga)\Ec_{i_2p}^{\Gc\Fc}(\Ga) \\
+\sum_{i_1,i_2,i_3=1}^n \Big(\prod_{\substack{1\leq j\leq n \\ j\neq i_2}}\Gc_{j}(\Ga)\Big)\Big(\prod_{\substack{1\leq j\leq n \\ j\neq i_3}} \Hc_j(\Ga)\Big)\Big(\prod_{\substack{1\leq j\leq n \\ j\neq i_1}} \Fc_j(\Ga)\Big)\ipp*{\comm{\ds\Gc_{i_2}[0]}{\comm{\ds\Hc_{i_3}[0]}{\ds\Fc_{i_1}[0]}_{\G_\infty}}_{\G_\infty},\Ga}_{\G_\infty-\G_\infty^*},
\end{multline}
\begin{multline}
\ipp*{\comm{\ds\Hc[\Gamma]}{\ds\pb{\Fc}{\Gc}_{\G_\infty^{*}}[\Gamma]}_{\G_\infty},\Gamma}_{\G_\infty-\G_\infty^*}\\
=\sum_{i_1,i_2,i_3=1}^n\sum_{\substack{1\leq p\leq n \\ p \neq i_1}}  \Big(\prod_{\substack{1\leq j\leq n \\ j\neq i_3}}\Hc_{j}(\Ga)\Big) \Big(\prod_{\substack{1\leq j\leq n \\ j\neq i_1,p}} \Fc_j(\Ga)\Big)\Big(\prod_{\substack{1\leq j\leq n \\ j\neq i_2}} \Gc_j(\Ga)\Big) \Ec_{i_1i_2}^{\Fc\Gc}(\Ga)\Ec_{i_3p}^{\Hc\Fc}(\Ga)\\
+\sum_{i_1,i_2,i_3=1}^n \sum_{\substack{1\leq p\leq n \\ p \neq i_2}} \Big(\prod_{\substack{1\leq j\leq n \\ j\neq i_3}}\Hc_{j}(\Ga)\Big) \Big(\prod_{\substack{1\leq j\leq n \\ j\neq i_1}} \Fc_j(\Ga)\Big)\Big(\prod_{\substack{1\leq j\leq n \\ j\neq i_2,p}} \Gc_j(\Ga)\Big) \Ec_{i_1i_2}^{\Fc\Gc}(\Ga)\Ec_{i_3p}^{\Hc\Gc}(\Ga) \\
+\sum_{i_1,i_2,i_3=1}^n \Big(\prod_{\substack{1\leq j\leq n \\ j\neq i_3}}\Hc_{j}(\Ga)\Big)\Big(\prod_{\substack{1\leq j\leq n \\ j\neq i_1}} \Fc_j(\Ga)\Big)\Big(\prod_{\substack{1\leq j\leq n \\ j\neq i_2}} \Gc_j(\Ga)\Big)\ipp*{\comm{\ds\Hc_{i_3}[0]}{\comm{\ds\Fc_{i_1}[0]}{\ds\Gc_{i_2}[0]}_{\G_\infty}}_{\G_\infty},\Ga}_{\G_\infty-\G_\infty^*},
\end{multline}
where $\Ec_{i_1i_2}^{\Fc\Gc}, \Ec_{i_3p}^{\Hc\Fc},\Ec_{i_3p}^{\Hc\Gc}$ are defined analogously to above.

Since the Lie bracket $\comm{\cdot}{\cdot}_{\G_\infty}$ satisfies the Jacobi identity, we have
\begin{multline}
0 = \comm{\ds\Fc_{i_1}[0]}{\comm{\ds\Gc_{i_2}[0]}{\ds\Hc_{i_3}[0]}_{\G_\infty}}_{\G_\infty} + \comm{\ds\Gc_{i_2}[0]}{\comm{\ds\Hc_{i_3}[0]}{\ds\Fc_{i_1}[0]}_{\G_\infty}}_{\G_\infty}\\
 + \comm{\ds\Hc_{i_3}[0]}{\comm{\ds\Fc_{i_1}[0]}{\ds\Gc_{i_2}[0]}_{\G_\infty}}_{\G_\infty},
\end{multline}
implying
\begin{multline}
0 = \sum_{i_1,i_2,i_3=1}^n\Bigg(\Big(\prod_{\substack{1\leq j\leq n \\ j\neq i_1}}\Fc_{j}(\Ga)\Big)\Big(\prod_{\substack{1\leq j\leq n \\ j\neq i_2}} \Gc_j(\Ga)\Big)\Big(\prod_{\substack{1\leq j\leq n \\ j\neq i_3}} \Hc_j(\Ga)\Big)\\
\times\ipp*{\comm{\ds\Fc_{i_1}[0]}{\comm{\ds\Gc_{i_2}[0]}{\ds\Hc_{i_3}[0]}_{\G_\infty}}_{\G_\infty},\Ga}_{\G_\infty-\G_\infty^*}\\
+\Big(\prod_{\substack{1\leq j\leq n \\ j\neq i_2}}\Gc_{j}(\Ga)\Big)\Big(\prod_{\substack{1\leq j\leq n \\ j\neq i_3}} \Hc_j(\Ga)\Big)\Big(\prod_{\substack{1\leq j\leq n \\ j\neq i_1}} \Fc_j(\Ga)\Big)\ipp*{\comm{\ds\Gc_{i_2}[0]}{\comm{\ds\Hc_{i_3}[0]}{\ds\Fc_{i_1}[0]}_{\G_\infty}}_{\G_\infty},\Ga}_{\G_\infty-\G_\infty^*}\\
+\Big(\prod_{\substack{1\leq j\leq n \\ j\neq i_3}}\Hc_{j}(\Ga)\Big)\Big(\prod_{\substack{1\leq j\leq n \\ j\neq i_1}} \Fc_j(\Ga)\Big)\Big(\prod_{\substack{1\leq j\leq n \\ j\neq i_2}} \Gc_j(\Ga)\Big)\ipp*{\comm{\ds\Hc_{i_3}[0]}{\comm{\ds\Fc_{i_1}[0]}{\ds\Gc_{i_2}[0]}_{\G_\infty}}_{\G_\infty},\Ga}_{\G_\infty-\G_\infty^*}\Bigg).
\end{multline}
Since $\Ec_{i_1i_2}^{\Fc\Gc}(\Ga) = -\Ec_{i_2i_1}^{\Gc\Fc}(\Ga)$ by antisymmetry of the Lie bracket, it follows from swapping $i_1\leftrightarrow p$ that
\begin{multline}
0=\sum_{i_1,i_2,i_3=1}^n\sum_{\substack{1\leq p\leq n \\ p \neq i_1}}  \Big(\prod_{\substack{1\leq j\leq n \\ j\neq i_2}}\Gc_{j}(\Ga)\Big) \Big(\prod_{\substack{1\leq j\leq n \\ j\neq i_3}} \Hc_j(\Ga)\Big)\Big(\prod_{\substack{1\leq j\leq n \\ j\neq i_1,p}} \Fc_j(\Ga)\Big) \Ec_{i_3i_1}^{\Hc\Fc}(\Ga)\Ec_{i_2p}^{\Gc\Fc}(\Ga)\\
+\sum_{i_1,i_2,i_3=1}^n\sum_{\substack{1\leq p\leq n \\ p \neq i_1}}  \Big(\prod_{\substack{1\leq j\leq n \\ j\neq i_3}}\Hc_{j}(\Ga)\Big) \Big(\prod_{\substack{1\leq j\leq n \\ j\neq i_1,p}} \Fc_j(\Ga)\Big)\Big(\prod_{\substack{1\leq j\leq n \\ j\neq i_2}} \Gc_j(\Ga)\Big) \Ec_{i_1i_2}^{\Fc\Gc}(\Ga)\Ec_{i_3p}^{\Hc\Fc}(\Ga).
\end{multline}
Similarly, by swapping $i_2\leftrightarrow p$, we see that
\begin{multline}
0=\sum_{i_1,i_2,i_3=1}^n \sum_{\substack{1\leq p\leq n \\ p \neq i_2}}\Big(\prod_{\substack{1\leq j\leq n \\ j\neq i_1}}\Fc_{j}(\Ga)\Big) \Big(\prod_{\substack{1\leq j\leq n \\ j\neq i_2,p}} \Gc_j(\Ga)\Big)\Big(\prod_{\substack{1\leq j\leq n \\ j\neq i_3}} \Hc_j(\Ga)\Big) \Ec_{i_2i_3}^{\Gc\Hc}(\Ga)\Ec_{i_1p}^{\Fc\Gc}(\Ga) \\
+\sum_{i_1,i_2,i_3=1}^n \sum_{\substack{1\leq p\leq n \\ p \neq i_2}} \Big(\prod_{\substack{1\leq j\leq n \\ j\neq i_3}}\Hc_{j}(\Ga)\Big) \Big(\prod_{\substack{1\leq j\leq n \\ j\neq i_1}} \Fc_j(\Ga)\Big)\Big(\prod_{\substack{1\leq j\leq n \\ j\neq i_2,p}} \Gc_j(\Ga)\Big) \Ec_{i_1i_2}^{\Fc\Gc}(\Ga)\Ec_{i_3p}^{\Hc\Gc}(\Ga)
\end{multline}
and by swapping $i_3\leftrightarrow p$,
\begin{multline}
0=\sum_{i_1,i_2,i_3=1}^n\sum_{\substack{1\leq p\leq n \\ p \neq i_3}}  \Big(\prod_{\substack{1\leq j\leq n \\ j\neq i_2}}\Gc_{j}(\Ga)\Big) \Big(\prod_{\substack{1\leq j\leq n \\ j\neq i_3,p}} \Hc_j(\Ga)\Big)\Big(\prod_{\substack{1\leq j\leq n \\ j\neq i_1}} \Fc_j(\Ga)\Big) \Ec_{i_3i_1}^{\Hc\Fc}(\Ga)\Ec_{i_2p}^{\Gc\Hc}(\Ga)\\
+ \sum_{i_1,i_2,i_3=1}^n\sum_{\substack{1\leq p\leq n \\ p \neq i_3}}   \Big(\prod_{\substack{1\leq j\leq n \\ j\neq i_1}}\Fc_{j}(\Ga)\Big) \Big(\prod_{\substack{1\leq j\leq n \\ j\neq i_2}} \Gc_j(\Ga)\Big)\Big(\prod_{\substack{1\leq j\leq n \\ j\neq i_3,p}} \Hc_j(\Ga)\Big) \Ec_{i_2i_3}^{\Gc\Hc}(\Ga)\Ec_{i_1p}^{\Fc\Hc}(\Ga).
\end{multline}
After a little bookkeeping, we realize the Jacobi identity \eqref{eq:JIwts} has been shown.

Finally, we claim that $\pb{\cdot}{\cdot}_{\G_\infty^{*}}$ satisfies the Leibniz rule \eqref{eq:LR}. Since $\ds(\Fc\Gc)[\Gamma]=\Fc(\Gamma)\ds\Gc[\Gamma] + \Gc(\Gamma)\ds\Fc[\Gamma]$ by the Leibniz rule for the G\^ateaux derivative, we see that
\begin{align}
\pb{\Fc\Gc}{\Hc}_{\G_\infty^{*}}(\Gamma) &=\ipp*{\comm{\ds(\Fc\Gc)[\Gamma]}{\ds\Hc[\Gamma]}_{\G_\infty},\Gamma}_{\G_\infty-\G_\infty^*} \nonumber\\
&=\Fc(\Gamma)\ipp*{\comm{\ds\Gc[\Gamma]}{\ds\Hc[\Gamma]}_{\G_\infty},\Gamma}_{\G_\infty-\G_\infty^*} + \Gc(\Gamma)\ipp*{\comm{\ds\Fc[\Gamma]}{\ds\Hc[\Gamma]}_{\G_\infty},\Gamma}_{\G_\infty-\G_\infty^*} \nonumber\\
&=\Fc(\Gamma)\pb{\Gc}{\Hc}_{\G_\infty^{*}}(\Gamma) + \Gc(\Gamma)\pb{\Fc}{\Hc}_{\G_\infty^{*}}(\Gamma),
\end{align}
where the penultimate equality follows by bilinearity of the Lie bracket and duality pairing and the ultimate equality follows from the definition of the Poisson bracket.
\end{proof}

We next verify that $\A_{\infty}$ satisfies the non-degeneracy property \ref{assWP2}.

\begin{lemma}\label{lem:GinfWP2}
$\A_{\infty}$ satisfies property \ref{assWP2} in \cref{def:WPVS}.
\end{lemma}
\begin{proof}
Let $\Gamma \in\G_\infty^{*}$ and $v \in \G_\infty^{*}$. Suppose that $\ds\Fc[\Gamma](v) = 0$ for every $\Fc\in\A_\infty$. We will show that $v = 0$.

Consider functionals of the form $\Fc_{f,k_0}(\cdot ) \coloneqq \ipp{F_{k_0},\cdot}_{\G_\infty-\G_\infty^*}$, where
\begin{equation}
F_{k_0}^{(k)} \coloneqq
\begin{cases}
f^{(k_0)}, & {k=k_0} \\
0, & {\text{otherwise}}
\end{cases},
\end{equation}
for $k_0\in\N$ and $f^{(k_0)}\in \g_{k_0}$. $\Fc_{f,k_0}$ is an expectation, hence in $\A_\infty$. Since $\Fc_{f,k_0}$ is linear, we have $\ds\Fc_{f,k_0}[\Gamma](\cdot) = \Fc_{f,k_0}(\cdot)$, so if $v = (v^{(k)})_{k=1}^\infty \in \G_\infty^{*} $ is as above, we have by definition of $\Fc_{f,k_0}$ that
\begin{equation}
\Fc_{f,k_0}(v) = \ipp*{f^{(k_0)}, v^{(k_0)}}_{\g_k-\g_k^*} = 0.
\end{equation}
Since $f^{(k_0)}\in\g_{k_0}$ was arbitrary, it follows that $v^{(k_0)} =0$; and since $k_0\in\N$ was arbitrary, it follows that $v=0$. 
\end{proof}

We now turn to verifying property \ref{assWP3} concerning the Hamiltonian vector field. Unlike the $N$-particle situation, we do not \emph{a priori} know that $X_{\Gc}$ exists for an element $\Gc\in\A_\infty$, let alone have an explicit formula for $X_{\Gc}$. To show \ref{assWP3}, we will find a candidate vector field $X_{\Gc}$, for any $\Gc\in\Cc^\infty(\G_\infty^*)$, with the property that
\begin{equation}\label{eq:hamVFchar}
\forall \Fc\in\Cc^\infty(\G_\infty^*), \ \Ga\in\G_\infty^*, \qquad \pb{\Fc}{\Gc}_{\G_\infty^*}(\Ga) = \ds\Fc[\Ga](X_{\Gc}(\Ga)).
\end{equation}
Unfortunately, $\Gc\in \Cc^\infty(\G_\infty^*)$ is not enough information for us to prove that $X_{\Gc}$ is $\Cc^\infty$; however, we are able to show by direct computation that if $\Gc\in\A_\infty$, then it is $\Cc^\infty$. As previously commented, this issue is a primary reason for the introduction of the algebra $\A_\infty$.

Just as with the $N$-particle case, having an explicit formula for $X_{\Gc}$ is advantageous (cf. \cite[Lemma 6.15]{MNPRS2020} for the quantum case). Indeed, we will use such a formula to show in \cref{ssec:HamVH} that the Vlasov hierarchy can be interpreted as a Hamiltonian equation of motion on the weak Poisson vector space $(\G_\infty^*, \A_\infty, \pb{\cdot}{\cdot}_{\G_\infty^*})$. 

\begin{proposition}\label{lem:infhamVF}
If $\Gc\in \Cc^\infty(\G_\infty^*)$, then there exists a unique vector field $X_{\Gc}:\G_\infty^*\rightarrow\G_\infty^*$ satisfying \eqref{eq:hamVFchar}, which is given as follows: for $\ell\in\N$ and any $\Ga=(\ga^{(k)})_{k=1}^\infty\in \G_\infty^*$,
\begin{equation}\label{eq:infhamVF}
X_{\Gc}(\Ga)^{(\ell)} = \sum_{j=1}^\infty j\int_{(\R^{2d})^{j-1}}d\uz_{\ell+1;\ell+j-1}\pb{\sum_{\al=1}^\ell \ds\Gc[\Ga]_{(\al,\ell+1,\ldots,\ell+j-1)}^{(j)}}{\gamma^{(\ell+j-1)}}_{(\R^{2d})^{\ell+j-1}}.
\end{equation}
If $\Gc\in\A_{\infty}$, then $X_{\Gc}$ as defined in \eqref{eq:infhamVF} belongs to $\Cc^\infty(\G_\infty^*,\G_\infty^*)$. 
\end{proposition}
\begin{proof}
One can prove the proposition by repeating the argument for the analogous $N$-particle result \cref{prop:NhamVF}, focusing only on the case $r=1$. Instead, we show how to obtain the result as an $N\rightarrow\infty$ limiting consequence of \cref{prop:NhamVF}.

Recall that $r_0(N) = \max\{1,\ell+j-N\}$. It is evident that $r_0(N)\xrightarrow[]{N\rightarrow\infty} 1$ uniformly over fixed finite subsets of $(\ell,j)\in\{1,\ldots,N\}^2$. On the other hand, from \cref{rem:Nconlim} we have that $C_{\ell j Nr}\rightarrow\indic_{r=1}$ as $N\rightarrow\infty$. So at least formally, we expect from letting $N\rightarrow\infty$ in the $N$-particle Hamiltonian vector field formula \eqref{eq:NhamVF} that the Hamiltonian vector field associated to the functional $\Gc \in \Cc^\infty(\G_\infty^*)$ with respect to the Poisson bracket $\pb{\cdot}{\cdot}_{\G_\infty^*}$ is given by the right-hand side of \eqref{eq:infhamVF}. We then need to check that
\begin{equation}\label{eq:infhamWTS}
\forall \Fc\in\Cc^\infty(\G_\infty^*), \ \Ga\in\G_\infty^*, \qquad \ds\Fc[\Ga](X_{\Gc}(\Ga))= \pb{\Fc}{\Gc}_{\G_\infty^*}(\Ga).
\end{equation}

Fix $\Ga_0=(\ga_0^{(k)})_{k=1}^\infty\in\G_\infty^*$, and let $\Fc_{\Ga_0}, \Gc_{\Ga_0}$ be the expectations generated by $\ds\Fc[\Ga_0], \ds\Gc[\Ga_0]$, respectively. Let $M_0$ denote the maximal component index $k$ such that $\ds\Fc[\Ga_0]^{(k)},\ds\Gc[\Ga_0]^{(k)}$ are nonzero. 

By definition of the $\G_\infty^*$ Poisson bracket, for any $\Ga\in\G_\infty^*$,
\begin{equation}
\pb{\Fc}{\Gc}_{\G_\infty^*}(\Ga) = \ipp{\comm{\ds\Fc[\Ga]}{\ds\Gc[\Ga]}_{\G_\infty},\Ga}_{\G_\infty-\G_\infty^*} = \sum_{k=1}^{\infty} \ipp{\comm{\ds\Fc[\Ga]}{\ds\Gc[\Ga]}_{\G_\infty}^{(k)}, \ga^{(k)}}_{\g_k-\g_k^*}.
\end{equation}
Since $\ds\Fc[\Ga_0]^{(k)} = \ds\Gc[\Ga_0]^{(k)} = 0$ for $k>M_0$, we see from \eqref{eq:GinfLB} that $\comm{\ds\Fc[\Ga_0]}{\ds\Gc[\Ga_0]}_{\G_\infty}^{(k)} = 0$ for $k> 2M_0$. Indeed, $k=\ell+j-1>2M_0$ implies $\min(\ell,j)>M_0$, therefore
\begin{equation}
\Sym_k\paren*{\ds\Fc[\Ga_0]^{(\ell)} \wedge_1 \ds\Gc[\Ga_0]^{(j)}} = 0.
\end{equation}
By projection onto the first $N$ components, $\ds\Fc[\Ga_0],\ds\Gc[\Ga_0] \in \G_{N}$ for any $N>M_0$; and by examination of \eqref{eq:GNLBform}, we also have $\comm{\ds\Fc[\Ga_0]}{\ds\Gc[\Ga_0]}_{\G_N}^{(k)}=0$ for any $2M_0< k\leq N$. Furthermore, $\Fc_{\Ga_0},\Gc_{\Ga_0}$ are linear functionals on $\G_N^*$ for any $N>M_0$. By separate continuity of the distributional pairing and \cref{prop:GinfLA},
\begin{align}
\ipp{\comm{\ds\Fc[\Ga_0]}{\ds\Gc[\Ga_0]}_{\G_\infty}^{(k)}, \ga_0^{(k)}}_{\g_k-\g_k^*} &= \lim_{N\rightarrow\infty} \ipp{\comm{\ds\Fc[\Ga_0]}{\ds\Gc[\Ga_0]}_{\G_N}^{(k)}, \ga_0^{(k)}}_{\g_k-\g_k^*} \nn\\
&=\lim_{N\rightarrow\infty}\ipp{\comm{\ds\Fc_{\Ga_0}[\Ga_0]}{\ds\Gc_{\Ga_0}[\Ga_0]}_{\G_N}^{(k)}, \ga_0^{(k)}}_{\g_k-\g_k^*} .
\end{align}
Now introducing the notation $\Ga_{0,M}$ to denote the projection of $\Ga_{0}$ onto the first $M$ components, we have for $N\geq 2M_0+1$,
\begin{align}
\sum_{k=1}^{2M_0+1} \ipp*{\comm{\ds\Fc_{\Ga_0}[\Ga_0]}{\ds\Gc_{\Ga_0}[\Ga_0]}_{\G_N}^{(k)}, \ga_0^{(k)}}_{\g_k-\g_k^*}  &=\sum_{k=1}^{N} \ipp*{\comm{\ds\Fc_{\Ga_0}[\Ga_0]}{\ds\Gc_{\Ga_0}[\Ga_0]}_{\G_N}^{(k)}, \ga_0^{(k)}}_{\g_k-\g_k^*} \nn\\
&=\pb{\Fc_{\Ga_0}}{\Gc_{\Ga_0}}_{\G_N^*}(\Ga_{0,N}) \nn\\
&=\ds\Fc_{\Ga_0}[\Ga_0]\paren*{X_{\Gc_{\Ga_0},N}(\Ga_{0,N})} \nn\\
&=\sum_{\ell=1}^{M_0} \ipp*{\ds\Fc_{\Ga_0}[\Ga_0]^{(\ell)},X_{\Gc_{\Ga_0,N}}(\Ga_{0,N})^{(\ell)}}_{\g_\ell-\g_\ell^*},
\end{align}
where the penultimate line follows from \cref{prop:NhamVF} and using that $\ds\Fc_{\Ga_0}[\Ga_0]^{(\ell)}=0$ for $\ell>M_0$. Here, the subscript $N$ in $X_{\Gc_{\Ga_0},N}$ signifies that the Hamiltonian vector field is computed with respect to the bracket $\pb{\cdot}{\cdot}_{\G_N^*}$. By definition, $\ds\Fc_{\Ga_0}[\Ga_0] = \ds\Fc[\Ga_0]$ and from \cref{prop:NhamVF} again, using that $\ds\Gc[\Ga_0]^{(j)} = 0$ for $j>M_0$, we see that for $1\leq \ell\leq M_0$, 
\begin{multline}
X_{\Gc_{\Ga_0},N}(\Ga_{0,N})^{(\ell)}  \\
 = \sum_{j=1}^{M_0} \sum_{r=1}^{\min\{\ell,j\}}C_{\ell j Nr}{j\choose r}\int_{(\R^{2d})^{j-1}}d\uz_{\ell+1;\ell+j-1}\pb{\sum_{\ab_r\in P_r^\ell} \ds\Gc[\Ga_0]_{(\ab_r,\ell+1,\ldots,\ell+j-r)}^{(j)}}{\gamma^{(\ell+j-1)}}_{(\R^{2d})^{\ell+j-1}},
\end{multline}
provided $N$ is sufficiently large. By our previous remarks, the preceding right-hand side converges in $\g_\ell^*$ as $N\rightarrow\infty$ to $X_{\Gc}(\Ga_0)^{(\ell)}$ as defined by \eqref{eq:infhamVF} uniformly over finite subsets of indices $\ell$. After a little bookkeeping, we realize we have shown that
\begin{align}
\pb{\Fc}{\Gc}_{\G_\infty^*}(\Ga_0) &= \sum_{\ell=1}^{M_0} \lim_{N\rightarrow\infty}\ipp*{\ds\Fc[\Ga_0]^{(\ell)},X_{\Gc_{\Ga_0,N}}(\Ga_{0,N})^{(\ell)}}_{\g_\ell-\g_\ell^*} \nn\\
&=\sum_{\ell=1}^{M_0} \ipp*{\ds\Fc[\Ga_0]^{(\ell)}, X_{\Gc}(\Ga_0)^{(\ell)}}_{\g_\ell-\g_\ell^*},
\end{align}
where to obtain the penultimate line we use the separate continuity of the distributional pairing. Comparing this expression with \eqref{eq:infhamWTS}, we are done.

We now verify that $X_{\Gc}\in \Cc^\infty(\G_\infty^*,\G_\infty^*)$, assuming $\Gc\in\A_\infty$. By the observation \eqref{eq:Ainfform} for the structure of elements in $\A_\infty$ and using linearity, we can reduce to the case where $\Gc=\Gc_1\cdots\Gc_n$ is a finite product of expectations. 

By the Leibniz rule for the operator $\ds$,
\begin{equation}
\forall\Ga\in\G_\infty^*, \qquad \ds\Gc[\Ga] = \sum_{i=1}^n \Big(\prod_{\substack{1\leq q\leq n \\ q\neq i}}\Gc_q(\Ga)\Big)\ds\Gc_i[0].
\end{equation}
Since the $\Gc_q(\Ga)$ are just real numbers, we can use the bilinearity of the Poisson bracket $\pb{\cdot}{\cdot}_{(\R^{2d})^{\ell+j-1}}$ to write
\begin{multline}\label{eq:Gc12pb}
\pb{\sum_{\al=1}^\ell \ds\Gc[\Ga]_{(\al,\ell+1,\ldots,\ell+j-r)}^{(j)}}{\gamma^{(\ell+j-1)}}_{(\R^{2d})^{\ell+j-1}} \\
=\sum_{i=1}^N \Big(\prod_{\substack{1\leq q\leq n \\ q\neq i}}\Gc_q(\Ga)\Big)\pb{\sum_{\al=1}^\ell \ds\Gc_i[0]_{(\al,\ell+1,\ldots,\ell+j-r)}^{(j)}}{\gamma^{(\ell+j-1)}}_{(\R^{2d})^{\ell+j-1}}.
\end{multline}
Being a linear combination of products of derivatives, the expression corresponding to the Poisson bracket in the second line defines a map in the variable $\ga^{(\ell+j-1)}$ which belongs to $\Cc^\infty(\g_{\ell+j-1}^*, \g_{\ell+j-1}^*)$. Since the map $\G_\infty^*\rightarrow\g_{\ell+j-1}^*, \ \Ga\mapsto \ga^{(\ell+j-1)}$ is also $\Cc^\infty$ and $\Gc_1,\ldots,\Gc_n$ are $\Cc^\infty$ real-valued maps, it follows that the second line of \eqref{eq:Gc12pb} is in $\Cc^{\infty}(\G_\infty^*, \g_{\ell+j-1}^*)$. Now,
\begin{multline}
\int_{\Rd^{j-1}}d\uz_{\ell+1;\ell+j-1}\pb{\sum_{\al=1}^\ell \ds\Gc[\Ga]_{(\al,\ell+1,\ldots,\ell+j-r)}^{(j)}}{\gamma^{(\ell+j-1)}}_{(\R^{2d})^{\ell+j-1}}  \\
= \sum_{i=1}^n \Big(\prod_{\substack{1\leq q\leq n \\ q\neq i}}\Gc_q(\Ga)\Big)\int_{\Rd^{j-1}}d\uz_{\ell+1;\ell+j-1}\pb{\sum_{\al=1}^\ell \ds\Gc_i[0]_{(\al,\ell+1,\ldots,\ell+j-r)}^{(j)}}{\gamma^{(\ell+j-1)}}_{(\R^{2d})^{\ell+j-1}}.
\end{multline}
By \cref{prop:eptrdual}, we see that each of the summands in the right-hand side is in $\Cc^\infty(\G_\infty^*,\g_\ell^*)$, hence their sum is as well. Finally, multiplying by $j$ and summing over $1\leq j\leq M_0$, where $M_0$ is the maximum of all $k$ such that $\ds\Gc_1[0]^{(k)},\ds\Gc_2[0]^{(k)}$ are nonzero, we also obtain a map in $\Cc^\infty(\G_\infty^*,\g_{\ell}^*)$. Since $\ell$ was arbitrary, we conclude that $X_{\Gc}\in \Cc^\infty(\G_\infty^*,\G_\infty^*)$.
\end{proof}

\subsection{The Poisson morphism $\iota:\g_1^*\rightarrow \G_\infty^*$}\label{ssec:infgeomPM}
We close \cref{sec:infgeom} by proving \cref{thm:PM}, which asserts that the trivial embedding $\iota:\g_1^*\rightarrow\G_\infty^*$ introduced in \eqref{eq:iotadef} is a morphism of Poisson vector spaces in the sense of \cref{def:PM}.

To prove \cref{thm:PM}, we first need a formula for the the G\^ateaux derivative of $\iota$. The proof of the following lemma is a simple application of the product rule, which we leave for the reader to check.

\begin{lemma}\label{lem:diota}
The map $\iota\in\Cc^\infty(\g_1^*,\G_\infty^*)$. Moreover, for any $n\in\N$, $\mu,\nu_1,\ldots,\nu_n \in\g_1^*$, we have
\begin{equation}
\ds^n\iota[\mu]^{(k)}(\nu_1,\ldots,\nu_n) = \begin{cases} 0 , & {n>k} \\ \sum_{(i_1,\ldots,i_n)\in P_n^k}\mathcal{I}_{(i_1,\ldots,i_n)} , & {k\leq n}.\end{cases}
\end{equation}
where
\begin{equation}
\mathcal{I}_{(i_1,\cdots,i_n)} \coloneqq \bigotimes_{i=1}^n \zeta_i, \qquad \text{with} \ \zeta_i = \begin{cases}\mu, & {i\notin\{i_1,\ldots,i_n\}} \\ \nu_k, & {i=i_k}. \end{cases}
\end{equation}
\end{lemma}

\begin{proof}[Proof of \cref{thm:PM}]
Since $\iota$ is a $\Cc^\infty$ map and composition of $\Cc^\infty$ maps is again $\Cc^\infty$, we have that $\iota^*\Cc^\infty(\G_\infty^*) \subset \Cc^\infty(\g_1^*)$, implying \emph{a fortiori} that $\iota^*\A_\infty \subset \Cc^\infty(\g_1^*)$. To verify that $\iota$ is a morphism of Poisson vector spaces, we need to check that
\begin{equation}
\iota^*\pb{\cdot}{\cdot}_{\G_\infty^*} = \pb{\iota^*\cdot}{\iota^*\cdot}_{\g_1^*}.
\end{equation}

To this end, let $\Fc_\infty,\Gc_\infty\in \Cc^\infty(\G_\infty^*)$, and set $\Fc \coloneqq \Fc_\infty\circ\iota, \Gc\coloneqq \Gc_\infty\circ\iota$. For any $\mu\in\g_1^*$, we compute
\begin{align}
\pb{\Fc_\infty}{\Gc_\infty}_{\G_\infty^*}(\iota(\mu)) &= \ds\Fc_\infty[\iota(\mu)]\Big(X_{\Gc_\infty}(\iota(\mu))\Big) \nn\\
&= \sum_{\ell=1}^\infty \int_{\Rd^\ell}d\uz_\ell\ds\Fc_\infty[\iota(\mu)]^{(\ell)}(\uz_\ell)\Bigg(\sum_{j=1}^\infty j \int_{\Rd^{j-1}}d\uz_{\ell+1;\ell+j-1}\nn\\
&\qquad \pb{\sum_{\alpha=1}^\ell\ds\Gc_\infty[\iota(\mu)]_{(\alpha,\ell+1,\ldots,\ell+j-1)}^{(j)}}{\mu^{\otimes \ell+j-1}}_{\Rd^{\ell+j-1}}(\uz_{\ell+j-1})\Bigg).\label{eq:pbrhsub}
\end{align}
Above, we have implicitly used the Hamiltonian vector field formula \eqref{eq:infhamVF} on $X_{\Gc_\infty}$. Let us analyze the inner integral, which after unpacking the Poisson bracket and using the linearity of the marginal, equals
\begin{multline}
\sum_{\be=1}^{\ell+j-1}\sum_{\al=1}^{\ell}\Bigg(\int_{\Rd^{j-1}}d\uz_{\ell+1;\ell+j-1}\nabla_{x_\beta}\ds\Gc_\infty[\iota(\mu)]_{(\alpha,\ell+1,\ldots,\ell+j-1)}^{(j)}(\uz_{\ell+j-1})\cdot\nabla_{v_\beta}\mu^{\otimes \ell+j-1}(\uz_{\ell+j-1}) \\
-\int_{\Rd^{j-1}}d\uz_{\ell+1;\ell+j-1}\nabla_{v_\beta}\ds\Gc_\infty[\iota(\mu)]_{(\alpha,\ell+1,\ldots,\ell+j-1)}^{(j)}(\uz_{\ell+j-1})\cdot\nabla_{x_\beta}\mu^{\otimes \ell+j-1}(\uz_{\ell+j-1})\Bigg).
\end{multline}
Observe that for every $1\leq \alpha\leq \ell$ and $1\leq \be\leq \ell+j-1$,
\begin{multline}
\int_{\Rd^{j-1}}d\uz_{\ell+1;\ell+j-1}\nabla_{x_\beta}\ds\Gc_\infty[\iota(\mu)]_{(\alpha,\ell+1,\ldots,\ell+j-1)}^{(j)}(\uz_{\ell+j-1})\cdot\nabla_{v_\beta}\mu^{\otimes \ell+j-1}(\uz_{\ell+j-1}) \\
=\mu^{\otimes \alpha-1} \otimes \Big(\nabla_x\phi_{\Gc,j} \cdot \nabla_v\mu\Big)\otimes \mu^{\otimes \ell-\alpha},
\end{multline}
if $\beta =\alpha$ and zero otherwise, where $\phi_{\Gc,j}$ is the unique test function in $\g_1$ with the property that
\begin{equation}\label{eq:phiGjdef}
\forall \nu\in\g_1^*, \qquad \ipp*{\phi_{\Gc,j},\nu}_{\g_1-\g_1^*} = \ipp*{\ds\Gc_\infty[\iota(\mu)]^{(j)}, \nu\otimes \mu^{\otimes j-1}}_{\Cc^\infty(\Rd^j)-\Ec'(\Rd^j)}.
\end{equation}
We note that since $\ds\Gc_{\infty}[\iota(\mu)]^{(j)} = 0$ for all but finitely many $j$, we have $\phi_{\Gc,j}=0$ for all but finitely many $j$. Similarly, we have that
\begin{multline}
\int_{\Rd^{j-1}}d\uz_{\ell+1;\ell+j-1}\nabla_{v_\beta}\ds\Gc_\infty[\iota(\mu)]_{(\alpha,\ell+1,\ldots,\ell+j-1)}^{(j)}(\uz_{\ell+j-1})\cdot\nabla_{x_\beta}\mu^{\otimes \ell+j-1}(\uz_{\ell+j-1}) \\
=\mu^{\otimes \alpha-1} \otimes \Big(\nabla_v\phi_{\Gc,j} \cdot \nabla_x\mu\Big)\otimes \mu^{\otimes \ell-\alpha}.
\end{multline}
Therefore, recalling \cref{lem:diota} specialized to $n=1$,
\begin{multline}
\sum_{j=1}^\infty j \int_{\Rd^{j-1}}d\uz_{\ell+1;\ell+j-1} \pb{\sum_{\alpha=1}^\ell\ds\Gc_\infty[\iota(\mu)]_{(\alpha,\ell+1,\ldots,\ell+j-1)}^{(j)}}{\mu^{\otimes \ell+j-1}}_{\Rd^{\ell+j-1}}\\
=\ds\iota[\mu]^{(\ell)}\paren*{\sum_{j=1}^\infty j (\nabla_x\phi_{\Gc,j} \cdot \nabla_v\mu-\nabla_v\phi_{\Gc,j} \cdot \nabla_x\mu)}.
\end{multline}
Substituting this identity into the right-hand side of \eqref{eq:pbrhsub}, we arrive at
\begin{equation}\label{eq:crulesub}
\pb{\Fc_\infty}{\Gc_\infty}_{\G_\infty^*}(\iota(\mu)) =\sum_{\ell=1}^\infty \ipp*{\ds\Fc_\infty[\iota(\mu)]^{(\ell)}, \ds\iota[\mu]^{(\ell)}\Big(\sum_{j=1}^\infty j (\nabla_x\phi_{\Gc,j} \cdot \nabla_v\mu-\nabla_v\phi_{\Gc,j} \cdot \nabla_x\mu)\Big)}_{\g_\ell-\g_\ell^*}.
\end{equation}
By the chain rule and the definition of the functional $\Fc$,
\begin{equation}
\forall \mu,\nu\in\g_1^*, \qquad \ds\Fc[\mu](\nu) = \ds\Fc_\infty[\iota(\mu)]\Big(\ds\iota[\mu](\nu)\Big).
\end{equation}
Applying this identity to the right-hand side of \eqref{eq:crulesub} with  $\nu= \sum_{j=1}^\infty j (\nabla_x\phi_{\Gc,j} \cdot \nabla_v\mu-\nabla_v\phi_{\Gc,j} \cdot \nabla_x\mu)$, we obtain that
\begin{align}
\pb{\Fc_\infty}{\Gc_\infty}_{\G_\infty^*}(\iota(\mu)) &= \ds\Fc[\mu]\paren*{\sum_{j=1}^\infty j(\nabla_x\phi_{\Gc,j} \cdot \nabla_v\mu-\nabla_v\phi_{\Gc,j} \cdot \nabla_x\mu)} \nn\\
&= \ipp*{\ds\Fc[\mu], \sum_{j=1}^\infty j(\nabla_x\phi_{\Gc,j} \cdot \nabla_v\mu-\nabla_v\phi_{\Gc,j} \cdot \nabla_x\mu)}_{\g_1-\g_1^*} \nn\\
&=\ipp*{\nabla_x\ds\Fc[\mu]\cdot\nabla_v\sum_{j=1}^\infty j\phi_{\Gc,j}-\nabla_v\ds\Fc[\mu]\cdot\nabla_x\sum_{j=1}^\infty j\phi_{\Gc,j},\mu}_{\g_1-\g_1^*} \nn\\
&=\ipp*{\comm{\ds\Fc[\mu]}{\sum_{j=1}^\infty j\phi_{\Gc,j}}_{\g_1},\mu}_{\g_1-\g_1^*},\label{eq:pbret}
\end{align}
where the second line comes from identification of $\ds\Fc[\mu]$ as an element of $\g_1$, the third line comes from integration by parts, and the fourth line is by definition of the Lie bracket $\comm{\cdot}{\cdot}_{\g_1}$.

In order to conclude the proof, we need to analyze the functions $\phi_{\Gc,j}$. More precisely, returning to the definition \eqref{eq:phiGjdef}, we see from the $\Ss_j$-symmetry of $\ds\Gc_\infty[\iota(\mu)]^{(j)}$ that
\begin{align}
\forall \nu\in \g_1^*, \qquad j\ipp{\phi_{\Gc,j},\nu}_{\g_1-\g_1^*} &= \sum_{\alpha=1}^j\ipp*{\ds\Gc_\infty[\iota(\mu)]^{(j)},  \mu^{\otimes\alpha-1}\otimes \nu\otimes \mu^{\otimes j-\alpha}}_{\Cc^\infty(\Rd^j)-\Ec'(\Rd^j)} \nn\\
&=\ipp*{\ds\Gc_\infty[\iota(\mu)]^{(j)}, \ds\iota[\mu]^{(j)}(\nu)}_{\g_j-\g_j^*}.
\end{align}
Hence,
\begin{equation}
\ipp*{\sum_{j=1}^\infty j\phi_{\Gc,j}, \nu}_{\g_1-\g_1^*} = \sum_{j=1}^\infty \ipp*{\ds\Gc_\infty[\iota(\mu)]^{(j)}, \ds\iota[\mu]^{(j)}(\nu)}_{\g_j-\g_j^*} = \ds\Gc[\mu](\nu),
\end{equation}
where the ultimate equality follows from the chain rule and the definition of $\Gc$. Thus, $\sum_{j=1}^\infty j\phi_{\Gc,j}$ is the unique element of $\g_1$ identifiable with the G\^ateaux derivative $\ds\Gc[\mu]$. Returning to \eqref{eq:pbret}, we have shown that
\begin{equation}
\pb{\Fc_\infty}{\Gc_\infty}_{\G_\infty^*}(\iota(\mu))  =\ipp*{\comm{\ds\Fc[\mu]}{\ds\Gc[\mu]}_{\g_1},\mu}_{\g_1-\g_1^*} = \pb{\Fc}{\Gc}_{\g_1^*}(\mu),
\end{equation}
where the ultimate equality is tautological. This is precisely what we needed to show, and therefore the proof of \cref{thm:PM} is complete.
\end{proof}

\section{Hamiltonian Flows}\label{sec:Ham}
This last section of the article is devoted to the proofs of our Hamiltonian flows results, \cref{prop:Vlas} and \Cref{thm:bbgky,thm:Vlh}, announced in \cref{ssec:OutN}. These results respectively show that the Vlasov equation \eqref{eq:Vl}, BBGKY hierarchy \eqref{eq:BBGKY}, and Vlasov hierarchy \eqref{eq:VlH} each admits a Hamiltonian formulation. We mention again that while it has been known for some time that both the Vlasov equation and BBGKY hierarchy are Hamiltonian, the fact that the Vlasov hierarchy is also Hamiltonian appears to be a new observation.

\subsection{Vlasov}\label{ssec:HamVl}
We start with \cref{prop:Vlas} for the Vlasov equation, which one should view as putting the formal calculations of \cite{IT1976,Morrison1980,Gibbons1981,MW1982} on firm functional-analytic footing.

\begin{proof}[Proof of \cref{prop:Vlas}]
Recall the definition \eqref{eq:HVl} of $\Hc_{Vl}$. We compute the Hamiltonian vector field of $\Hvl$ with respect to the Poisson bracket $\pb{\cdot}{\cdot}_{\g_1^*}$, denoted by $X_{\Hvl}$, as follows. First, we compute the G\^ateaux derivative of $\Hvl$. Observe that for any $\ga,\d\ga\in \g_1^*$, it follows from the linearity of the kinetic energy and bilinearity of the potential energy that
\begin{align}
\lim_{\ep\rightarrow 0} \frac{\Hvl(\ga+\ep\d\ga)-\Hvl(\ga)}{\ep} &= \lim_{\ep\rightarrow 0}\Bigg(\frac{1}{2} \ipp*{|v|^2, \d\ga}_{\g_1-\g_1^*}+ 2\ipp*{W\ast \rho, \d\rho}_{\g_1-\g_1^*} \nn\\
&\qquad+ \ep\ipp{W\ast\d\rho,\d\rho}_{\g_1-\g_1^*}\Bigg) \nn\\
&=\frac{1}{2} \ipp*{|v|^2, \d\ga}_{\g_1-\g_1^*}+ 2\ipp*{W\ast \rho, \d\rho}_{\g_1-\g_1^*},
\end{align}
where above we have introduced the notation $\d\rho\coloneqq \int_{\R^d}d\d\ga(\cdot,v)$ for the density of $\d\ga$. Thus, we can identify the G\^ateaux derivative $\ds\Hvl[\ga]$ as the element of $\g_1$ given by
\begin{equation}\label{eq:Hvlid}
\ds\Hvl[\ga] = \frac{1}{2}|v|^2 + 2W\ast\rho,
\end{equation}
where the convolution $W\ast\rho$ is taken in the distributional sense.\footnote{Here, we are using the well-known fact that the convolution of an element of $\Cc^\infty(\R^{2d})$ with a distribution of compact support is again in $\Cc^\infty(\R^{2d})$.} For any functional $\Fc \in C^\infty(\g_1^*)$, we have by definition of the Poisson bracket $\pb{\cdot}{\cdot}_{\g_1^*}$ that
\begin{align}
\pb{\Fc}{\Hvl}_{\g_1^*}(\ga) = \ipp{\comm{\ds\Fc[\ga]}{\ds\Hvl[\ga]}_{\g_1},\ga}_{\g_1-\g_1^*}.
\end{align}
To compactify the notation, let us set $f\coloneqq \ds\Fc[\ga]$ and $h\coloneqq \ds\Hvl[\ga]$, the dependence on $\ga$ being implicit. Note that $h$ equals the right-hand side of \eqref{eq:Hvlid}. Unpacking the definition of the Lie bracket $\comm{f}{h}_{\g_1}$, we have
\begin{align}
\pb{\Fc}{\Hvl}_{\g_1^*}(\ga) &= \ipp*{(\nabla_x f\cdot \nabla_v h - \nabla_v f\cdot \nabla_x h),\ga}_{\g_1-\g_1^*} \nn\\
&=\ipp*{f,-(\nabla_v h\cdot\nabla_x\ga-\nabla_x h\cdot\nabla_v \ga)}_{\g_1-\g_1^*} \nn\\
&= \ds\Fc[\mu]\paren*{-\paren*{\nabla_v h\cdot\nabla_x\ga-\nabla_x h\cdot\nabla_v \ga}},
\end{align}
where the penultimate line follows from integration by parts (i.e., the definition of the distributional derivative) together with the fact that $\nabla_x\nabla_v h= \nabla_v\nabla_x h$ by the smoothness of $h$ and the ultimate line follows from the definition of $f$. Substituting in the right-hand side of \eqref{eq:Hvlid} for $h$, we arrive at the identity
\begin{equation}
\pb{\Fc}{\Hvl}_{\g_1^*}(\ga)= \ds\Fc[\ga]\paren*{- \paren*{v\cdot\nabla_x\ga -2(\nabla W\ast\rho)\cdot\nabla_v\ga}},
\end{equation}
which, by the characterizing property of the Hamiltonian vector field, implies the identity
\begin{equation}
X_{\Hvl}(\ga) =  - \paren*{v\cdot\nabla_x\ga -2(\nabla_xW\ast\rho)\cdot\nabla_v\ga}.
\end{equation}
Thus, the Vlasov equation \eqref{eq:Vl} is equivalently to the infinite-dimensional ODE
\begin{equation}
\dot{\ga} = X_{\Hvl}(\ga),
\end{equation}
as originally claimed.
\end{proof}

\begin{remark}\label{rem:EMVlHam}
We can use the Hamiltonian formulation to show that the empirical measure map $\iota_{EM}: \Rd^N\rightarrow \g_1^*$ introduced in \eqref{eq:iotaEMdef} sends solutions of the Newtonian system \eqref{eq:New} to (weak) solutions of the Vlasov equation. We compute
\begin{align}
(\iota_{EM}^*\Hvl)(\uz_N) &= \frac{1}{2}\int_{\Rd}d\Big(\iota_{EM}(\uz_N)\Big)(z)|v|^2  + \int_{\Rd^2}d\Big(\iota_{EM}(\uz_N)\Big)^{\otimes 2}(z,z')W(x-x') \nn\\
&= \frac{1}{2N}\sum_{i=1}^N |v_i|^2 + \frac{1}{N^2}\sum_{i,j=1}^N W(x_i-x_j) \nn\\
&=\Hc_{New}(\uz_N).
\end{align}
Since $\iota_{EM}$ is a Poisson morphism by \cref{prop:EMpm}, it follows that
\begin{align}
\forall \Fc\in\Cc^\infty(\g_1^*), \qquad \paren*{\iota_{EM}^*\pb{\Fc}{\Hvl}_{\g_1^*}}(\uz_N) &= \pb{\iota_{EM}^*\Fc}{\iota_{EM}^*\Hvl}_{N}(\uz_N) \nn\\
&= \pb{\iota_{EM}^*\Fc}{\Hc_{New}}_N(\uz_N).
\end{align}
Now if $\uz_N^t$ is a solution to \eqref{eq:New}, then
\begin{align}
\frac{d}{dt} \paren*{\iota_{EM}^*\Fc}(\uz_N^t) = \pb{\iota_{EM}^*\Fc}{\Hc_{New}}_{N}(\uz_N^t) = \pb{\Fc}{\Hvl}_{\g_1^*}(\iota_{EM}(\uz_N^t)).
\end{align}
Since $\Fc\in\Cc^\infty(\g_1^*)$ was arbitrary, the claim follows.
\end{remark}

\subsection{BBGKY hierarchy}\label{ssec:HamBBGKY}
We next turn to \cref{thm:bbgky} for the BBGKY hierarchy. As mentioned in \cref{ssec:OutN}, this result is the classical analogue of \cite[Theorem 2.3]{MNPRS2020} asserting that the quantum BBGKY hierarchy is Hamiltonian, which was not known prior to that work. Interestingly, the proof of the cited result was inspired by the formal computations of \cite{MMW1984} for the classical BBGKY hierarchy. It will not surprise the reader then to learn that the proof of \cref{thm:bbgky} here is algebraically similar to that of \cite[Theorem 2.3]{MNPRS2020} and the core is the calculation of the Hamiltonian vector field $X_{\mathcal{H}_{BBGKY}}$.

\begin{proof}[Proof of \cref{thm:bbgky}]
As $N$ will be fixed throughout the proof, we drop $N$ in our subscripts when there is no ambiguity. We recall from \cref{prop:NhamVF} that given any $\mathcal{G} \in \mathcal{C}^\infty(\G_N^*)$, the Hamiltonian vector field $X_{\Gc}$ is given by formula
\begin{multline}\label{eq:BBGKYGvf}
\forall 1\leq \ell \leq N, \qquad X_{\mathcal{G}}(\Ga)^{(\ell)}= \sum_{j=1}^N C_{\ell j N r }\sum_{r= r_0}^{\min(\ell,j)}\binom{j}{r}  \int_{(\R^{2d})^{k-\ell}}  d \underline{z}_{\ell+1;k}\\
\pb{\sum_{\ab_r \in P_r^\ell} \ds\mathcal{G} [\Ga]^{(j)}_{(\ab_r, \ell+1, \dots, \ell+j - r) }}{\ga^{(k)}}_{(\R^{2d})^k},
\end{multline}
where $C_{\ell j N r} = \frac{(N-\ell)! (N-j)!}{(N-1)! (N- \ell - j +r)!}$, $k = \min(\ell + j -1, N)$, and $r_0 = \max( 1, \ell + j -N)$. Note that the bracket in formula \eqref{eq:BBGKYGvf} is well-defined for $g^{(j)} \in \mathfrak{g}_j$ and $\ga^{(k)} \in \mathfrak{g}_k^*$, as explained in the paragraph after the statement of \cref{prop:NhamVF}.
Recall the definition \eqref{eq:BBGKYham} of $\Hc_{BBGKY}$. Now note that by the linearity of $\mathcal{H}_{BBGKY}$ we may identify
\begin{equation}
\ds\mathcal{H}_{BBGKY} [\Ga]= \W_{BBGKY} = \paren*{\frac{1}{2}|v|^2, \frac{(N-1)}{N}W(x_1-x_2) + \frac{W(0)}{N}, 0,\ldots,}.
\end{equation}
Consequently, $\ds\Hc_{BBGKY}[\Ga]$ is constant in $\Ga$ and $\ds\mathcal{H}_{BBGKY} [\Ga]^{(j)} = 0$ for $3 \leq j \leq N$. For $j=1$, we have 
\begin{equation} 
r_0 = \max(1, \ell + 1 -N ) = 1 \qquad \text{ and }\qquad  k = \min( \ell ,N ) = \ell,
\end{equation}
implying
\begin{equation}
\ds\Hc_{BBGKY}[\Ga]_{(\ab_r,\ell+1,\ldots,\ell+j-r)}^{(j)} = (\W_{BBGKY}^{(1)})_{(a)}, \qquad 1\leq a\leq \ell.
\end{equation}
For $j = 2$, we have 
\begin{equation}
    r_0 = \max( 1, \ell + 2 - N ) = \begin{cases}
    1, & \ell \leq N -1\\
    2,  & \ell = N, 
    \end{cases}\quad  \text{ and } \quad k = \min( \ell + 1, N ) = \begin{cases} 
    \ell + 1 , & \ell \leq N-1  \\
    N, & \ell = N,
    \end{cases}
\end{equation}
implying
\begin{equation}
\ds\Hc_{BBGKY}[\Ga]_{(\ab_r,\ell+1,\ldots,\ell+j-r)}^{(j)} = \begin{cases} (\W_{BBGKY}^{(2)})_{(a,\ell+1)}, & r=1 \\ (\W_{BBGKY}^{(2)})_{(\ab_2)}, & {r=2}.\end{cases}
\end{equation}
So our vector field reduces to 
\begin{multline}\label{eq:BBGKYpre}
X_{\mathcal{H}_{BBGKY}}(\Ga)^{(\ell)} =  C_{\ell 1 N 1 }\pb{\sum_{a=1}^\ell{(\bm{W}_{BBGKY}^{(1)})}_{(a)}}{\ga^{(\ell)}}_{(\R^{2d})^\ell} \\
+ C_{\ell 2 N r }\sum_{r= r_0}^{\min(\ell,2)}\binom{2}{r}  \int_{\Rd^{k-\ell}}  d\underline{z}_{\ell+1;k}\pb{\sum_{\ab_r\in P_r^\ell} {\bm{W}_{BBGKY}^{(2)}}_{(\ab_r, \ell+1)}}{\ga^{(k)}}_{(\R^{2d})^k},
\end{multline}
where $(\ab_r,\ell+1)$ should really be replaced by $(\ab_r)$ if $r=2$ and the integration is understood as vacuous if $k-\ell\leq 0$. The relevant cases are when $\ell =1$, when $2\leq \ell \leq N-1$, and when $\ell = N$.\footnote{The case $\ell =1$ is singled out to take care of $\min(\ell,2)$ in the second sum.} We proceed to consider each of these cases individually.
\begin{enumerate}
    \item The case $\ell =1$. The formula \eqref{eq:BBGKYpre} further simplifies to 
  \begin{align}
    X_{\mathcal{H}_{BBGKY}}(\Ga)^{(1)} & =  C_{1 1 N 1 } \pb{\bm{W}_{BBGKY}^{(1)}}{\ga^{(1)}}_{\R^{2d}}  \nonumber\\
    &\phantom{=} + C_{1 2 N 1 }\binom{2}{1}  \int_{\R^{2d}}  d z_2 \pb{{(\bm{W}_{BBGKY}}^{(2)}_{(1, 2)}}{\ga^{(2)}}_{\Rd^2} \nonumber\\
    & = \pb{\frac{1}{2} |v_1|^2}{\ga^{(1)}}_{\R^{2d}}  +2\frac{(N-1)}{N}\int_{\R^{2d}}  d z_2\pb{W(x_1-x_2)}{\ga^{(2)}}_{(\R^{2d})^2}.
\end{align}

\item The case $2 \leq \ell \leq N-1$. The formula  \eqref{eq:BBGKYpre} in this case becomes
\begin{align}
    X_{\mathcal{H}_{BBGKY}}(\Ga)^{(\ell)} & =  C_{\ell 1 N 1 } \binom{1}{1}  \pb{\sum_{a = 1}^\ell {\bm{W}_{BBGKY}^{(1)}}_{(a) }}{\ga^{(\ell)}}_{(\R^{2d})^\ell}   \nonumber\\
    &\phantom{=} + \sum_{r= 1}^{2}C_{\ell 2 N r }\binom{2}{r}  \int_{\R^{2d}}  d {z}_{\ell+1}\pb{\sum_{\ab_r \in P_r^\ell} {(\bm{W}_{BBGKY}^{(2)})}_{(\ab_r, \ell+1) }}{\ga^{(\ell+1)}}_{(\R^{2d})^{\ell+1}}  \nonumber\\
    & = C_{\ell 1 N 1 }   \pb{\sum_{a = 1}^\ell {(\bm{W}_{BBGKY}^{(1)})}_{(a) }}{\ga^{(\ell)}}_{(\R^{2d})^\ell} \nonumber  \\
    &\phantom{=} + 2C_{\ell 2 N 1 } \int_{\R^{2d}}  d {z}_{\ell+1}\pb{\sum_{a =1}^\ell {(\bm{W}_{BBGKY}^{(2)})}_{(a, \ell+1)}}{ \ga^{(\ell+1)} }_{(\R^{2d})^{\ell+1}}  \nonumber\\
    &\phantom{=} + C_{\ell 2 N 2 }  \int_{\R^{2d}}  d {z}_{\ell+1}\pb{\sum_{\ab_2 \in P_2^\ell} {(\bm{W}_{BBGKY}^{(2)})}_{(\ab_2)}}{ \ga^{(\ell+1)}}_{(\R^{2d})^{\ell+1}}.
\end{align}
We can calculate the constants explicitly as 
\begin{equation}
    C_{\ell 1 N 1} = \frac{ (N- \ell)!(N- 1)!}{(N-1)! ( N- \ell)!} = 1,
\end{equation}
\begin{equation}
  C_{\ell 2 N 1} = \frac{(N-\ell)! (N-2)!}{(N-1)! (N- \ell - 1)!} =   \frac{N-\ell}{N-1},
\end{equation}
\begin{equation}
 C_{\ell 2 N 2} = \frac{(N-\ell)! (N-2)!}{(N-1)! (N- \ell)!} = \frac{1}{N-1}.
\end{equation}
Moreover, by definition of $\W_{BBGKY}^{(2)}$ we have that 
\begin{align}
\sum_{a=1}^\ell (\bm{W}_{BBGKY}^{(2)})_{(a, \ell+1) } &= \frac{(N-1)}{N}\sum_{a=1}^\ell  W(x_{a} - x_{\ell+1}) + \frac{\ell W(0)}{N}, \\
\sum_{\ab_2 \in P_2^\ell} (\bm{W}_{BBGKY}^{(2)})_{(\ab_2) } &= \frac{(N-1)}{N}\sum_{1\leq i\neq j\leq \ell} W(x_i - x_j) + \frac{\ell(\ell-1)}{N}W(0).
\end{align}
Putting all of these simplifications together and using that the Poisson bracket with a constant is zero, we arrive at 
\begin{multline}
     X_{\mathcal{H}_{BBGKY}}(\Ga)^{(\ell)} = \pb{\sum_{a = 1}^\ell \frac{1}{2} |v_a|^2}{\ga^{(\ell)}}_{(\R^{2d})^\ell}   \\
      + \frac{2(N-\ell)}{N} \int_{\R^{2d}}  d {z}_{\ell+1}\pb{\sum_{a=1}^\ell  W(x_{a} - x_{\ell+1})}{\ga^{(\ell+1)}}_{\R^{2d(\ell+1)}}\\
    + \frac{1}{N}\int_{\R^{2d}}  d {z}_{\ell+1}\pb{\sum_{1\leq i\neq j\leq \ell} W(x_i - x_j)}{\ga^{(\ell+1)}}_{(\R^{2d})^{\ell+1}}.
\end{multline}
One can replace the sum $\sum_{1\leq i\neq j\leq \ell} W(x_i-x_j)$ with $\sum_{i,j=1}^\ell W(x_i-x_j)$ in the third term since $W$ is continuous at the origin and $\pb{W(0)}{\ga^{(\ell+1)}}_{\Rd^{\ell+1}}=0$.

\item The case $\ell = N$. In this case,  \eqref{eq:BBGKYpre} becomes
\begin{align}
      X_{\mathcal{H}_{BBGKY}}(\Ga)^{(N)} & =  C_{N 1 N 1 } \pb{\sum_{a=1}^N (\bm{W}_{BBGKY}^{(1)})_{(a)}}{\ga^{(N)}}_{(\R^{2d})^N} \nonumber\\
    &\phantom{=} +C_{N2N2} \pb{\sum_{\ab_2 \in P_2^N} (\bm{W}_{BBGKY}^{(2)})_{(\ab_2)}}{\ga^{(N)}}_{(\R^{2d})^N} \nonumber \\
    &= \pb{\sum_{a=1}^N\frac{1}{2} |v_a|^2 +\frac1N\sum_{i,j=1}^N W(x_i - x_j) } {\ga^{(N)}}_{(\R^{2d})^N},
\end{align}
where to obtain the ultimate line we have used the bilinearity of the Poisson bracket to combine both terms in the penultimate line. 
\end{enumerate}

Evaluating the Poisson brackets and comparing the resulting expressions with \eqref{eq:BBGKY} (remember that $\ga^{(N)}$ by convention satisfies the Liouville equation \eqref{eq:Lio}), we see that $\Ga^t = (\ga^{(\ell),t})_{\ell=1}^N$ is a solution to the BBGKY hierarchy if and only if $\dot{\Ga}^t = X_{\Hc_{BBGKY}}(\Ga^t)$. Hence, the proof is complete.
\end{proof}

\begin{remark}
Similar to \cref{rem:EMVlHam}, we can use the Hamiltonian formulation to show that the $N$-hierarchy of marginals of a solution to the Liouville equation \eqref{eq:Lio} is a solution to the BBGKY hierarchy \eqref{eq:BBGKY}. Indeed, since $\iota_{mar}$ is a Poisson morphism by \cref{prop:marPo},
\begin{equation}\label{eq:marPMid}
\forall \Fc\in \Cc^\infty(\G_N^*), \qquad \iota_{mar}^*\pb{\Fc}{\Hc_{BBGKY}}_{\G_N^*} = \pb{\iota_{mar}^*\Fc}{\iota_{mar}^*\Hc_{BBGKY}}_{\g_N^*}.
\end{equation}
By definition of $\iota_{mar}$ and $\Hc_{BBGKY}$,
\begin{align}
\forall \ga \in\g_N^*,\qquad (\iota_{mar}^*\Hc_{BBGKY})(\ga) &= \ipp*{\frac{1}{2}|v|^2,\ga^{(1)}}_{\g_1-\g_1^*} + \ipp*{\frac{(N-1)}{N}W(x_1-x_2) + \frac{W(0)}{N}, \ga^{(2)}}_{\g_2-\g_2^*} \nn\\
&= \ipp*{\Sym_{N}\paren*{\frac{1}{2}|v_1|^2 + \frac{(N-1)}{N}W(x_1-x_2)+\frac{W(0)}{N}},\ga}_{\g_N-\g_N^*}.\label{eq:BSymsub}
\end{align}
Given any distinct integers $1\leq j_1,\ldots,j_k\leq N$,
\begin{align}
|\{\sigma\in\Ss_N : (\sigma(1),\ldots,\sigma(k)) = (j_1,\ldots,j_k)\}| = (N-k)!.
\end{align}
Hence,
\begin{align}
\Sym_{N}\paren*{\frac{1}{2}|v_1|^2 + \frac{(N-1)}{N}W(x_1-x_2)+\frac{W(0)}{N}} &= \frac{1}{2N}\sum_{i=1}^N |v_i|^2 + \frac{1}{N^2}\sum_{1\leq i\neq j\leq N} W(x_i-x_j) + \frac{W(0)}{N}\nn\\
&=\Hc_{New},
\end{align}
which, upon substitution into \eqref{eq:BSymsub}, implies
\begin{align}
\iota_{mar}^*\Hc_{BBGKY}(\ga) = \ipp*{\Hc_{New}, \ga}_{\g_N-\g_N^*} = \Hc_{Lio}(\ga).
\end{align}
Combining this identity with \eqref{eq:marPMid} and using the fact (shown in \cref{ssec:NgeomLio}) that the Liouville equation is Hamiltonian, we see that if $\ga^t$ is a solution to the Liouville equation, then
\begin{align}
\forall \Fc\in\Cc^\infty(\G_N^*), \qquad \frac{d}{dt}\Fc(\iota_{mar}(\ga^t)) = \frac{d}{dt}(\iota_{mar}^*\Fc)(\ga^t) &= \pb{\iota_{mar}^*\Fc}{\Hc_{Lio}}_{\g_N^*}(\ga^t) \nn\\
&= \pb{\Fc}{\Hc_{BBGKY}}_{\G_N^*}(\iota_{mar}(\ga^t)).
\end{align}
Since $\Fc$ was arbitrary, we conclude that $\frac{d}{dt}\iota_{mar}(\ga^t) = X_{\Hc_{BBGKY}}(\iota_{mar}(\ga^t))$, that is $\iota_{mar}(\ga^t)$ is a solution of the BBGKY hierarchy as claimed.
\end{remark}

\subsection{Vlasov hierarchy}\label{ssec:HamVH}
We close out \cref{sec:Ham} with the proof of \cref{thm:Vlh} for the Vlasov hierarchy. As commented in \cref{ssec:Outinf}, \cref{thm:Vlh} is the classical analogue of \cite[Theorem 2.10]{MNPRS2020} demonstrating a Hamiltonian formulation for the Gross-Pitaevskii hierarchy,\footnote{This result is not just aesthetically pleasing: it was subsequently used in \cite{MNPRS2_2019} to investigate the origins of the 1D cubic NLS as an integrable classical field theory from an integrable quantum field theory.}  which again was a new observation. As for the $N$-particle level, the proof of \cref{thm:Vlh} proceeds algebraically similarly to that of \cite[Theorem 2.10]{MNPRS2020}; and as with the proof of \cref{thm:bbgky} carried out in the previous subsection, the main step is the computation of the Hamiltonian vector field $X_{\Hvlh}$.

\begin{proof}[Proof of \cref{thm:Vlh}]
Applying the formula \eqref{eq:infhamVF} of \cref{lem:infhamVF}, we have the identity
\begin{equation}\label{eq:XVHform}
X_{\Hvlh}(\Ga)^{(\ell)} = \sum_{j=1}^\infty j \int_{(\R^{2d})^{j-1}}d\uz_{\ell+1;\ell+j-1}\pb{\sum_{a=1}^\ell\ds\Hvlh[\Ga]_{(a,\ell+1,\ldots,\ell+j-1)}^{(j)}}{\ga^{(\ell+j-1)}}_{\Rd^{\ell+j-1}},
\end{equation}
so we are reduced to computing the bracket in the integrand. We remind the reader of the following notation conventions: if $j=1$, the integration is vacuous and $(a,\ell+1,\ldots,\ell+j-1)=a$; if $j=2$, then $(a,\ell+1,\dots,\ell+j-1)=(a,\ell+1)$.

Since $\Hvlh$ is linear, it is evident, upon recalling the definition \eqref{eq:VlHham}, that $\ds\Hvlh[\Ga]$ is identifiable with $\W_{VlH}$ through the pairing $\ipp{\cdot,\cdot}_{\G_\infty-\G_\infty^*}$. In particular, $\ds\Hvlh[\Ga]$ is constant in $\Ga$ and $\ds\Hvlh[\Ga]^{(j)} = 0$ for $j\geq 3$. For $j=1$, we have that
\begin{equation}
\ds\Hvlh[\Ga]_{(a,\ell+1,\ldots,\ell+j-1)}^{(j)} = (\W_{VlH}^{(1)})_{(a)}= \frac{1}{2}|v_a|^2,
\end{equation}
and for $j=2$, we have that
\begin{equation}
\ds\Hvlh[\Ga]_{(a,\ell+1,\ldots,\ell+j-1)}^{(j)} = (\W_{VlH}^{(2)})_{(a,\ell+1)}  =W(x_a-x_{\ell+1}).
\end{equation}
Now for each $1\leq a\leq \ell$, it follows that
\begin{align}
\pb{(\W_{VlH}^{(1)})_{(a)}}{\ga^{(\ell)}}_{\Rd^\ell} &= \sum_{\beta=1}^\ell \Big(\nabla_{x_\beta}(\W_{VlH}^{(1)})_{(a)}\cdot\nabla_{v_\beta}\ga^{(\ell)} - \nabla_{v_\beta}(\W_{VlH}^{(1)})_{(a)}\cdot\nabla_{x_\beta}\ga^{(\ell)}\Big) \nn\\
&=\nabla_{x_a}(\W_{VlH}^{(1)})_{(a)}\cdot\nabla_{v_a}\ga^{(\ell)} - \nabla_{v_a}(\W_{VlH}^{(1)})_{(a)}\cdot\nabla_{x_a}\ga^{(\ell)}\nn\\
&=-v_a\cdot\nabla_{x_a}\ga^{(\ell)}. \label{eq:VlH1p}
\end{align}
Similarly,
\begin{align}
\pb{(\W_{VlH}^{(2)})_{(a,\ell+1)}}{\ga^{(\ell+1)}}_{\Rd^{\ell+1}} &=\sum_{\beta=1}^\ell \Big(\nabla_{x_\beta}(\W_{VlH}^{(2)})_{(a,\ell+1)}\cdot\nabla_{v_\beta}\ga^{(\ell+1)} \nn\\
&\qquad- \nabla_{v_\beta}(\W_{VlH}^{(2)})_{(a,\ell+1)}\cdot\nabla_{x_\beta}\ga^{(\ell+1)}\Big) \nn\\
&=\nabla_{x_a}(\W_{VlH}^{(2)})_{(a,\ell+1)}\cdot\nabla_{v_a}\ga^{(\ell+1)} - \nabla_{v_a}(\W_{VlH}^{(2)})_{(a,\ell+1)}\cdot\nabla_{x_a}\ga^{(\ell+1)}\nn\\
&=\nabla W(x_a-x_{\ell+1})\cdot\nabla_{v_a}\ga^{(\ell+1)}.\label{eq:VlH2p}
\end{align}
Substituting the identities \eqref{eq:VlH1p}, \eqref{eq:VlH2p} into \eqref{eq:XVHform}, we arrive at
\begin{multline}
\sum_{j=1}^\infty j\int_{(\R^{2d})^{j-1}}d\uz_{\ell+1;\ell+j-1}\pb{\sum_{a=1}^\ell\ds\Hvlh[\Ga]_{(a,\ell+1,\ldots,\ell+j-1)}^{(j)}}{\ga^{(\ell+j-1)}}_{\Rd^{\ell+j-1}}\\
= \sum_{a=1}^\ell \Big(-v_a\cdot\nabla_{x_a}\ga^{(\ell)}+2\int_{\R^{2d}}dz_{\ell+1}\nabla W(x_a-x_{\ell+1})\cdot\nabla_{v_a}\ga^{(\ell+1)}\Big).
\end{multline}
Comparing this expression to \eqref{eq:VlH}, we see that $\Ga^t=(\ga^{(\ell),t})_{\ell=1}^\infty$ is a solution to the Vlasov hierarchy if and only if $\dot{\Ga}^t = X_{\Hvlh}(\Ga^t)$, hence the proof of the theorem is complete.
\end{proof}

\begin{remark}\label{rem:VlHVl}
We end this paper by using the Hamiltonian formulation to show that the factorization map $\iota:\g_1^*\rightarrow\G_\infty^*$, introduced in \eqref{eq:iotadef}, maps solutions of the Vlasov equation to solutions of the Vlasov hierarchy.

Observe that
\begin{align}
\forall\ga \in \g_1^*, \qquad (\iota^*\Hvlh)(\ga) &= \ipp*{\frac12|v|^2,\ga}_{\g_1-\g_1^*} + \ipp*{W(x_1-x_2), \ga^{\otimes 2}}_{\g_2-\g_2^*} \nn\\
&=\ipp*{\frac12|v|^2,\ga}_{\g_1-\g_1^*} + \ipp*{W\ast\rho,\rho}_{\g_1-\g_q^*} \nn\\
&= \Hvl(\ga). \label{eq:VlHampback}
\end{align}
In other words, the pullback of the Vlasov hierarchy Hamiltonian equals the Vlasov Hamiltonian, as originally announced in \cref{ssec:intromr}. Since $\iota$ is a Poisson morphism by \cref{thm:PM}, it follows from \eqref{eq:VlHampback} and the Hamiltonian formulation of the Vlasov equation proven in \cref{ssec:HamVl} that if $\ga^t$ is a solution to the Vlasov equation,
\begin{align}
\forall \Fc\in\Cc^\infty(\G_\infty^*), \qquad \frac{d}{dt}\Fc(\iota(\ga^t)) = \frac{d}{dt}(\iota^*\Fc)(\ga^t) &= \pb{\iota^*\Fc}{\Hvl}_{\g_1^*}(\ga^t) \nn\\
&=\pb{\Fc}{\Hvlh}_{\G_\infty^*}(\ga^t).
\end{align}
Since $\Fc$ was arbitrary, we conclude that $\frac{d}{dt}\iota(\ga^t) = X_{\Hvlh}(\iota(\ga^t))$, that is $\iota(\ga^t)$ is a solution of the Vlasov hierarchy.
\end{remark}

\newpage
\begin{table}[H]
   \small 
   \centering 
   \begin{tabular}{ll} 
   \textbf{Symbol} & \textbf{Definition} \\ 
      \hline
	$z, z_i$ & $(x,v)$, $(x_i,v_i)$\\
   $\ul{z}_{k}$ & $(z_1, \ldots, z_k)$ \\
   $\ul{z}_{m_1; m_k}$ & $(z_{m_1}, \ldots, z_{m_k})$ \\
   $\ul{z}_{i;i+k}$ & $(z_i, \ldots, z_{i+k})$ \\
      $d\ul{z}_{k}$ & $ dz_1 \cdots dz_k$ \\
         $d\ul{z}_{i;i+k}$ & $dz_i \cdots dz_{i+k}$ \\
      $\N$, $\N_0$ & natural numbers exclusive, inclusive of $0$\\   
      $\Ss_{k}$ & symmetric group on $k$ elements \\
      $P_{k}^N$ & set of $k$-tuples $(i_1,\ldots,i_k)$ drawn from $\{1,\ldots,N\}$, \eqref{eq:PkNdef}\\
      $\bm{j}_k$ & $k$-tuple $(j_1,\ldots,j_k)$\\
      $x^{\times k}$ & $k$-fold Cartesian product of $x$ with itself \\ 
             $\phi^{\otimes k}$ & $k$-fold tensor of $\phi$ with itself\\
      $\Cc^\infty(\R^{k}), \Ec'(\R^{k})$ & smooth functions on $\R^k$ and distributions on $\R^k$ with compact support \\
      $\ipp{\cdot,\cdot}$ & duality pairing \\
      $\ds\Fc$ & G\^ateaux derivative of $\Fc$, \eqref{eq:gddef} \\
      $X_{\Fc}$ & Hamiltonian vector field associated to $\Fc$, \eqref{eq:HamVFuniq} \\
       $f_{(j_1,\ldots, j_k)}^{(k)}$, $f_{\bm{j}_k}^{(k)}$ & $N$-particle extension of $k$-particle observable acting on $j_1,\ldots,j_k$ coordinates, \eqref{eq:fj1jkdef}\\
       $\Sym_k(f^{(k)})$ & $k$-particle symmetrization operator, \eqref{eq:Symdef}  \\
       $\pb{\cdot}{\cdot}_{\Rd^N}$, $\pb{\cdot}{\cdot}_{N}$ & standard Poisson bracket on $\Rd^N$, \eqref{eq:pbN}/rescaled standard Poisson bracket, \eqref{eq:NewHamdef}\\
	$\g_k, \g_k^*$ & spaces of $k$-particle observables/states, \eqref{eq:gkdef}/\eqref{eq:gk*def} \\
	$\G_N, \G_N^*$ & space of $N$-hierarchies of observables/states, \eqref{eq:GNdef}/\eqref{eq:GN*def} \\
	$\G_\infty, \G_\infty^*$ & space of $\infty$-hierarchies of observables/states, \eqref{eq:Ginfdef}/\eqref{eq:Ginf*def} \\
	$\comm{\cdot}{\cdot}_{\g_k}, \pb{\cdot}{\cdot}_{\g_k^*}$ & Lie bracket/Lie-Poisson bracket for $k$-particle observables/states, \eqref{eq:gkLBdef}/\eqref{eq:gk*LPdef} \\
	$\comm{\cdot}{\cdot}_{\G_N}, \pb{\cdot}{\cdot}_{\G_N^*}$ & Lie bracket/Lie-Poisson bracket for $N$-hierarchies of obervables/states, \eqref{eq:GNLBdef},\eqref{eq:GNLBform}/\eqref{eq:GN*LPdef}\\
	$\comm{\cdot}{\cdot}_{\G_\infty},\pb{\cdot}{\cdot}_{\G_\infty^*}$ & Lie bracket/Lie-Poisson bracket for $\infty$-hierarchies of observables/states, \eqref{eq:GinfLB}/\eqref{eq:Ginf*LPdef} \\
	$\A_\infty$ & Unital subalgebra of $\Cc^\infty(\G_\infty^*)$ generated by constants and expectations, \eqref{eq:Ainfdef} \\
	$\wedge_r$ & r-fold contraction, \eqref{r wedge}\\
	$\ep_{k,N}$ & embedding of $k$-particle observable in space of $N$-particle observables, \eqref{eq:epkNdef} \\
        $\int_{\Rd^{N-k}}d\uz_{k+1;N}$ & $k$-particle marginal, \eqref{eq:mardef} \\
        $\Hc_{New}$ & Newton Hamiltonian functional, \eqref{eq:NewHamdef} \\
        $\Hc_{Lio}$ & Liouville Hamiltonian functional, \eqref{eq:LioHam}\\
        $\Hc_{BBGKY}$, $\W_{BBGKY}$ & BBGKY Hamiltonian functional/generator, \eqref{eq:BBGKYham}/\eqref{eq:BBGKYW} \\
        $\Hc_{VlH}$, $\W_{VlH}$ & Vlasov hierarchy Hamiltonian functional/generator, \eqref{eq:VlHham}/\eqref{eq:VlHW} \\
        $\Hc_{Vl}$ & Vlasov Hamiltonian functional, \eqref{eq:HVl} \\
        $\iota_{EM}$ & empirical measure map, \eqref{eq:iotaEMdef}\\
        $\iota_{Lio}$ & Liouville map, \eqref{eq:iotaLiodef}\\
        $\iota_{\ep}$ & Lie algebra homomorphism induced by $\{\ep_{k,N}\}_{k=1}^N$, \cref{eq:iotaepNdef}\\
        $\iota_{mar}$ & marginals map, \eqref{homsum formula}\\
        $\iota$ & factorization map, \eqref{eq:iotadef}\\
   \end{tabular}
   \caption{Notation}
    \label{tab:not}
\end{table}

\newpage
\bibliographystyle{alpha}
\bibliography{GPHam}
\end{document}